\documentclass[12pt,a4paper,twoside,reqno]{amsart}
\usepackage{amssymb,latexsym}
\usepackage{amsmath,amsthm}
\usepackage{breqn}
\usepackage{stmaryrd} 

\usepackage[margin=2.5cm]{geometry}

 \usepackage[unicode,colorlinks,plainpages=false,hyperindex=true,bookmarksnumbered=true,bookmarksopen=false,pdfpagelabels]{hyperref}

 \hypersetup{urlcolor=cyan,linkcolor=blue,citecolor=red,colorlinks=true} 
\usepackage{xcolor}
\usepackage{tikz} 
\usetikzlibrary{arrows,shapes}  

\usepackage{xfrac}  
\usepackage{lscape}
\usepackage{colortbl,hhline}
\makeatletter
\renewcommand*{\p@section}{\,}
\renewcommand*{\p@subsection}{\S\,}
\renewcommand*{\p@subsubsection}{\S\,}
\makeatother
\usepackage{caption}

\newtheorem{thm}{Theorem}[section]
\newtheorem{cor}[thm]{Corollary}
\newtheorem{lem}[thm]{Lemma}
\newtheorem{prop}[thm]{Proposition}
\newtheorem{exmp}[thm]{Example}
\newtheorem{rem}[thm]{Remark}
\newtheorem{conj}[thm]{Conjecture}
\newtheorem{defn}[thm]{Definition}

\numberwithin{equation}{section}

\newcommand{\CC}{\ensuremath{\mathbb{C}}}

\newcommand{\Z}{\ensuremath{\mathbb{Z}}}
\newcommand{\kk}{\ensuremath{\Bbbk}}

      \newcommand{\out}{\operatorname{out}}

\newcommand{\tr}{\operatorname{tr}}

\newcommand{\Hom}{\operatorname{Hom}}

\newcommand{\Id}{\operatorname{Id}}
\newcommand{\Gl}{\operatorname{GL}}
\newcommand{\GL}{\operatorname{GL}}
\newcommand{\Rep}{\operatorname{Rep}}
\newcommand{\End}{\operatorname{End}}
\newcommand{\Spec}{\operatorname{Spec}}

\newcommand{\VdB}{\operatorname{VdB}}

\newcommand{\sgn}{\operatorname{sgn}}

\newcommand\dgal[1]{  \left\{\!\!\left\{#1\right\}\!\!\right\} }

\newcommand{\lr}[1]{
  \{\mkern-6mu\{#1\}\mkern-6mu\}}

\newcommand{\qP}{\ensuremath{\mathrm{qP}}}

\newcommand{\Ac}{\ensuremath{\mathcal{A}}}
\newcommand{\Bc}{\ensuremath{\mathcal{B}}}
\newcommand{\Fc}{\ensuremath{\mathcal{F}}}

\newcommand{\rc}{\ensuremath{\mathfrak{c}}}

\newcommand*\circled[1]{\tikz[baseline=(char.base)]{
            \node[shape=circle,draw,inner sep=2pt] (char) {#1};}}

\begin{document}

\title[NC Poisson geometry of certain wild character varieties]{On the noncommutative Poisson geometry of certain wild character varieties}

\author{Maxime Fairon}
 \address[Maxime Fairon]{ 
  Institut de Mathématiques de Bourgogne.
Université Bourgogne Europe, CNRS, IMB UMR 5584, F-21000 Dijon, France
}
 \email{maxime.fairon@u-bourgogne.fr}
 
\author{David Fern\'andez}

 \address[David Fern\'andez]{ 
Departamento de Matem\'atica e Informática Aplicadas a las Ingenier\'ias Civil y Naval, 
ETSI Caminos, Canales y Puertos. Universidad Politécnica de Madrid. 
Calle del Profesor Aranguren 3, 
28040 Madrid, Spain
}
\email{david.fernandezalv@upm.es}

\begin{abstract}  
To show that certain wild character varieties are multiplicative analogues of quiver varieties, 
Boalch introduced colored multiplicative quiver varieties. 
They form a class of (nondegenerate) Poisson varieties attached to colored quivers whose representation theory is controlled by fission algebras: noncommutative algebras generalizing the multiplicative preprojective algebras of Crawley-Boevey and Shaw. 
Previously, Van den Bergh exploited the Kontsevich--Rosenberg principle to prove that the natural Poisson structure of any  multiplicative quiver variety with tautological coloring is induced by an $H_0$-Poisson structure on the underlying multiplicative preprojective algebra; indeed, it turns out that this noncommutative structure comes from a Hamiltonian double quasi-Poisson algebra constructed from the quiver itself. 
In this article we conjecture that, via the Kontsevich--Rosenberg principle, the natural Poisson structure on each colored multiplicative quiver variety is induced by an $H_0$-Poisson structure on the underlying fission algebra which, in turn, is obtained from a Hamiltonian double quasi-Poisson algebra attached to the colored quiver. 
We study some consequences of this conjecture and we prove it in two significant cases: 
the interval and the triangle. \\


\end{abstract}

\addcontentsline{toc}{subsection}{Contents}  

\maketitle

 \setcounter{tocdepth}{2} 
 
\renewcommand{\baselinestretch}{0.85}\normalsize
\tableofcontents
\renewcommand{\baselinestretch}{1.0}\normalsize
 

 \section{Introduction}
 
 The Kontsevich--Rosenberg principle \cite{KR00} states that a noncommutative structure on an associative algebra $A$ has algebro-geometric meaning if it induces the corresponding standard algebro-geometric structures on its representation schemes.
A significant application of this principle has yielded a program to explain symplectic/Poisson structures on interesting moduli spaces in terms of canonical noncommutative structures on associative algebras. 

In \cite{VdB1,VdB2}, Van den Bergh defined double Poisson algebras and double quasi-Poisson algebras to develop noncommutative analogues of Poisson geometry and quasi-Poisson geometry \cite{AMM98,AKSM02}. 
Using the first of these newly introduced structures, he was able to translate the Poisson structure of quiver varieties directly at the level of the quivers. 
More importantly, he defined a \emph{Hamiltonian double quasi-Poisson algebra}\footnote{We follow the terminology from \cite{VdB2}. This is also called a \emph{quasi-Hamiltonian algebra} \cite{F19,F20,VdB1}.} 
structure on (a localization of) the path algebra of the quiver $\overline{A}_2$ given by 
two vertices and two arrows $a,a^*$. 
Then, in agreement with the Kontsevich--Rosenberg principle, the corresponding representation schemes are Hamiltonian quasi-Poisson manifolds \cite{AKSM02}. 
Furthermore, using fusion as well as quasi-Hamiltonian reduction, Van den Bergh obtained a natural Poisson structure on each multiplicative quiver variety \cite{CBShaw,Ya}. This follows from the fact that the canonical Hamiltonian double quasi-Poisson algebra on the attached quiver induces an $H_0$-Poisson structure \cite{CB11} on the corresponding multiplicative preprojective algebra \cite{CBShaw} ---\,see Table \ref{Table} for an overview and Theorem \ref{tm:VdB-mult-preproj-quasi-Hamilt} for the precise statements. 
In another situation, Massuyeau and Turaev \cite{MT14} proved that the well-known Poisson structure on $\Hom\big(\pi,\GL_N(\mathbb{R})\big)/\GL_N(\mathbb{R})$ (here, $\pi$ stands for the fundamental group of an oriented surface with boundary) is determined by a double quasi-Poisson bracket on the group algebra $\mathbb{R}\pi$.

In this article, we show a new instance of this far-reaching program by revealing that certain Poisson structures on wild character varieties are induced by noncommutative Poisson structures encoded on fission algebras.

\subsection{Previous results}
\label{sec:intro-Boalch-work}

\subsubsection{Fission varieties}
Tame character varieties, usually simply called  character varieties, are spaces of complex fundamental group representations of Riemann surfaces. 
Wild character varieties form a class of algebraic varieties that generalize them,
since they enrich
representations to so-called Stokes representations that fix extra monodromy data around each puncture.
From the analytic perspective, 
they parametrize more general classes of connections where irregular singularities are allowed along the boundary. 
Interestingly, wild character varieties are symplectic/Poisson algebraic varieties. This result was derived by Boalch \cite{Bo01} as an adaptation of the infinite-dimensional Atiyah--Bott symplectic quotient to this irregular/wild setting. Later, more algebraic finite-dimensional  approaches were privileged, based on convenient extensions of the quasi-Hamiltonian geometry of \cite{AMM98} from the tame to the wild setting \cite{Bo07}. 
We recommend the excellent \cite{Bo13,Bo14a} as introductions to this exciting area of research.

All (tame) character varieties may be constructed inductively by gluing together conjugacy classes and pairs of pants. Analogously, all wild character varieties may be obtained inductively by gluing together conjugacy classes, pairs of pants, and \emph{higher fission spaces} ${}_G\mathcal{A}^r_H$ with $r\geq 1$. 
As explained in \cite[\S3]{Bo14}, the latter can be described in a pretty simple way and carry a quasi-Hamiltonian $(G\times H)$-structure, see \cite[Theorem 3.1]{Bo14}; if $r=1$ the space ${}_G\mathcal{A}^1_H$ is just the double of $G$ from \cite{AMM98}. 
However, the converse is not true: if we consider the class of symplectic/Poisson or quasi-Hamiltonian varieties that arise by gluing together these pieces and performing reduction, we will obtain a larger class of spaces than wild character varieties: \emph{fission varieties} \cite[\S3.2]{Bo14}.

\subsubsection{Multiplicative quiver varieties (after Boalch)} \label{sec:intro-Boalch-MQV}

As mentioned above, the work of Van den Bergh \cite{VdB1} was motivated by the multiplicative analogues of Nakajima's quiver varieties, as introduced by Crawley-Boevey and Shaw in \cite{CBShaw}. 
To sketch his result, we define the \emph{Van den Bergh space} (see \cite{BK16,Bo14,BEF,Ya})
\begin{equation}
\mathcal{B}^{\VdB}(V,W):=\Big\{(a,b)\in \Hom(W,V)\oplus \Hom(V,W)\mid (1+ab)\text{ is invertible}\Big\}
\label{eq:VdB-spaces-intro}
\end{equation}
for two finite dimensional complex vector spaces $V$ and $W$.
The key observation is that $\mathcal{B}^{\VdB}(V,W)$ is a nondegenerate Hamiltonian quasi-Poisson $\GL(V)\times\GL(W)$-space \cite{AKSM02} (that is, a quasi-Hamiltonian  $\GL(V)\times\GL(W)$-space). 
The corresponding Poisson variety obtained by quasi-Hamiltonian reduction is the multiplicative quiver variety associated with the quiver $\overline{A}_2$. 
By fusion of different copies of \eqref{eq:VdB-spaces-intro} and then performing reduction, we can obtain a Poisson structure on any multiplicative quiver variety.    
Thus, we can regard the Van den Bergh spaces as the building blocks that underlie the (quasi-)Poisson geometry of multiplicative quiver varieties. 

In fact, the Hamiltonian quasi-Poisson spaces $\mathcal{B}^{\VdB}(V,W)$ 
are examples of Boalch's fission varieties. In \cite[\S4]{Bo14}, Boalch observed that if $G=\GL(V\oplus W)$ and $H=\GL(V)\times\GL(W)$, there exists an isomorphism of (nondegenerate) Hamiltonian quasi-Poisson $H$-spaces between $\mathcal{B}^{\VdB}(V,W)$ and $\mathcal{A}^2(V,W)\sslash_1 G$, the quasi-Hamiltonian reduction of the higher fission space $\mathcal{A}^2(V,W):={}_G\mathcal{A}^2_H$ with respect to $G$ at the identity $1 \in G$. 
Next, Boalch makes an observation of the utmost importance: since the Van den Bergh spaces \eqref{eq:VdB-spaces-intro} are the basic building blocks for the multiplicative quiver varieties of \cite{CBShaw}, his results suggest how to find other building blocks, 
such as $\mathcal{A}^2(V_1,\ldots,V_k)\sslash_1 G$ (the higher fission space is $\mathcal{A}^2(V_1,\ldots,V_k):={}_G\mathcal{A}^2_H$, $G=\GL(\oplus_{i=1}^kV_i)$ and $H=\times_{i=1}^k\GL(V_i)$). 
Then, these additional building blocks allow to construct more general multiplicative quiver varieties by quasi-Hamiltonian reduction. The varieties hence obtained will again be fission varieties. 

\subsubsection{Colored quivers and colored multiplicative quiver varieties} \label{sec:intro-color-things}
Inspired by the construction of Nakajima quiver varieties (see Table \ref{Table}), Boalch also realized that constructing more general multiplicative quiver varieties
amounts to attaching an algebraic symplectic manifold to a graph and some data on the graph. 
From this point of view, he came up with \emph{colored graphs}.  
In brief, these are finite graphs $\Upsilon$ with nodes $I$ and a coloring of each edge 
(\emph{i.e.,} a map $\mathfrak{c}\colon \Upsilon\to C$ to the set $C$ of colors) 
such that each monochromatic subgraph is a complete $k$-partite graph ---\,see Definition \ref{Def:ColQ}.
Now, given a colored quiver $\Upsilon$, fix a finite dimensional graded complex vector space $V=\bigoplus_{i\in I}V_i$. 
By \cite[Corollary 5.7]{B15}, 
 there exists a canonical nonempty $H$-invariant open subset $\Rep^*(\Upsilon,V)\subset\Rep(\Upsilon,V)$
which is a quasi-Hamiltonian $H$-space and admits a group-valued moment map $\widehat{\Phi}:\Rep^*(\Upsilon,V) \to H$.

Next, given parameters $d\in\mathbb{Z}^I_{\geq 0}$ (corresponding to the vector space $V=\bigoplus_{i\in I}V_i$ through $\dim V_i=d_i$), and $q\in(\mathbb{C}^*)^I$, 
the \emph{colored multiplicative quiver variety} $\mathcal{M}(\Upsilon,q,d)$ (attached to the colored quiver $\Upsilon$) 
is the quasi-Hamiltonian reduction of $\Rep^*(\Upsilon,V)$ at the value $q$ of the moment map. That is 
\[
 \mathcal{M}(\Upsilon,q,d):=\Rep^*(\Upsilon,V)\sslash_q H :=\widehat{\Phi}^{-1}\Big(\prod_I q_i \Id_{V_i}\Big)/\!\!/H\,.
\]
We work in the algebraic category, so that the quotient on the right-hand side is taken according to affine geometric invariant theory (GIT). 
Consequently, we get that the open set $\mathcal{M}^{\text{st}}(\Upsilon,q,d)\subset\mathcal{M}(\Upsilon,q,d)$ of stable points is a smooth algebraic symplectic manifold.
It is important to note \cite[Proposition 6.8]{B15} that for two inequivalent ways to color $\Upsilon$, the corresponding colored multiplicative quiver varieties will not be isomorphic in general. Furthermore, if each monochromatic subgraph $\Upsilon_c=\mathfrak{c}^{-1}(c)$, where $c$ is a color, just consists of one arrow (this is the ``tautological coloring''), then $\mathcal{M}(\Upsilon,q,d)$ is a multiplicative quiver variety in the usual sense of \cite{CBShaw}. 
In consequence, colored multiplicative quiver varieties generalize standard multiplicative quiver varieties.

\subsubsection{Fission algebras} \label{sec:intro-fission-algebras}
Finally, we can introduce one of the main objects of study in this article: \emph{fission algebras}.  
As shown in Table \ref{Table}, deformed (additive) preprojective algebras govern Nakajima quiver varieties, while multiplicative preprojective algebras control multiplicative preprojective varieties. Furthermore, Boalch \cite[\S12]{B15} introduced fission algebras, a class of associative noncommutative algebras linked to colored quivers, that regulate colored multiplicative quiver varieties. As a result, fission algebras serve as a generalization of the multiplicative quiver algebras presented in \cite{CBShaw}
---as a byproduct, the former are linked to the generalized double affine Hecke algebras of Etingof, Oblomkov and Rains \cite{EOR07}.
We define the \emph{Boalch algebra} $\mathcal{B}(\Upsilon)$ as the path algebra of a quiver with relations obtained from  the input colored quiver $\Upsilon$ by adding arrows and loops, see Definition \ref{Def:Boalch-algebra}. Then, the fission algebra $\mathcal{F}^q(\Upsilon)$ is the quotient of $\mathcal{B}(\Upsilon)$ by the two-sided ideal essentially generated by setting the loops equal to $q$; they will be carefully introduced in \ref{sec:fission-algebras}. 
Remarkably, as explained in \cite[\S12]{B15} and illustrated in Table \ref{Table}, the representation theory of $\mathcal{F}^q(\Upsilon)$ governs the geometry of the colored multiplicative quiver varieties $\mathcal{M}(\Upsilon,q,d)$ for each dimension vector $d$. 

\subsubsection{Van den Bergh's noncommutative perspective} \label{sec:intro-VdB-perspective}
The key point that motivated the present article is that the Hamiltonian quasi-Poisson structure on the Van den Bergh spaces $\mathcal{B}^{\VdB}(V,W)$ introduced in \eqref{eq:VdB-spaces-intro} can be obtained from a natural noncommutative structure. Indeed, 
this structure can be derived from a (nondegenerate) double quasi-Poisson bracket and a multiplicative moment map on a suitable localization of $\kk\overline{A}_2$ (see \cite[Theorem 6.5.1]{VdB1} and \cite[Proposition 8.3.1]{VdB2}), as well as the application of the Kontsevich--Rosenberg principle (see \cite[\S 7.12]{VdB1}). 
In other words, the Hamiltonian quasi-Poisson $\Gl(V)\times\Gl(W)$-structure on $\mathcal{B}^{\VdB}(V,W)$ is induced by a Hamiltonian double quasi-Poisson structure on a certain quiver path algebra. If we perform quasi-Hamiltonian reduction and consider the corresponding multiplicative quiver variety, its Poisson structure can be understood in terms of an  $H_0$-Poisson structure \cite{CB11} on the multiplicative preprojective algebra $\Lambda^q(A_2)$ associated with $A_2$, see \cite[Theorem 6.8.1]{VdB1}. Furthermore, the same statements hold for an arbitrary quiver and multiplicative quiver variety. This viewpoint is summarized in Table \ref{Table}.

In complete analogy with this original case, we would like to encode the Poisson structure of \emph{all} colored multiplicative quiver varieties introduced by Boalch directly into the fission algebra $\mathcal{F}^q(\Upsilon)$, as we explain now.

 \subsection{Main results} 
 \label{sec:intro-our-results} 
 
Our aim is to make a step towards the far-reaching program of understanding the different classes of Poisson varieties introduced in \ref{sec:intro-Boalch-work}  using noncommutative structures and the Kontsevich--Rosenberg principle. Moreover, since such varieties can be obtained from a Hamiltonian quasi-Poisson space, we want to encode the structure of these spaces in terms of their noncommutative counterparts: Hamiltonian  double quasi-Poisson algebras, see Definition \ref{defn:double-quasi-Hamiltonian}.   
 More precisely, the initial motivation for this article was to deepen Boalch's insight on the ``noncommutative quasi-Hamiltonian spaces''  that yield the higher fission spaces.
 One of the themes in \ref{sec:intro-Boalch-work} was the successive generalization of the quiver $\overline{A}_2$ (and consequently of the Van den Bergh spaces \eqref{eq:VdB-spaces-intro}) in terms of complete $k$-partite graphs and colored quivers. Hence, as displayed in Table \ref{Table}, in this article we want to extend the noncommutative picture that was originally unveiled by Van den Bergh in this direction. 
 
It is interesting to observe that the quiver $\overline{A}_2$ (and $\mathcal{B}^{\VdB}(V,W)$) may be alternatively generalized in terms of the double of the quiver $\Gamma_n$ on two vertices and $n$ equioriented arrows.
This direction has been undertaken in \cite{FF}. In that article, we prove that an appropriate localization of the path algebra $\kk\overline{\Gamma}_n$ is endowed with an explicit Hamiltonian double quasi-Poisson structure, whose multiplicative moment map is given in terms of Euler continuants, which were introduced by Euler in 1764. This result generalizes the Hamiltonian double quasi-Poisson algebra associated with the quiver $\overline{A}_2$ by Van den Bergh because  $\overline{\Gamma}_1=\overline{A}_2$.
This shows that some of the geometric structures found by Boalch for a particular class of wild character varieties in \cite{B18} have a noncommutative origin in application of the Kontsevich--Rosenberg principle.

 \subsubsection{The Conjecture for colored multiplicative quiver varieties} 
 Recall from \ref{sec:intro-Boalch-MQV} that Boalch introduced colored multiplicative quiver varieties
as generalizations of multiplicative quiver varieties, being attached to colored quivers.  
In this article, we formulate Conjecture \ref{Conj:Main} which states that for each colored quiver $\Upsilon$ over $I$ vertices, the Boalch algebra $\Bc(\Upsilon)$ is endowed with a Hamiltonian double quasi-Poisson structure. 
The significance of this conjecture relies on the following two consequences:
\begin{enumerate}
\item [(i)]
Let $V=\bigoplus_I V_i$ be an $I$-graded vector space, and form $H=\prod_I\GL(V_i)$.
The combination of Conjecture \ref{Conj:Main} and the Kontsevich--Rosenberg principle (see Theorem \ref{thm:KR-theorem}) induces a Hamiltonian quasi-Poisson $H$-structure on $\Rep(\Bc(\Upsilon),V)\simeq\Rep^*(\Upsilon, V)$ in the sense of \cite{AKSM02}.
\item [(ii)]
By Conjecture \ref{Conj:Main} and \cite[Proposition 5.1.5]{VdB1}, the fission algebra $\mathcal{F}^q(\Upsilon)$ carries an $H_0$-Poisson structure (in the sense of \cite{CB11}). Then, the Kontsevich--Rosenberg principle induces a standard Poisson structure on the colored multiplicative quiver varieties.
\end{enumerate}
As a byproduct of this construction, establishing the conjecture would extend the dictionary summarized in Table \ref{Table} that relates various algebro-geometric structures defined either on associative algebras constructed from quivers, or on associated moduli spaces of representations.
To ascertain Conjecture \ref{Conj:Main}, we first note that we can simplify the task: 
if one proves the conjecture for monochromatic quivers,
then it holds for any colored quiver. This is proved in Lemma \ref{Lem:MonochromaticConjecture}. We will also prove Conjecture \ref{Conj:Main} in two interesting cases, which we explain in \ref{sec:intro-our-111}. 
Let us also note that (smooth loci of) colored multiplicative quiver varieties are symplectic varieties, 
obtained from quasi-Hamiltonian spaces in the sense of \cite{AMM98}. Thus, we believe that 
the double quasi-Poisson structures appearing in Conjecture \ref{Conj:Main} are nondegenerate, giving rise to \emph{quasi-bisymplectic algebras} (see \cite[\S 6]{VdB2}). 
We shall pursue this direction in the future.

Before delving into our results, let us address an important remark.
It is tempting to say that, to tackle Conjecture \ref{Conj:Main}, one only needs to consider the  formulas that exist in the known geometric context \cite{Bo14,B15} and make them ``noncommutative''.  
However, there are two substantial issues in this simple idea. 
First, while we consider these spaces as equipped with a Poisson bracket, the latter is in fact nondegenerate and the varieties are usually defined as symplectic varieties. They are obtained by reduction from quasi-Hamiltonian spaces endowed with a ``quasi-symplectic'' form \cite{AMM98}, and the corresponding  nondegenerate quasi-Poisson bracket is unknown in general. Hence, one would need to properly work out the quasi-symplectic -- quasi-Poisson yoga for these varieties, which is tricky. 
Second, understanding the quasi-Poisson variety as the representation space associated with a particular algebra is not obvious; see the proof of \cite[Proposition 5.3]{B15}.  
Therefore, the naive approach of writing the quasi-Poisson bracket in a noncommutative form gives rise to two nontrivial steps. This is the reason why in this article we preferred to stay at the level of noncommutative algebras and we tried to obtain a Hamiltonian double quasi-Poisson algebra structure by analyzing the algebras $\Bc(\Upsilon)$ and $\mathcal{F}^q(\Upsilon)$ directly.

 \subsubsection{Towards the conjecture: the monochromatic interval and triangle} \label{sec:intro-our-111}
We first undertake the study of Conjecture \ref{Conj:Main} by considering the monochromatic interval case. If we consider the partition $\{1,2\}=\{1\}\sqcup \{2\}$, the input is the complete $2$-partite graph $\mathcal{I}$ with two vertices $\{1,2\}$ and one colored edge.
Then, applying the construction in \ref{sec:fission-algebras}, the output is the Boalch algebra $\mathcal{B}(\mathcal{I})$, which is explicitly described in \ref{sec:the-11-monochromatic-case}. 
The first result of this article is Proposition \ref{Pr:OneArrow}. It states that the algebra $\Bc(\mathcal{I})$ admits a double quasi-Poisson bracket, explicitly defined on generators in Lemma \ref{Lemma:OneArrow}, with multiplicative moment map given by $\gamma_1+\gamma_2$, the sum of two loops. Thus, $\Bc(\mathcal{I})$ is a Hamiltonian double quasi-Poisson algebra, confirming Conjecture \ref{Conj:Main} and its consequences for the monochromatic interval. The idea of the proof is to reexpress $\Bc(\mathcal{I})$ in terms of Van den Bergh's Hamiltonian double quasi-Poisson algebra \cite{VdB1} associated with the one-arrow quiver $v_{12}:2\to 1$.  

Next, we address the case that will be expectedly significant in the general proof of Conjecture \ref{Conj:Main}.
It corresponds to the Boalch algebra $\mathcal{B}(\Delta)$ attached to the monochromatic triangle, which is the complete $3$-partite graph $\Delta$ over three vertices $\{1,2,3\}$ with the partition $\{1\}\sqcup \{2\}\sqcup \{3\}$. In this case, $\mathcal{B}(\Delta)$ is generated by the symbols $\{e_i, v_{ij}, v_{ji}, w_{ij}, w_{ji},  \gamma^{\pm 1}_i\}$, 
where $1\leq i<j\leq 3$, that is, we have 12 arrows and 3 loops (with their inverses); see Figure \ref{fig:Quiver111}. 
These symbols are subject to the intricate defining relation \eqref{eq:relation-Boalch-algebra-big}. 
This relation can be further decoupled into 9 highly non-linear equations written in \eqref{eq:Boalch-eq-111}, attending to the occurring idempotents.

Theorem \ref{thm:triangle} is the most fundamental result of this article: we succeed in explicitly proving that $\mathcal{B}(\Delta)$ is a Hamiltonian double quasi-Poisson algebra, whose multiplicative moment map is given by $\gamma_1+\gamma_2+\gamma_3$. So, Conjecture \ref{Conj:Main} and its consequences hold for the monochromatic triangle ---\,see Corollaries \ref{cor:1} and \ref{cor:2}.

We should emphasize that the double quasi-Poisson bracket found in Theorem \ref{thm:triangle} is quite surprising because some counter-intuitive linear terms appear in some (but not all of the) double brackets. However, these terms are crucial to prove that $\gamma_1+\gamma_2+\gamma_3$ is a multiplicative moment map for $\lr{-,-}$ using the non-linear identities \eqref{eq:Boalch-eq-111}.
In this sense, the double brackets dramatically depend on the involved indices. 
Furthermore, the proof of the quasi-Poisson property (\emph{i.e.}, equation \eqref{qPabc} holds) for $\lr{-,-}$ on $\Bc(\Delta)$ is intricate. 
Indeed, in Section \ref{sec:proof-quasi-Poisson-property} we derive general conditions for a particular class of double brackets to be quasi-Poisson, and we end up by showing that the double bracket defined on $\Bc(\Delta)$ as part of Theorem \ref{thm:triangle} satisfies those conditions. We hope that the derivations performed in Section \ref{sec:proof-quasi-Poisson-property} will be crucial to tackle the remaining cases of Conjecture \ref{Conj:Main}.

\medskip

To conclude, we would like to draw the reader's attention to a potential application of Theorem \ref{thm:triangle}. It was observed in \cite{BEF,CF1,F} that tautologically colored multiplicative quiver varieties attached to an extension of the cyclic quiver can naturally be regarded as the phase spaces of some integrable systems.  Furthermore, the double quasi-Poisson bracket associated with these quivers was an essential tool to prove the Poisson-commutativity of the functions defining the integrable systems under consideration in \cite{CF1,F}. 
Therefore, we expect that these works can be adapted to the case of colored multiplicative quiver varieties and yield new classes of integrable systems. 
The first non-trivial case to consider may be the extension of the monochromatic triangle, seen as  a cyclic quiver on $3$ vertices (with a \emph{different} color), by one arrow as in \cite{CF1}. One  would certainly benefit from the double quasi-Poisson bracket unveiled in Theorem \ref{thm:triangle} for calculations.

\medskip

\noindent\textbf{Layout of the article.} In Section \ref{sec:nc-Poisson-geom}, we start by introducing double quasi-Poisson brackets and Hamiltonian double quasi-Poisson algebras following Van den Bergh \cite{VdB1}, which are central throughout this work. We also explain how they are related to the (noncommutative) $H_0$-Poisson structure of Crawley-Boevey \cite{CB11}, and their importance with respect to the Kontsevich--Rosenberg principle. We define colored quivers and fission algebras in Section \ref{sec:conjecture} following Boalch \cite{B15}, and we also introduce the \emph{Boalch algebra}, from which fission algebras can be obtained as quotients. We can then turn to the main subject of this article: the statement of Conjecture \ref{Conj:Main} about the existence of a Hamiltonian double quasi-Poisson structure on Boalch algebras.  We prove two instances of the conjecture in Section \ref{sec:Towards}, which are associated with the monochromatic interval and the monochromatic triangle. Section \ref{sec:proof-quasi-Poisson-property} deals with a general result needed to prove the Conjecture in the triangle case. In Appendix \ref{sec:the-complete-list}, for the reader's convenience, we give the complete list of double brackets for the monochromatic triangle case.

 \medskip

\noindent\textbf{Comment on the field.} 
We always assume the field $\kk$ to be of characteristic zero.  
When considering GIT quotients, $\kk$ needs also to be algebraically closed so that invariant elements are trace functions, cf.~\ref{sec:KR-principle}. 
For comparison with results of Boalch, one needs $\kk=\CC$. 
Note that, for computations involving double quasi-Poisson algebras, one could simply ask for $\kk$ to be a general field where $2,3$ are invertible.  

 \medskip

\noindent\textbf{Acknowledgments.} M. F. was supported by a Rankin-Sneddon Research Fellowship of the University of Glasgow.  
D. F. was supported by the Alexander von Humboldt Stiftung in the framework of an Alexander von Humboldt professorship endowed by the German Federal Ministry of Education and Research. 
The authors wish to thank Luis \'Alvarez-C\'onsul, Damien Calaque, Oleg Chalykh and William Crawley-Boevey for useful discussions and interesting comments. 
Special thanks are due to the referee for helping us improve the text, and to Philip Boalch for enlightening discussions and email exchanges.

\begin{center}
\begin{landscape}
\begin{small}
\begin{table}[t]
\begin{tabular}[center]{|c|c|c|c|c||c|c|}  
\hline
&  \multicolumn{4}{|c|}{\textbf{\textsf{Noncommutative setting}}}  &  \multicolumn{2}{|c|}{\textbf{\textsf{Commutative Setting}} }  
\\
\hline
& \textbf{Graphs}  &    \begin{tabular}{@{}c@{}}\textbf{Associative} \\  \textbf{algebra}\end{tabular}& \textbf{Moment map}  &  \begin{tabular}{@{}c@{}}\textbf{Ruling associative}\\  \textbf{algebra}\end{tabular}&  \begin{tabular}{@{}c@{}}\textbf{Representation}\\ \textbf{variety}\end{tabular} &   \begin{tabular}{@{}c@{}}\textbf{Quiver}\\ \textbf{variety}\end{tabular}
\\ 
\hline
 {\scriptsize\textbf{Nakajima}} &
    \begin{tabular}{@{}c@{}}Quiver \\  $ Q$\end{tabular}
&      \begin{tabular}{@{}@{}c@{}}  {\tiny \circled{A}}\\ Path  algebra \\   $\kk \overline{Q}$\end{tabular}
& \begin{tiny}$\mu=\sum_{a\in Q}[a,a^*]$\end{tiny}
 &      \begin{tabular}{@{}@{}c@{}} Deformed (additive) \\ preprojective algebra\\  $ \Pi^\lambda(Q)= \sfrac{\kk \overline{Q}}{(\mu-\lambda)}   $\end{tabular}
&   \begin{tabular}{@{}c@{}}{\tiny \boxed{A}}\\ \\ $\Rep(\kk \overline{Q},V)$\end{tabular}
&    \begin{tabular}{@{}@{}c@{}}  Nakajima \\quiver variety \\ $ \mathcal{N}(Q,\lambda, d)=$  \\ \begin{tiny}$\Rep(\Pi^\lambda(Q),V))/\!\!/H   $\end{tiny}\end{tabular}
 \\
\hline
    \begin{tabular}{@{}@{}c@{}}{\scriptsize\textbf{Crawley-Boevey}} \\{\scriptsize\textbf{--Shaw};}\\  
    {\scriptsize\textbf{Van den Bergh}} \end{tabular} &
    \begin{tabular}{@{}c@{}}Quiver\\   $Q$\end{tabular}
&      \begin{tabular}{@{}@{}c@{}} {\tiny \circled{B}  }\\Localized \\ path algebra \\  $\mathcal{A}(Q)$ \end{tabular}
 & \begin{tiny}$ \Phi=\prod\limits_{a\in \overline{Q}}(1+aa^*)^{\epsilon(a)}$\end{tiny}
&  \begin{tabular}{@{}@{}c@{}}Multiplicative \\ preprojective algebra \\   $ \Lambda^q(Q)= \sfrac{\mathcal{A}(Q)}{(\Phi-q)}   $ \end{tabular}
 &   \begin{tabular}{@{}c@{}} {\tiny \boxed{B}} \\ \\ $\Rep(\mathcal{A}(Q),V)$ \end{tabular}
 &     \begin{tabular}{@{}@{}@{}@{}c@{}} Multiplicative\\ quiver variety\\  $ \mathbb{M}(Q,q, d)=$ \\\begin{tiny} $\Rep\big(\Lambda^q(Q),V\big)/\!\!/ H $ \end{tiny} \end{tabular}
 \\ 
\hline
{\scriptsize\textbf{Boalch}} &
    \begin{tabular}{@{}@{}@{}@{}c@{}}Colored\\ quiver\\  $\Upsilon$\\ \\ \emph{Definition \ref{Def:ColQ}} \end{tabular}
    &  \cellcolor{cyan!50}  \begin{tabular}{@{}@{}@{}c@{}}  {\tiny \circled{C}}\\ Boalch algebra \\  $\mathcal{B}(\Upsilon)$ \\  \\ \emph{Definition \ref{Def:Boalch-algebra}}\end{tabular}
&  \cellcolor{cyan!50}  \begin{tabular}{@{}@{}c@{}} \begin{tiny}$\Phi=\prod\limits_{\substack{\gamma\text{ loop}\\ c\text{ color}}}\gamma_c$ \end{tiny} \\ \\ \emph{Eqs. \eqref{Eq:wvg} and  \eqref{eq:moment-map-state,ment-conj}} \end{tabular}
 &   \cellcolor{cyan!50}  \begin{tabular}{@{}@{}c@{}c@{}}Fission algebra\\  
 $ \mathcal{F}^q(\Upsilon)= \sfrac{\mathcal{B}(\Upsilon)}{(\Phi-q)}$ \\ \\  \emph{Definition \ref{def:def-fission-algebras}} \end{tabular}
  &    \begin{tabular}{@{}c@{}c@{}} {\tiny \boxed{C}}\\ \\ $\Rep(\mathcal{B}(\Upsilon),V)$ \\ \\ Lemma \ref{lem:rep-non-empty}\end{tabular}
   &    \begin{tabular}{@{}@{}@{}@{}@{}@{}@{}c@{}} Colored\\  multiplicative\\  quiver variety\\ $ \mathcal{M}(\Upsilon,q, d)= $  \\  \begin{tiny}$\Rep\big(\mathcal{F}^q(\Upsilon),V\big)/\!\!/H   $\end{tiny}\\ \ref{sec:KR-principle}\end{tabular}
  \\ 
\hline
\hline
       \begin{tabular}{@{}@{}c@{}}\textbf{Algebro-}\\\textbf{geometric}\\   \textbf{structures}\end{tabular}
       &
   &   \begin{tabular}{@{}@{}@{}c@{}}\begin{tiny}Hamiltonian double\end{tiny}\\  \begin{tiny}Poisson: \end{tiny} \begin{tiny}\circled{A} \end{tiny}\\ \begin{tiny}quasi-Poisson:\end{tiny} \begin{tiny}\circled{B}\end{tiny}, 
   \begin{tiny}\circled{C}\end{tiny}\end{tabular}
& 
& $H_0$-Poisson
&    \begin{tabular}{@{}@{}c@{}}\begin{tiny}(Non-degenerate)\end{tiny}\\ \begin{tiny}Hamiltonian\end{tiny} \\ \begin{tiny}Poisson: \boxed{A}\end{tiny} \\  \begin{tiny}quasi-Poisson: \boxed{B}\end{tiny},\begin{tiny} \boxed{C}\end{tiny}    \end{tabular}
& Poisson
   \\
   \hline
\end{tabular}
\caption{Comparison between (additive) preprojective algebras, multiplicative preprojective algebras and fission algebras, as well as their inputs and representation theories (over an algebraically closed field $\kk$ of characteristic $0$). The blue cells highlight new results obtained (or conjectured) on these objects in the present work. \\
\medskip
For multiplicative preprojective algebras, given a quiver $Q=(I,Q)$, we define $\epsilon\colon\overline{Q}\to\{\pm 1\}$ as $\epsilon(a)=+1$ if $a\in Q$, and $\epsilon(a)=-1$ if $a\in\overline{Q}\setminus Q$. For representations, we pick a vector space $V=\bigoplus_{s\in I}V_s$ with dimension vector $\dim V_s=d_s$, $(d_s)\in\Z_{\geq0}^I$, and we reduce with respect to $H=\prod_{s\in I}\GL_{d_s}(\kk)$ after a choice of trivializations $V_s\cong \kk^{d_s}$.}
  \label{Table}
  \end{table}
\end{small}
\end{landscape}
\end{center}


\section{Noncommutative Quasi-Poisson Geometry}
\label{sec:nc-Poisson-geom}

In this section we introduce the noncommutative structures that will appear in relation to Boalch's fission algebras and colored multiplicative quiver varieties: Hamiltonian double quasi-Poisson algebras (\ref{sec:sec-double-quasi-Ham}) and $H_0$-Poisson structures (\ref{sec:H0-Poisson}). The former induce Hamiltonian quasi-Poisson structures \cite{AKSM02} on the representation schemes, while the latter induce usual Poisson structures on them. Thus, they satisfy the Kontsevich--Rosenberg principle, which will be explained in \ref{sec:KR-principle}. 
Finally, in \ref{sec:CBS-multiplicative} we review Van den Bergh's works \cite{VdB1,VdB2}, where these noncommutative structures naturally appear in the setting of multiplicative preprojective algebras.

 \subsection{Hamiltonian double quasi-Poisson algebras}
 \label{sec:sec-double-quasi-Ham}

In this section we follow \cite{VdB1}, see also \cite{CBEG07,F19}.
We fix a finitely generated associative unital algebra $A$ over a field $\kk$ of characteristic zero, and we write $\otimes=\otimes_\kk$ for brevity. Unless otherwise stated, given $n\in\Z_{\geq0}$, we will consider the $A$-bimodule $A^{\otimes n}$ endowed with its \emph{outer} bimodule structure $A^{\otimes n}_{\out}$:
\[
b_1(a_1\otimes\cdots\otimes a_n)b_2=b_1a_1\otimes \cdots\otimes a_nb_2\quad  \text{in } A^{\otimes n}_{\out},
\]
where $a_1\otimes\cdots\otimes a_n\in A^{\otimes n}$ and $b_1,b_2\in A$. Moreover, if $\mathbb{S}_n$ denotes the group of permutations of $n$ elements $\{1,\dots, n\}$, given $s\in\mathbb{S}_n$ and $a=a_1\otimes\cdots\otimes a_n\in A^{\otimes n}$, we define
\[
\tau_{s}(a)=a_{s^{-1}(1)}\otimes \cdots \otimes a_{s^{-1}(n)}.
\]

An \emph{$n$-bracket} \cite[Definition 2.2.1]{VdB1} is a linear map $\lr{-,\dots,-}\colon A^{\otimes n}\to A^{\otimes n}$ (or equivalently a map $A^{\times n}\to A^{\otimes n}$ linear in all the arguments)  satisfying 
\begin{align*}
&\tau_{(1\dots n)}\circ \lr{-,\dots,-}\circ \tau^{-1}_{(1\dots n)}=(-1)^{n+1}\lr{-,\dots,-},
\\
&\lr{a_1,a_2,\dots, a_{n-1},a_na'_n}=a_n\lr{a_1,a_2,\dots, a_{n-1},a'_n}+\lr{a_1,a_2,\dots, a_{n-1},a_n}a'_n,
\end{align*}
for all $a_1,\dots,a_n,a'_n\in A$. 
In other words, the first identity means that the $n$-bracket $\lr{-,\dots,-}$ is cyclically antisymmetric, and the second one states that it is a derivation $A\to A^{\otimes n}$ in its last argument for $A^{\otimes n}_{\out}$.  
We will call a 2-bracket (resp. a 3-bracket) a \emph{double} bracket (resp. a \emph{triple} bracket).
Since double brackets will be essential in this article, we provide their explicit definition:

\begin{defn}[\cite{VdB1}]
A \emph{double bracket} on $A$ is a $\kk$-bilinear map $\dgal{-,-}:A\times A \to A \otimes A$, satisfying for any $a,b,c \in A$,
\begin{align} 
 \dgal{a,b}&=-\tau_{(12)}\dgal{b,a} &&\text{(cyclic antisymmetry)},  \label{Eq:cycanti}
\\
 \dgal{a,bc}&=\dgal{a,b}c+b\dgal{a,c}  &&\text{(right Leibniz rule)}. \label{Eq:outder}
\end{align}
\end{defn}
Using \eqref{Eq:cycanti}, it is straightforward to see that \eqref{Eq:outder} is equivalent to 
\begin{equation}\label{Eq:inder}
 \dgal{bc,a}=\dgal{b,a}\ast c+b\ast\dgal{c,a} \qquad \text{\emph{(left Leibniz rule)}}, 
\end{equation}
where $\ast$ denotes the \emph{inner} $A$-bimodule structure on $A\otimes A$, given by $a \ast (b'\otimes b'')\ast a'=( b' a') \otimes (a b'')$; throughout this article, we shall use Sweedler's notation.
From \eqref{Eq:outder} and \eqref{Eq:inder}, note that it suffices to define double brackets on generators of $A$.

Next, recall that given an associative $\kk$-algebra $B$, a \emph{$B$-algebra} will mean an associative algebra $A$ together with a unit preserving algebra morphism $B\to A$.
From now on, we assume that the unit in $A$ admits a decomposition $1=\sum_{s\in I} e_s$ in terms of a finite set of orthogonal idempotents, i.e. $|I|\in \Z_{>0}$ and $e_s e_t = \delta_{st} e_s$.  In that case, we view $A$ as a $B$-algebra for $B=\oplus_{s\in I} \kk e_s$. Then, we naturally extend the definition of an $n$-bracket to require $B$-bilinearity: it vanishes if one of its arguments is in $B$.

Given a double bracket $\lr{-,-}$ on $A$, let us introduce the following extension:
\[
\lr{a,b\otimes c}_L:=\lr{a,b}\otimes c\in A^{\otimes 3},
\]
for all $a,b,c\in A$. This allows us to define the following operation:
\begin{equation}
\label{Eq:TripBr}
 \lr{a,b,c}:=\lr{a,\lr{b,c}}_L+\tau_{(123)}\lr{b,\lr{c,a}}_L+\tau_{(132)}\lr{c,\lr{a,b}}_L\,; 
\end{equation}
it is easy to see that \eqref{Eq:TripBr} is a triple bracket. 
Let us also define the following operation 
\begin{equation}
\begin{aligned}
&\quad \lr{a,b,c}_{\qP}
\\
&:=\frac14 \sum_{s\in I} \Big(
c e_s a \otimes e_s b \otimes e_s  - c e_s a \otimes e_s \otimes b e_s - c e_s \otimes a e_s b \otimes e_s 
+ c e_s \otimes a e_s \otimes b e_s \\
&\qquad \quad - e_s a \otimes e_s b \otimes e_s c + e_s a \otimes e_s \otimes b e_s c + e_s \otimes a e_s b \otimes e_s c - e_s \otimes a e_s \otimes b e_s c \Big),
\end{aligned}
\label{eq:triple-bracket-E3}
\end{equation}
which is also a triple bracket. 

\begin{defn}[\cite{VdB1}, Definition 5.1.1]
Let $A$ be an associative $B$-algebra equipped with a double bracket $\lr{-,-}$.
The double bracket is called \emph{quasi-Poisson} (or we say the \emph{quasi-Poisson property} holds) if we can equate the triple brackets \eqref{Eq:TripBr} and \eqref{eq:triple-bracket-E3}, that is, 
\begin{equation}
  \dgal{a,b,c}=\dgal{a,b,c}_{\qP}
   \label{qPabc}
\end{equation}
for all $a,b,c\in A$. The pair $(A,\lr{-,-})$ is called a \emph{double quasi-Poisson algebra}.
\end{defn}

There is a related construction with a condition simpler than \eqref{qPabc}. Namely, a {\emph{double Poisson bracket} \cite[Definition 2.3.2]{VdB1} on $A$ is a double bracket $\lr{-,-}\colon A^{\otimes 2}\to A^{\otimes 2}$ such that \eqref{Eq:TripBr} vanishes (that is, $\lr{-,-,-}=0$). 
In that case, the pair $(A,\lr{-,-})$ is called a \emph{double Poisson algebra}. 

\begin{rem}
By construction, the operations \eqref{Eq:TripBr} and \eqref{eq:triple-bracket-E3} are triple brackets. Thus, it is a simple computation to check that if  \eqref{qPabc} holds on generators of $A$, then it holds for any triple of elements in $A$. 
\label{rem:modified-Jacobi-only-gen}
\end{rem}

Finally, Van den Bergh adapted the important notion of multiplicative moment maps to the noncommutative context given by double quasi-Poisson algebras.

\begin{defn}[\cite{VdB1}, Definition 5.1.4]
Let $(A,\dgal{-,-})$ be a double quasi-Poisson algebra. A \emph{multiplicative moment map} is an invertible element $\Phi=\sum_{s\in I}\Phi_s$ with $\Phi_s\in e_sAe_s$ such that, for all $a\in A$ and $s\in I$, we have
\begin{equation} \label{Phim}
 \dgal{\Phi_s,a}=\frac12 \Big(ae_s\otimes \Phi_s-e_s \otimes \Phi_s a +  a \Phi_s \otimes e_s-\Phi_s \otimes e_s a\Big).
\end{equation}
Then we call the triple $(A,\dgal{-,-},\Phi)$ a \emph{Hamiltonian double quasi-Poisson algebra}. 
\label{defn:double-quasi-Hamiltonian}
\end{defn}
We observe that \eqref{Phim} and the invertibility of $\Phi_s\in e_s A e_s$ imply for all $a\in A$
\begin{equation*} \label{PhimInv}
 \dgal{\Phi_s^{-1},a}=-\frac12 \Big(a \Phi_s^{-1}\otimes e_s-\Phi_s^{-1} \otimes e_s a +  a e_s \otimes \Phi_s^{-1}-e_s \otimes \Phi_s^{-1} a\Big).
\end{equation*}

\subsection{\texorpdfstring{$H_0$}{H_0}-Poisson algebras}
\label{sec:H0-Poisson}
Recall that the zeroth Hochschild homology of $A$ is the vector space $H_0(A)=A/[A,A]$, where $[A,A]$ is the subset of $A$ spanned by the commutators. We write $\overline{a}$ for the image of $a\in A$ in $A/[A,A]$.

\begin{defn}[\cite{CB11}]\label{defn:H0-Poisson}
An \emph{$H_0$-Poisson structure} on $A$ is a Lie bracket $\{-,-\}$ on $H_0 (A)$ such that, for all $\overline{a}\in H_0(A)$, the map $\{\overline{a},-\}\colon H_0(A)\to H_0(A)$ lifts to a derivation $A\to A$.
\label{def:H0-Poisson-str}
\end{defn}

Though there are examples of $H_0$-Poisson structures that do not come from double Poisson structures (see \cite[p. 208]{CB11}), double Poisson algebras induce $H_0$-Poisson algebras in a direct way. As proved in \cite[Corollary 2.4.6]{VdB1}, if $(A,\lr{-,-})$ is a double Poisson algebra and we define the \emph{associated bracket} $\{-,-\}:=m\circ\lr{-,-}$, then $H_0(A)$ equipped with the bracket $\{-,-\}$ is a Lie bracket, thus obtaining an $H_0$-Poisson structure. 

Remarkably, $H_0$-Poisson structures can be obtained from Hamiltonian double quasi-Poisson algebras,  as in Definition \ref{defn:double-quasi-Hamiltonian}. This is due to the following noncommutative counterpart of quasi-Hamiltonian reduction:

\begin{prop}[\cite{VdB1}, Proposition 5.1.5]
Consider a Hamiltonian double quasi-Poisson algebra $(A,\lr{-,-},\Phi)$, and fix an invertible element $q\in B^\times$. We define $\overline{A}=A/(\Phi-q)$. Then the associated bracket $\{-,-\}$ descends to an $H_0$-Poisson structure on $\overline{A}$.
\label{prop:double-quasi-Ham-red}
\end{prop}

\subsection{The Kontsevich--Rosenberg principle}
\label{sec:KR-principle}
The \emph{Kontsevich--Rosenberg principle}~\cite{KR00} states that a noncommutative structure on an associative algebra $A$ has algebro-geometric meaning if it induces the corresponding standard algebro-geometric structures on the representation schemes $\Rep(A,d)$.

From now on, let $\kk$ be an algebraically closed field of characteristic zero, $B=\oplus^n_{s=1}\kk e_s$ be a semisimple $\kk$-algebra, and $A$ be an associative $B$-algebra, which is finitely generated (over $B$). Following \cite[\S7]{VdB1}, let $d=(d_1,\dots,d_n)\in\Z_{\geq0}^n$, and we put $N:=\sum^n_{s=1} d_s$. Also, we assume that $B$ is diagonally embedded in $M_N(\kk):=M_{N\times N}(\kk)$---\,the idempotent $e_s$ is nothing but the identity matrix in $M_{d_s\times d_s}(\kk)$.
 Now, we define the functor on the category of commutative $\kk$-algebras 
\begin{equation}
\Rep_d (A)\colon \texttt{CommAlg}_\kk\longrightarrow \texttt{Sets}, \quad C\longmapsto \Hom_{B}\Big(A, M_N(C)\Big).
\label{eq:functor-of-representation}
\end{equation}
Building on the work of Bergman \cite{Bergman} and Cohn \cite{Cohn} (see \cite{BKR,BFR} for excellent expositions and insightful generalizations), it is well known that the functor \eqref{eq:functor-of-representation} is representable: we have an adjunction (see \cite[Proposition 3]{BFR})
\begin{equation*}
\Hom_\kk(A_d,C)=\Hom_{B}\Big(A,M_N(C)\Big).
\label{adjunction-1}
\end{equation*}
Consequently, we define the \emph{representation scheme} as the
 affine scheme $\Rep(A,d):=\Spec(A_d)$. 
More explicitly, $A_d$ can be described as the commutative algebra generated by the symbols $\{a_{ij}\mid a\in A,\; 1\leq  i,j\leq N\}$ subject to the relations (for $a,a'\in A$, $\lambda\in\kk$, $i,j\in\{1,\dots, N\}$)
\[
(\lambda a)_{ij}=\lambda a_{ij},\quad (a+a')_{ij}=a_{ij}+a'_{ij},\quad (aa')_{ij}= \sum_{1\leq k\leq N} a_{ik}a'_{k j},\quad (e_s)_{ij}=\Delta^s_{ij},
\]
where $\Delta^s_{ij}=1$, if $i=j$ and $\sum_{t<s}d_t<i\leq \sum_{t\leq s}d_t$, and $\Delta^s_{ij}=0$ otherwise, see \cite[p. 211]{CB11}. 

Next, for $a\in A$, we define the element $\tr (a)=\sum^{N}_{s=1}a_{ss}\in A_d$. Since $\tr(ab)=\tr(ba)$, it induces a map $A/[A,A]\to A_d$, which we also denote $\tr$. 
The group $\GL_{d}:=\prod^n_{s=1}\GL_{d_s}(\kk)$ acts on $A_d$ via $g.a_{ij}=\sum^N_{k=1}\sum^N_{\ell=1}g_{ik}a_{k\ell}g^{-1}_{\ell j}$, for all $g\in\GL_{d}$, $a\in A$, $1\leq i,j\leq N$.
The celebrated Le Bruyn--Procesi theorem states that $A^{\GL_{d}}_d$ is the algebra generated by the functions $\tr(a)$ for $a\in A$. On the geometric side, $A^{\GL_{d}}_d$ is the coordinate algebra of the GIT quotient $\Rep(A,d)/\!/\GL_{d}$ classifying isomorphism classes of semi-simple $A$-modules of dimension vector $d$.
The following result will be important in this article:

\begin{thm}[\cite{VdB1}, Theorem 7.12.2, Proposition 7.13.2; \cite{CB11}, Theorem 4.5]
 \hfill
\begin{enumerate}
\item [\textup{(i)}]
Let $\lr{-,-}$ be a double quasi-Poisson bracket on a $B$-algebra $A$. We define 
\[
\{a_{ij},b_{uv}\}=\lr{a,b}^\prime_{uj}\lr{a,b}''_{iv},
\]
for $a,b\in A$. Then $\{-,-\}$ defines a quasi-Poisson bracket on $A_d$. Furthermore, a Hamiltonian double quasi-Poisson algebra induces a Hamiltonian quasi-Poisson  $\GL_{d}$-structure on $\Rep(A,d)$ in the sense of \cite{AKSM02}.
\item [\textup{(ii)}]
If $A$ is equipped with an $H_0$-Poisson structure with associated Lie bracket $\{-,-\}$ then $A^{\GL_{d}}_d$ has a unique Poisson structure with the property
\[
\{\tr(a),\tr(b)\}=\tr\{\overline{a},\overline{b}\},
\]
where $a,b\in A$, and $\overline{a},\overline{b}\in A/[A,A]$. 
\end{enumerate}
\label{thm:KR-theorem}
\end{thm}

\subsection{Application: multiplicative preprojective algebras}
\label{sec:CBS-multiplicative}
A \emph{quiver} $Q$ is an oriented graph, with set of vertices $I$ and set of arrows $Q$. 
We form the double $\overline{Q}$ of $Q$ with the same vertex set $I$ by adding an arrow going in the opposite direction for each $a\in Q$. More precisely, define on $Q$ the tail (resp. head) map $t:Q\to I$ (resp. $h:Q\to I$) which sends an arrow $a\in Q$ to its tail/starting vertex $t(a)\in I$ (resp. head/ending vertex $h(a)\in I$). 
Then $\overline{Q}$ is obtained by adding an arrow $a^\ast : h(a)\to t(a)$ for each $a\in Q$. We naturally extend $h,t$ to $\overline{Q}$, and set $(a^\ast)^\ast=a$ for each $a\in Q$ in order to get an involution $a \mapsto a^\ast$ on $\overline{Q}$. Finally, define $\epsilon:\overline{Q}\to \{\pm1\}$ as the map which takes value $+1$ (resp. $-1$) on arrows originally in $Q$ (resp. on arrows in $\overline{Q}\setminus Q$). 

Fix a field $\kk$ of characteristic zero, and form $\kk\overline{Q}$ as the path algebra of the double $\overline{Q}$ by reading paths from right to left.  We also form the algebra $\Ac(Q)$ obtained by universal localization from the set $S=\{1+a a^\ast \,|\, a \in \overline{Q}\}$. This is equivalent to add local inverses $( e_{h(a)}+a a^\ast)^{-1}$ for each $a \in \overline{Q}$ (i.e. they are inverses to $e_{h(a)}+aa^\ast$ in $e_{h(a)}\Ac(Q) e_{h(a)}$). The algebras $\kk \overline{Q}$ and  $\Ac(Q)$ are seen as $B$-algebras for $B=\oplus_{s\in I}\kk e_s$. 

In their study of a multiplicative version of the important Deligne--Simpson problem in representation theory, Crawley-Boevey--Shaw introduced multiplicative preprojective algebras, whose representation schemes are multiplicative quiver varieties.
 
\begin{defn}[\cite{CBShaw}] Let $Q$ be a quiver with vertex set $I$. Fix a total order $<$ of the arrows of $\overline Q$. The \emph{multiplicative preprojective algebra} with parameter $q\in (\kk^\times)^I$ is the algebra 
\begin{equation*}
 \Lambda^q(Q):=\Ac(Q)/R_q\,,
\end{equation*}
where $R_q$ is the ideal generated by 
\begin{equation} \label{Eq:MQV1}
 \prod_{a\in \overline{Q}}(1+aa^\ast)^{\epsilon(a)} -\sum_{s\in I} q_s e_s\,,
\end{equation}
and the product is taken with respect to the chosen total order.  
\end{defn}

Quite strikingly, Van den Bergh \cite{VdB1} realized that the natural Poisson structure on multiplicative quiver varieties attached to $Q$ is induced, via the Kontsevich--Rosenberg principle, from an $H_0$-Poisson structure (Definition \ref{def:H0-Poisson-str}) on $\Lambda^q(Q)$, which in turn is induced by a Hamiltonian double quasi-Poisson algebra (Definition \ref{defn:double-quasi-Hamiltonian}) on $\Ac(Q)$. 
As emphasized in the introduction, our article can be regarded as a generalization of this point of view to fission algebras, which generalize multiplicative preprojective algebras as we will explain in Lemma \ref{L:AlgIso}. Therefore, for the reader's convenience, it shall be convenient to recall Van den Bergh's results first.

\begin{thm}[\cite{VdB1}, Theorems 6.5.1 and 6.7.1, and Proposition 6.8.1]
\hfill
\begin{enumerate}
\item [\textup{(i)}]
Let $A_2$ be the quiver with vertices $\{1,2\}$, and one arrow $a\colon 1\to 2$.
 Then $\Ac(A_2)$ carries a Hamiltonian double quasi-Poisson structure given by the double quasi-Poisson bracket
\begin{align*}
\lr{a,a}&=\;\;\,0\,,\qquad \quad \lr{a^*,a^*}=0\,;
\\
\lr{a,a^*}&=\;\;\,e_{1}\otimes e_{2} + \frac{1}{2}\Big(a^*a\otimes e_{2}+e_{1}\otimes aa^*\Big)\,;
\\
\lr{a^*,a}&=-e_{2}\otimes e_{1}-\frac{1}{2}\Big(e_{2}\otimes a^*a+aa^*\otimes e_{1}\Big)\,,
\end{align*}
and by the multiplicative moment map $\Phi=(1+aa^*)(1+a^*a)^{-1}$.
\item [\textup{(ii)}]
Let $Q$ be a quiver with vertex set $I$. Fix a total order $<$ on the arrows of $\overline{Q}$ ending at $s$, for each $s\in I$.
Then $\Ac(Q)$ carries a Hamiltonian double quasi-Poisson structure, 
whose multiplicative moment map is given by 
\begin{align*}
\Phi= \prod_{a\in \overline{Q}}(1+aa^\ast)^{\epsilon(a)}\,.
\end{align*}
Here, the product is taken so that the factors $\{e_s(1+aa^\ast)^{\epsilon(a)}e_s \mid h(a)=s\}$ appearing in $e_s\Phi e_s$ respect the chosen total order. 
\item [\textup{(iii)}]
Let $Q$ be a quiver with vertex set $I$. Fix a total order $<$ on the arrows of $\overline{Q}$ ending at $s$, for each $s\in I$. Then the multiplicative preprojective algebra (with parameter $q$) $\Lambda^q(Q)$ is endowed with an $H_0$-Poisson structure, as defined in \cite{CB11}.
\end{enumerate}
\label{tm:VdB-mult-preproj-quasi-Hamilt}
\end{thm}

\begin{rem}\label{rem:convention-on-arrows}
Note that  we follow the convention of reading paths from right to left in this article, as in the works of Boalch \cite{B15} and Crawley-Boevey--Shaw \cite{CBShaw}. This is \emph{different} from the convention of Van den Bergh  \cite{VdB1}. Therefore, the double quasi-Poisson bracket from Theorem \ref{tm:VdB-mult-preproj-quasi-Hamilt} \emph{(i)} looks different from the one given in   \cite[Theorem 6.5.1]{VdB1}, and in \emph{(ii)} we consider the ordering of the arrows ending (not starting) at a vertex. 
\end{rem}


\section{The Conjecture on fission algebras} \label{sec:conjecture}

In this section, we give a precise formulation of the conjecture on the noncommutative Poisson geometry of the colored multiplicative quiver varieties introduced by Boalch.  
In the first two subsections, we present the algebraic structures from \cite{B15} that are necessary to state Conjecture \ref{Conj:Main}.

 \subsection{Colored graphs and quivers} \label{ss:Gr}

Throughout this subsection we partly follow \cite[\S~3,\, 5]{B15}. 

\subsubsection{(Colored) graphs}
\label{ss:Gr-coloured}
\allowdisplaybreaks

Let $\Upsilon$ be a (non-oriented) graph, whose set of vertices is denoted by $I$, while its set of edges is denoted by $\Upsilon$.  

We say that $\Upsilon$ is a \emph{complete $k$-partite graph} if there is a partition of its vertices $I=\sqcup_{j\in J} I_j$ into $k$ non-empty subsets labeled by $J$, $|J|=k$, and such that two vertices $i,i'\in I$ are connected by a single edge if and only if $i\in I_j$ and $i'\in I_{j'}$ for $j,j'\in J$ satisfying $j\neq j'$. In other words, two vertices are connected by one edge if and only if they are in different parts of the graph.

Given $\Upsilon$ a $k$-partite graph, we can define on it an ordering by a choice of a total order of its parts $(I_j)_{j\in J}$ which we thus label by $I_1,\ldots,I_k$, and a total order $<_j$ on the elements of each $I_j$. This induces a total order on the vertex set $I$ of $\Upsilon$ by putting 
$i<i'$ if $i\in I_j,i' \in I_{j'}$ with $j<j'$, or if $i,i'\in I_j$ with $i<_j i'$. We then say that $\Upsilon$ is an \emph{ordered} $k$-partite graph.

\begin{defn} \label{Def:ColG}
 Let $\Upsilon$ be a graph. We say that $\Upsilon$ is a \emph{colored graph} if there exists  
  a finite set $C$ whose elements are called \emph{colors}, and a map $\rc:\Upsilon\to C$ such that for each $c\in C$ the subgraph 
  \begin{equation*}
   \Upsilon_c=\rc^{-1}(c)\subset \Upsilon
  \end{equation*}
is a  complete $k_c$-partite graph for some $k_c\geq 1$.   
\end{defn}
\begin{rem}
 Note that in \cite{B15} each preimage $\rc^{-1}(c)$ is allowed to be a union of complete $k$-partite graphs. Up to refining the choice of colors, this is equivalent to our definition. 
\end{rem}

\subsubsection{Colored quivers} 
\label{ss:Quiver}
Given an ordered $k$-partite graph $\Upsilon$, we can see it as a quiver $\Upsilon$ by replacing each edge $e$ between vertices $i,j\in I$ by an arrow $i\to j$ if $i>j$, or an arrow $j\to i$ if $j>i$. (Since the graph is $k$-partite, this operation is well-defined). The quiver $\Upsilon$ obtained in this way is called a \emph{monochromatic quiver}, or a \emph{simple colored quiver}. The name `color' here can be thought to refer to the chosen ordering.

\begin{defn}  \label{Def:ColQ}
 Let $\Upsilon$ be a quiver. We say that $\Upsilon$ is a \emph{colored quiver}  if there exists  
  a finite set $C$ whose elements are called colors, and a map $\rc:\Upsilon\to C$ such that for each $c\in C$ the subquiver 
  \begin{equation*}
   \Upsilon_c=\rc^{-1}(c)\subset \Upsilon
  \end{equation*}
is a monochromatic quiver obtained from a complete $k_c$-partite graph for some $k_c\geq 1$.  
\end{defn}
Equivalently, we can define a colored quiver $\Upsilon$  as a colored graph whose subgraphs $\Upsilon_c$ are ordered, and hence the edges can be replaced by arrows.

\begin{exmp}
A $2$-partite graph consists of one edge between two arrows, and the associated simple colored quiver is of the form $2\longrightarrow 1$. We can then see that any quiver $Q$ (in the sense of \ref{sec:CBS-multiplicative}) is a colored quiver. Indeed, if $Q$ is a quiver with vertex set $I$, we let $C=Q$ and $\rc:Q\to C$ be the identity. Then $Q_c=\rc^{-1}(c)$ is just $t(c)\overset{c}{\longrightarrow} h(c)$, which is a simple colored quiver. 
\end{exmp}

\subsection{Boalch algebras and fission algebras}
\label{sec:fission-algebras}

Following Boalch \cite[\S~12]{B15}, we assume that $\Upsilon$ is a colored quiver, with vertex set $I$ and color set $C$. We let $\overline{\Upsilon}$ denote the double of $\Upsilon$, obtained by adding a new opposite arrow  for each arrow initially in $\Upsilon$. Specifically, let $I_c$ denote the vertex set of the subgraph $\Upsilon_c=\rc^{-1}(c)$, hence of $\overline{\Upsilon}_c$. Since this subgraph is complete $k_c$-partite, $I_c=\sqcup_{1\leq j_c\leq k_c} I_{c,j_c}$ and there is exactly one arrow $i\to i'$ if $i\in I_{c,j_c}$, $i'\in I_{c,j'_c}$, and $i> i'$ with $j_c\neq j'_c$. This implies that when we consider the double $\overline{\Upsilon}$, we have for each color $c\in C$ and distinct indices $j_c,j_c'\in \{1,\ldots,k_c\}$ precisely one arrow $v_{c,i'i}:i \to i'$ for each $i\in I_{c,j_c}$ and $i'\in I_{c,j_c'}$.

We construct an extension of $\overline{\Upsilon}$, denoted $\widetilde \Upsilon$, by doing the following procedure for each color $c\in C$ : 
\begin{enumerate}
 \item For all $i,j\in I_c$ distinct, we add an edge $w_{c,ij}:j\to i$; 
 \item For all $i\in I_c$, we add a loop $\gamma_{c,i}:i\to i$. 
\end{enumerate}
The extension  $\widetilde \Upsilon$ is a quiver, but it is \emph{not} seen as a colored quiver.  An example of extension for the complete $3$-partite graph on $3$ vertices is depicted in Figure \ref{fig:Quiver111}. 

Next, we fix a field of characteristic zero denoted by $\kk$ and we form the path algebra $\kk\widetilde \Upsilon$. 
Note that our convention on reading paths (see Remark \ref{rem:convention-on-arrows}) implies in the path algebra that for any $c\in C$, 
\begin{equation*}
 v_{c,i'i}=e_{i'}v_{c,i'i}e_i\,, \quad w_{c,ij}=e_i w_{c,ij} e_j \,,  \quad \gamma_{c,i} = e_i \gamma_{c,i} e_i\,,
\end{equation*}
where $i,j\in I_c$ are distinct, and if $i\in I_{c,j_c}$,  
$i'$ is taken in a subpart $I_{c,j'_c}\subset I_c$ distinct from $I_{c,j_c}$.  
We will use  the path algebra $\kk\widetilde \Upsilon$ of the extended double $\widetilde\Upsilon$ to form a new algebra, denoted $\Bc(\Upsilon)$. To state its definition, we introduce for any $c\in C$ the elements\footnote{We sum over indices $i,j\in I_c$. Note that there are elements $w_{c,ij},w_{c,ji}$ for each $i<j$. We do not necessarily have elements $v_{c,ij},v_{c,ji}$ for each $i<j$, as we have such arrows if and only if $i,j$ belong to distinct elements of the partition $I_c=\sqcup_{j\in J}I_{c,j}$. Thus, we set $v_{c,ij}=0$ if $i,j\in I_{c,j_c}$ for some $j_c\in J$.} 
\begin{equation}
 \begin{aligned} \label{Eq:wvg}
  w_{c+}= 1_{I_c}+\sum_{i<j} w_{c,ij}\,, \quad &w_{c-}= 1_{I_c}+\sum_{i>j} w_{c,ij}\,, \quad  \gamma_c=\sum_{i\in I_c}\gamma_{c,i}\,,\\
  v_{c+}= 1_{I_c}+\sum_{i<j} v_{c,ij}\,, \quad &v_{c-}= 1_{I_c}+\sum_{i>j} v_{c,ij}\,,
 \end{aligned}
\end{equation}
where $1_{I_c}:=\sum_{s\in I_c}e_s$ is the idempotent corresponding to the unit  of the subgraph $\Upsilon_c$.  

\begin{defn} \label{Def:Boalch-algebra}
The \emph{Boalch algebra} $\Bc(\Upsilon)$ is the algebra obtained by constructing the universal localization $(\kk\widetilde \Upsilon)_S$ of the path algebra  $\kk\widetilde \Upsilon$  from the set $S=\{(1-1_{I_c})+\gamma_c \,|\, c \in C\}$, and taking  the quotient  by the ideal generated by the following $|C|$ relations 
\begin{equation} \label{Eq:RelB}
 v_{c-}v_{c+}=w_{c+}\gamma_c w_{c-}\,, \quad c\in C\,.
\end{equation}
\end{defn}
The localization is equivalent to adding local inverses $\gamma_{c,i}^{-1}\in e_i \kk\widetilde \Upsilon e_i$ satisfying 
\[
\gamma_{c,i}^{-1}\gamma_{c,i}=e_i=\gamma_{c,i}\gamma_{c,i}^{-1}\]
for each $i\in I_c$ and $c\in C$.
Using the idempotents, note that we can decompose \eqref{Eq:RelB} in three different ways:
\begin{subequations}
\label{Eq:RelB-decompo}
 \begin{align}
&e_i + \sum_{k<i} v_{c,ik}v_{c,ki} = \gamma_{c,i} + \sum_{\ell>i} w_{c,i\ell}\gamma_{c,\ell} w_{c,\ell i}\,, \label{Eq:RelB-decompo-ii}\\  
& v_{c,ij} + \sum_{\substack{k\in I_c \text{ s.t.}\\k<i \text{ and }k<j}} v_{c,ik}v_{c,kj} = 
 w_{c,ij}\gamma_{c,j} + \sum_{\substack{\ell\in I_c \text{ s.t.}\\ \ell>i \text{ and }\ell> j}}  w_{c,i\ell}\gamma_{c,\ell} w_{c,\ell j}\,, \quad \text{ for }i<j \,,   \label{Eq:RelB-decompo-less}\\ 
&  v_{ij} + \sum_{\substack{k\in I_c \text{ s.t.}\\k<i \text{ and }k<j}} v_{c,ik}v_{c,kj} = 
 \gamma_{c,i}w_{c,ij} + \sum_{\substack{\ell\in I_c \text{ s.t.}\\ \ell>i \text{ and }\ell> j}}  w_{c,i\ell}\gamma_{c,\ell} w_{c,\ell j}\,, \quad \text{ for }i>j \,.  \label{Eq:RelB-decompo-greater}
 \end{align}
\end{subequations}

While the algebra $\Bc(\Upsilon)$ is defined from an extension $\widetilde \Upsilon$ of the double quiver 
$\overline{\Upsilon}$, we are in position to explain that it can be directly obtained from the path algebra $\kk\overline{\Upsilon}$. 
This observation should be crucial to prove Conjecture \ref{Conj:Main}, stated below. 
Indeed, we think that the desired double quasi-Poisson bracket on $\Bc(\Upsilon)$ is always induced by one on $\kk\overline{\Upsilon}$ through localization. Therefore the double quasi-Poisson bracket only needs to be defined on each couple of generators taken from the arrows $v_{c,ij}\in \overline{\Upsilon}$. This is the strategy that we use to prove the new case of the monochromatic triangle $\Upsilon=\Delta$ in \ref{sec:the-monochromatic-111}.

\begin{lem}
\label{lem:localisation-Bgamma}
If $\Upsilon$ is monochromatic, i.e. $|C|=1$, there exists a chain of algebra homomorphisms over $\oplus_{s\in I}\kk e_s$ 
\begin{equation*}
 \kk \overline{\Upsilon}=: A_n \to A_{n-1} \to\ldots \to A_1 \to A_0:= \Bc(\Upsilon)\,, \quad n:=|I|\,, 
\end{equation*}
 such that $A_{k-1}$ is obtained by localization of $A_k$ at one element. 
In particular, if $\Upsilon$ is an arbitrary colored quiver, the Boalch algebra $\Bc(\Upsilon)$ is a localization of $\kk\overline{\Upsilon}$. 
\end{lem}
\begin{proof}
We first consider the case where $\Upsilon$ is monochromatic, and we identify its vertex set $I$ with $\{1,\ldots,n\}$ so that we identify the total orders on $I$ and  $\{1,\ldots,n\}$ as well, see \ref{ss:Gr-coloured}.  
We denote the elements of $\overline{\Upsilon}$ as $\{v_{ij}\}$, i.e. we drop the color index $c$ from the discussion made above. 
The proof is a noncommutative analogue of \cite[Proposition~5.3]{B15}.

We show by descending induction that, starting with $A_n$, we can form an algebra $A_k$ by localization of $A_{k+1}$ containing elements 
\begin{equation} \label{Eq:lem-localisation-cond}
\gamma_k\in A_k\,, \quad
\gamma_\ell^{\pm 1}\in A_k\,\, \forall .\ell >k\,, \quad 
 w_{\ell i},w_{i\ell}\in A_k\,\, \forall \,i<\ell \text{ with }\ell>k\,, 
\end{equation}
satisfying the relations from \eqref{Eq:RelB-decompo} that depend only on these elements (we drop the color index $c$).
For the base case, we note that $A_n=\kk \overline{\Upsilon}$ contains elements satisfying \eqref{Eq:lem-localisation-cond}  because 
\begin{equation*}
 \gamma_n:=e_n+\sum_{k<n} v_{nk}v_{kn} \in A_n
\end{equation*}
satisfies \eqref{Eq:RelB-decompo-ii} since there is no $\ell>n$. Next,  $A_{n-1}$ is defined from $A_n$ by localization at $\gamma_n$. We can introduce for each $i<n$ 
\begin{align*}
w_{in}:=& \Big( v_{in} + \sum_{j<i} v_{ij}v_{jn}  \Big) \gamma_{n}^{-1}\,\in A_{n-1}\,,  \\
w_{ni}:=& \gamma_{n}^{-1} \Big( v_{ni} + \sum_{j<i} v_{nj}v_{ji}  \Big) \,\in A_{n-1}\,, 
\end{align*}
and then 
\begin{equation*}
 \gamma_{n-1}:= e_{n-1} + \sum_{j<n-1} v_{n-1,j}v_{j,n-1} - w_{n-1,n}\gamma_{n} w_{n ,n-1}\,\in A_{n-1} \,.
\end{equation*}
We can observe that these elements satisfy the relations in \eqref{Eq:RelB-decompo}, hence the base case holds. 

Next, we assume that the induction hypothesis holds for $A_k$, and we define $A_{k-1}$ from $A_k$ by localization at $\gamma_k$. Thus, we only need to find elements $\gamma_{k-1},w_{k i},w_{ik}\in A_{k-1}$ having the required properties, because we already have  elements 
\begin{equation*} 
\gamma_\ell^{\pm 1}\,\, \forall \ell\geq k\,, \quad 
 w_{\ell i},w_{i\ell}\,\, \forall \,i<\ell \text{ with }\ell> k\,, 
\end{equation*}
belonging to $A_{k-1}$ with the desired properties by induction and localization. 

If we introduce for all $i<k$ 
\begin{equation*}
w_{ik}:= \Big( v_{ik} + \sum_{\substack{j\in I \text{ s.t.}\\j<i \text{ and }j<k}} v_{ij}v_{jk}  
 - \sum_{\substack{\ell\in I \text{ s.t.}\\ \ell>i \text{ and }\ell> k}}  w_{i\ell}\gamma_{\ell} w_{\ell k} \Big) \gamma_{k}^{-1}\,\in A_{k-1}\,,  
\end{equation*}
we easily see that \eqref{Eq:RelB-decompo-less} is satisfied for $i<j$ with $j=k$. In the same way, we can introduce an element $w_{kj}$ for all $j<k$ such that \eqref{Eq:RelB-decompo-greater} is satisfied for $i>j$ with $i=k$. It remains to find an element $\gamma_{k-1}\in A_{k-1}$, which we define from the previously obtained elements as  
\begin{equation*}
\gamma_{k-1}:= e_{k-1} + \sum_{j<k-1} v_{k-1,j}v_{j,k-1} - \sum_{\ell> k-1} w_{k-1,\ell}\gamma_{\ell} w_{\ell ,k-1} \,\in A_{k-1}\,.
\end{equation*}
This element satisfies \eqref{Eq:RelB-decompo-ii}, so we are done. 

By induction, we can obtain $A_1$, which we then localize at $\gamma_1$ to get $A_0$. All the elements have been constructed in order to satisfy \eqref{Eq:RelB-decompo}, which is equivalent to the defining relation \eqref{Eq:RelB} in $\Bc(\Upsilon)$. Thus $A_0=\Bc(\Upsilon)$ can be obtained by a chain of localizations, as expected. 

\medskip

If $\Upsilon$ is an arbitrary colored quiver, we repeat the above construction for each color  $c\in C$ in such a way that we introduce elements $(\gamma_{c,i},w_{c,ij})$ satisfying \eqref{Eq:RelB}. To conclude, it suffices to observe that we end up with $\Bc(\Upsilon)$. 
\end{proof}

\begin{rem}
Since path algebras of quivers are formally smooth, their localization is formally smooth too (see \cite[Proposition 5.3(2)]{CQ} or \cite[\S1.2, (C3)]{KR00}).
Consequently, Lemma \ref{lem:localisation-Bgamma} implies that Boalch algebras $\Bc(\Upsilon)$ are formally smooth.  
\end{rem}

\begin{defn}[\cite{B15}, \S~12] Let $\Upsilon$ be a colored quiver. Construct the extended double $\widetilde \Upsilon$ and the algebra $\Bc(\Upsilon)$ as above. 
Fix an ordering of the colors at each vertex $s\in I$. 
The \emph{fission algebra} with parameter $q\in (\kk^\times)^I$ is the algebra 
\begin{equation*}
 \Fc^q(\Upsilon):=\Bc(\Upsilon)/R_q\,,
\end{equation*}
where $R_q$ is the ideal generated by the $|I|$ elements
\begin{equation} \label{Eq:Fiss1}
 \prod_{\{c\in C \mid  s\in I_c\}} \gamma_{c,s}  - q_s e_s\,, \quad s\in I\,,
\end{equation}
where the product is taken with respect to the chosen order at the vertex $s\in I$.  
\label{def:def-fission-algebras}
\end{defn}

Fix a colored quiver $\Upsilon$ such that each monochromatic subquiver $\Upsilon_c=\rc^{-1}(c)\subset \Upsilon$ consists of exactly one arrow; we call this the \emph{tautological coloring}. 
Then, it is mentioned in \cite[\S~12]{B15} that fission algebras and multiplicative preprojective algebras are isomorphic. In fact, the Boalch algebra $\Bc(\Upsilon)$ and the algebra $\Ac(\Upsilon)$ defined by Van den Bergh (see \ref{sec:CBS-multiplicative}) are also isomorphic. 
We will prove this result as these isomorphisms will be important to establish the simplest case of Conjecture \ref{Conj:Main}, 
see \ref{sec:the-11-monochromatic-case}.   

\begin{lem} \label{L:AlgIso}
If $\Upsilon$ is endowed with the tautological coloring, 
then 
\begin{equation*}
 \Bc(\Upsilon)\simeq \Ac(\Upsilon)\quad \text{ and } \quad \Fc^q(\Upsilon)\simeq \Lambda^q(\Upsilon)\,.
\end{equation*}
\end{lem}
\begin{proof}
We make the identification $C\simeq \Upsilon$.  
 For a fixed color $c\in C$, $\Upsilon_c=\rc^{-1}(c)$ consists of exactly one arrow which we denote by $c:t(c)\to h(c)$. Going to the double $\overline{\Upsilon}$, we add $c^\ast: h(c)\to t(c)$. Let us simplify notations and put $t=t(c),h=h(c)$, so that $c=v_{c,ht},c^\ast=v_{c,th}$ with the above notations.  We must take  the ordering such that $h<t$, since by definition there is an arrow in $\Upsilon_c$ from $i=t$ to $j=h$ if and only if $i>j$. 
 
 Going to the extended double $\widetilde \Upsilon$, we add $w_{c,ht}:t\to h$, $w_{c,th}:h\to t$, $\gamma_{c,t}:t\to t$ and $\gamma_{c,h}:h\to h$. 
 The elements defined in \eqref{Eq:wvg} are 
 \begin{equation*}
 \begin{aligned} 
  w_{c+}= e_h+&e_t+ w_{c,ht}\,, \quad w_{c-}= e_h+e_c+ w_{c,th}\,, \quad  \gamma_c=\gamma_{c,t}+\gamma_{c,h}
  \\
&  v_{c+}= e_h+e_t+c\,, \quad v_{c-}= e_h+e_t+c^\ast\,. 
 \end{aligned}
\end{equation*}
Hence, the algebra $\Bc(\Upsilon)$ is obtained from $\kk \widetilde \Upsilon$ by inverting 
$\gamma_{c,t}$ (resp. $\gamma_{c,h}$) in $e_t \kk \widetilde \Upsilon e_t $ (resp. $e_h \kk \widetilde \Upsilon e_h$), then performing the quotient by the ideal generated by the following relations \eqref{Eq:RelB}: 
\begin{equation*}
 (e_h+e_t+c^\ast)(e_h+e_t+c)=(e_h+e_t+ w_{c,ht}) (\gamma_{c,t}+\gamma_{c,h}) (e_h+e_t+ w_{c,th})\,, 
\end{equation*}
for any $c\in C$. 
Decomposing with the idempotents, this amounts to 
 \begin{align*}
  e_h=&\gamma_{c,h}+ w_{c,ht}\gamma_{c,t}w_{c,th}\,, \\
  e_t+c^\ast c=& \gamma_{c,t}\,, \\
  c=& w_{c,ht}\gamma_{c,t}\,, \\
  c^\ast=& \gamma_{c,t}w_{c,th}\,.
 \end{align*}
By invertibility of $\gamma_{c,t}$ in $e_t \Bc(\Upsilon)e_t$ and $\gamma_{c,h}$ in $e_h \Bc(\Upsilon)e_h$, we find that 
\begin{equation*}
 w_{c,ht}=c\gamma_{c,t}^{-1}\,, \quad w_{c,th}=\gamma_{c,t}^{-1} c^\ast\,. 
\end{equation*}
In particular, we can omit these elements from the generators of $\Bc(\Upsilon)$. 
Furthermore, 
\begin{equation*}
 \gamma_{c,t}=e_t+c^\ast c\,, \quad 
 \gamma_{c,h}=e_h-w_{c,ht}\gamma_{c,t}w_{c,th} = e_h-c(e_t+c^\ast c)^{-1}c^\ast\,.
\end{equation*}
Noting that $[e_h-c(e_t+c^\ast c)^{-1}c^\ast](e_h+cc^\ast)=e_h$ and the same holds when we multiply $ \gamma_{c,h}$ on the left hand-side by $(e_h+cc^\ast)$, we get that 
\begin{equation*}
 \gamma_{c,h}=(e_h+cc^\ast)^{-1}\,.
\end{equation*}
In other words, we have obtained that $\Bc(\Upsilon)$ can be seen as the path algebra $\kk \overline{\Upsilon}$ localized for each $c\in C$ at 
\begin{equation*}
 (1-1_{I_c})+\gamma_{c,h}+\gamma_{c,t}=(1-1_{I_c})+ (e_h+cc^\ast)^{-1}+ (e_t+c^\ast c)=(1+c^\ast c)(1+cc^\ast)^{-1}\,,
\end{equation*}
or equivalently localized for each $c\in \overline{\Upsilon}$ at $1+cc^\ast$.  
Thus $\Bc(\Upsilon)\simeq \Ac(\Upsilon)$.

To get the second isomorphism, we rewrite  the defining relations \eqref{Eq:Fiss1} of $\Bc(\Upsilon)$ as 
\begin{equation} \label{Eq:lem-Alg-Iso-1}
\prod_{\substack{c\in C \text{ s.t.}\\ s=t(c)\text{ or } s=h(c)}} \gamma_{c,s} 
=   q_s e_s\,,
\end{equation}
where the product is taken with respect to the ordering at the vertex $s$. 
Let us extend the operation $(-)^\ast: \Upsilon\to \overline{\Upsilon}\setminus\Upsilon$, $c\mapsto c^\ast$, to an involution on $\overline{\Upsilon}$. 
Introduce $\eta:\overline{\Upsilon}\to \{\pm1\}$ by $\eta(c)=-1$ if $c\in \Upsilon$, and $\eta(c)=+1$ if $c\in \overline{\Upsilon}\setminus\Upsilon$. 
From the discussion made above to get the isomorphism $\Bc(\Upsilon)\simeq \Ac(\Upsilon)$, \eqref{Eq:lem-Alg-Iso-1} can be written as   
\begin{equation*} 
\prod_{\substack{c\in \overline{\Upsilon}\text{ s.t.}\\ s=t(c)\text{ or } s=h(c)}} (e_{h(c)}+cc^\ast)^{\eta(c)} =   q_s e_s\,,
\end{equation*}
or, after gathering all these relations together, 
\begin{equation}  \label{Eq:MQV3} 
\prod_{c\in \overline{\Upsilon}} (1+cc^\ast)^{\eta(c)} = \sum_{s\in I}  q_s e_s\,.
\end{equation}
This is the defining relation\footnote{What differs between \eqref{Eq:MQV3} and \eqref{Eq:MQV1} is that the exponents are opposite, i.e. $\eta(c)=-\epsilon(c)$.} of $\Lambda^q(\Upsilon^{op})$. Here, the multiplicative preprojective algebra  $\Lambda^q(\Upsilon^{op})$ is constructed from the (double $\overline{\Upsilon}^{op}$ of the) quiver $\Upsilon^{op}=\overline{\Upsilon} \setminus \Upsilon$ obtained by changing the direction of all the arrows in $\Upsilon$. 
But those algebras are independent of the direction of the arrows in the initial quiver by \cite[Theorem 1.4]{CBShaw}, hence 
\begin{equation*}
 \Fc^q(\Upsilon)\simeq \Lambda^q(\Upsilon^{op}) \simeq \Lambda^q(\Upsilon)\,,
\end{equation*}
as desired.
\end{proof}

\subsection{Noncommutative Poisson geometry of fission algebras: the conjecture} 
\label{sec:the conjecture}

Recall from  \cite{B15} that, for any colored quiver $\Upsilon$, 
the corresponding colored multiplicative quiver varieties $\mathcal{M}(\Upsilon,q,d)$ are  parametrized by semi-simple modules of the fission algebras $\Fc^q(\Upsilon)$. Indeed, they are GIT quotients of the spaces $\Rep(\Fc^q(\Upsilon),d)$. Furthermore, these varieties are obtained by quasi-Hamiltonian reduction from $\Rep(\Bc(\Upsilon),d)$, since the subspace $\Rep(\Fc^q(\Upsilon),d)$ corresponds to fixing the value of a multiplicative moment map. This suggests that there could exist a Hamiltonian double quasi-Poisson algebra structure on $\Bc(\Upsilon)$, for which 
$\prod_c \gamma_{c,s}$ is the component of the multiplicative moment map supported at the vertex $s\in I$.

\begin{conj} \label{Conj:Main}
For each colored quiver $\Upsilon$, the Boalch algebra $\Bc(\Upsilon)$ is endowed with a double quasi-Poisson bracket $\dgal{-,-}$ for which 
\begin{equation}
 \Phi=\sum_{s\in I}\Phi_s\,, \quad \Phi_s= \prod_{\{c\in C \mid  s\in I_c\}} \gamma_{c,s} \in e_s \Bc(\Upsilon) e_s\,,
 \label{eq:moment-map-state,ment-conj}
\end{equation}
 is a multiplicative moment map. 
 In other words, the triple $\big(\Bc(\Upsilon),\lr{-,-},\Phi\big)$ is a Hamiltonian double quasi-Poisson algebra.
\end{conj}

\begin{lem}  \label{Lem:MonochromaticConjecture}
 If Conjecture \ref{Conj:Main} holds on monochromatic quivers, then it holds for any colored quiver. 
\end{lem}
\begin{proof}
 Assume that $\Upsilon$ is an arbitrary colored quiver with color set $C$. Then $\Upsilon$ can be obtained from the monochromatic subquivers $\Upsilon_c=\rc^{-1}(c)$, $c\in C$, by suitably identifying their common vertices, see for example Figure \ref{fig:M1}. Recall the process of fusion described in \cite[\S2.5, 5.3]{VdB1}, which allows to identify idempotents in an algebra. We check that $\Bc(\Upsilon)$ can be obtained from the algebras $\Bc(\Upsilon_c)$ by successively performing fusion of their idempotents corresponding to the identification of the vertices in the quivers $\Upsilon_c$. 
 
 \vspace{-0.5cm} 
 
\begin{center}
 \rule{12cm}{0.1pt}
\end{center}

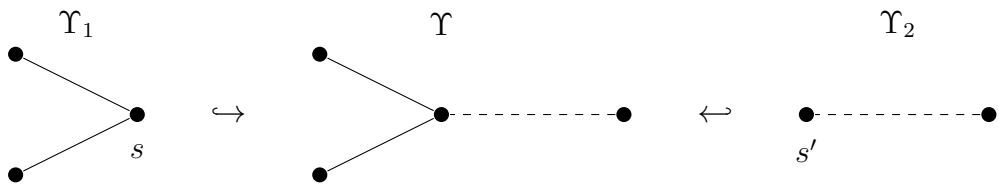
\begin{figure}[h]
\centering
   \begin{tikzpicture}[scale=0.8]
  \node[circle,thick,fill=black,inner sep=2pt]  (v0) at (0,0) {};
  \node[circle,thick,fill=black,inner sep=2pt]  (vu) at (-2,1) {};
   \node[circle,thick,fill=black,inner sep=2pt] (vd) at (-2,-1) {};
   \node[circle,thick,fill=black,inner sep=2pt] (vr) at (3,0) {};
\path (vu) edge  (v0) ;
\path (vd) edge  (v0) ;
\path (vu) edge  (vd) ;
\path[dashed] (vr) edge  (v0) ;
\node[circle,thick,fill=black,inner sep=2pt]  (vL0) at (-5,0) {};
\node[circle,thick,fill=black,inner sep=2pt]  (vR0) at (6,0) {};
\node[circle,thick,fill=black,inner sep=2pt]  (vLu) at (-7,1) {};
\node[circle,thick,fill=black,inner sep=2pt] (vLd) at (-7,-1) {};
\node[circle,thick,fill=black,inner sep=2pt] (vR) at (9,0) {};
\node (w0) at (-5,-0.6) {$s$};
\node (w0b) at (6,-0.6) {$s'$};
\path (vLu) edge  (vL0) ;
\path (vLd) edge  (vL0) ;
\path (vLd) edge  (vLu) ;
\path[dashed] (vR) edge  (vR0) ;
\node (la) at (-3.5,0) {$\hookrightarrow$};
\node (ra) at (4.5,0) {$\hookleftarrow$};
\node (n1) at (-6,1.5) {$\Upsilon_1$};
\node (nmid) at (0,1.5) {$\Upsilon$};
\node (n2) at (7.5,1.5) {$\Upsilon_2$};
   \end{tikzpicture}
\caption{The quiver $\Upsilon$ contains two monochromatic pieces, whose colors are represented by the solid and dashed styles of the lines. The quiver $\Upsilon$  can be obtained from $\Upsilon_1$ and $\Upsilon_2$ by identifying the vertices $s$ and $s'$.} 
\label{fig:M1}
\end{figure}
\vspace{-0.5cm}

\begin{center}
 \rule{12cm}{0.1pt}
\end{center}

The process of fusion yields a Hamiltonian double quasi-Poisson structure if the original algebras are endowed with one by \cite[Theorems 5.3.1 and 5.3.2]{VdB1} (see \cite[Theorems 2.14 and 2.15]{F19} in full generality). Thus, we get a Hamiltonian double quasi-Poisson algebra structure on $\Bc(\Upsilon)$ if there is one on each $\Bc(\Upsilon_c)$, which is our assumption. In particular, the component of the multiplicative moment map in $e_s \Bc(\Upsilon)e_s$, $s\in I$, is a product of the monochromatic moment maps $\gamma_{c,s}$ for each $\{c\in C \mid  I_c\ni s\}$. The latter product depends on the order in which the fusion of the idempotents is performed (see again the work of Van den Bergh \cite{VdB1}). Hence, we can take the order in which we perform fusion to be such that it coincides with the order fixed at each vertex of $\Upsilon$, which finishes the proof.  
\end{proof}

\begin{rem}
 Note that up to isomorphism, the order in which the fusion operation is performed is irrelevant \cite[Theorem 4.10]{F20}. Thus, up to isomorphism, the Hamiltonian double quasi-Poisson algebra structure on $\Bc(\Upsilon)$ only depends on the Hamiltonian double quasi-Poisson algebra structure on each $\Bc(\Upsilon_c)$, which are the Boalch algebras associated with the monochromatic subquivers.
\end{rem}

\begin{rem}
We expect that, as a consequence of Conjecture \ref{Conj:Main}, we get noncommutative versions of Corollary 6.6 and Theorem 6.7 of \cite{B15}. Regarding the first result, it means that the Hamiltonian double quasi-Poisson algebra structure should only depend on $\Upsilon$ seen as a graph \emph{without ordering}, up to isomorphism. Regarding the second result, it means that for any bipartite colored graph $\Upsilon(1,n)$ associated with the partition $(1,n)$ (it has one  color and is called star-shaped quiver), the Hamiltonian double quasi-Poisson algebra structure is conjecturally isomorphic to the one associated with the same graph where each arrow is assigned a different color. 

\end{rem}

To close this section, let us recall from \ref{sec:KR-principle} that a Hamiltonian double quasi-Poisson algebra structure induces a Hamiltonian quasi-Poisson structure  on representation schemes. Therefore, if the conjecture holds and the representation scheme associated with $\Bc(\Upsilon)$ is not empty, it carries a quasi-Poisson bracket and a group-valued moment map. 
\begin{lem}\label{lem:rep-non-empty}
Let $\Upsilon$ be a colored quiver with vertex set $I$, and set $B=\oplus_{s\in I} \kk e_s$. 
Given a dimension vector $d=(d_s)\in \Z_{\geq 0}^I$, define the representation scheme (relative to $B$) $\Rep(\Bc(\Upsilon),d)$ as in \ref{sec:KR-principle}.
 Then $\Rep(\Bc(\Upsilon),d)$ is not empty, and we have 
\begin{equation} \label{Eq:dim-rep-BoalchAlgebra}
 \dim_\kk\Rep(\Bc(\Upsilon),d) = 2\sum_{a\in \Upsilon} d_{t(a)}d_{h(a)}\,.
\end{equation}
\end{lem}
\begin{proof}
Set $N=\sum_{s\in I} d_s$. 
For the first part, note that the assignment $\rho_{\textrm{triv}}$ given by 
\begin{align*}
 \rho_{\textrm{triv}}(w_{c,ij})=0_N,\,\,\, \rho_{\textrm{triv}}(v_{c,ij})=0_N, \,\,\, \rho_{\textrm{triv}}(1-e_i+\gamma_{c,i})=\Id_N, \quad i,j\in I_c,\,\, c\in C\,,
\end{align*}
completely determines a representation $\rho_{\textrm{triv}}:\Bc(\Upsilon)\to \End(\CC^N)$ relative to $B$ by definition of $\Bc(\Upsilon)$. 

For the second part, we have seen in Lemma \ref{lem:localisation-Bgamma} that $\Bc(\Upsilon)$ can be obtained by localization of $\kk \overline{\Upsilon}$. Thus, if  $\Rep(\Bc(\Upsilon),d)$ is not empty, it has the dimension of $\Rep(\kk \overline{\Upsilon},d)$ which is given  by the right-hand side of \eqref{Eq:dim-rep-BoalchAlgebra}.
\end{proof}


\section{Towards the Conjecture: the monochromatic interval and triangle} \label{sec:Towards}

\subsection{The monochromatic interval \texorpdfstring{$\mathcal{I}$}{}}
\label{sec:the-11-monochromatic-case}

We consider the simplest colored quiver $\mathcal{I}$, which consists of one arrow $v_{12}:2\to 1$.
By definition, the algebra $\Bc(\mathcal{I})$ is generated by the symbols 
\begin{equation} \label{eq:generators-B-11}
 e_1,\,\, e_2,\,\, v_{12},\,\, v_{21},\,\, w_{12},\,\, w_{21},\,\, \gamma_1^{\pm 1},\,\, \gamma_2^{\pm 1},
\end{equation}
subject to the idempotent decomposition $e_1+e_2=1$, $e_1 e_2=0=e_2 e_1$, and 
\begin{equation*}
 \begin{aligned}
  v_{21}&=e_2 v_{21} e_1,\,\, & v_{12} &= e_1 v_{12} e_2,\, \, & w_{12}&=e_1 w_{12}e_2,\qquad  w_{21}=e_2 w_{21} e_1, 
  \\
  \gamma_1^{\pm 1}&=e_1 \gamma_1^{\pm 1} e_1,\,\, & \gamma_2^{\pm 1}&= e_2\gamma_2^{\pm 1} e_2, \, \,  & \gamma_j \gamma_j^{-1}&=e_j=\gamma_j^{-1}\gamma_j\,,
 \end{aligned}
\end{equation*}
together with the relation 
\begin{equation*}
\label{eq:relation-B-11-case}
 (1+v_{21})(1+v_{12})=(1+w_{12})(\gamma_1+\gamma_2)(1+w_{21})\,.
\end{equation*}
As part of Lemma \ref{L:AlgIso}, we have seen that\footnote{The vertices $1,2$ correspond to $h,t$ while $v_{12},v_{21}$ correspond to $c,c^\ast$ respectively.} 
this relation  amounts to 
\begin{equation} \label{Eq:RelBA}
\begin{aligned}
 &\gamma_2=e_2+v_{21}v_{12},\quad 
 w_{12}=v_{12} \gamma_2^{-1},\quad
 w_{21}=\gamma_2^{-1} v_{21}\,, 
 \\
 &\gamma_1=e_1 -w_{12}\gamma_2w_{21}=e_1-v_{12}(e_2+v_{21}v_{12})^{-1}v_{21}=(e_1+v_{12}v_{21})^{-1}\,.
\end{aligned}
\end{equation}
In particular, $e_1,e_2,v_{12},v_{21}$ (and the inverses $\gamma_j^{-1}$) generate the Boalch algebra  $\Bc(\mathcal{I})$. 

Meanwhile, as we pointed out in Theorem \ref{tm:VdB-mult-preproj-quasi-Hamilt}\,(i), we know from Van den Bergh's theory \cite{VdB1} (which is written for the opposite convention ---\,see Remark \ref{rem:convention-on-arrows}) that $\Ac(\mathcal{I})$ is a Hamiltonian double quasi-Poisson algebra for the double quasi-Poisson bracket 
\begin{equation}
\begin{aligned} \label{Eq:VdBfission}
  &\dgal{v_{12},v_{12}}=0\,, \quad \dgal{v_{21},v_{21}}=0\,, \\
 &\dgal{v_{21},v_{12}}=e_1\otimes e_2 + \frac12 \Big(v_{12}v_{21} \otimes e_2 + e_1 \otimes v_{21}v_{12}\Big)\,, \\
 &\dgal{v_{12},v_{21}}=-e_2\otimes e_1 - \frac12 \Big(e_1 \otimes v_{12}v_{21} + v_{21}v_{12} \otimes e_1\Big)\,.
\end{aligned}
\end{equation}
Note that the last equality is equivalent to the third one using the cyclic antisymmetry  \eqref{Eq:cycanti}. It is a simple exercise using \eqref{Eq:RelBA} to show that 
$\Phi:=\gamma_1 + \gamma_2$ is a multiplicative moment map, since $\gamma_1$ and $\gamma_2$ satisfy \eqref{Phim} with $s=1,2$ respectively. 
Indeed, it suffices to prove the claim on $v_{12},v_{21}$, and we can compute from the above double brackets that 
\begin{equation*}
\begin{aligned}
  \dgal{\gamma_2,v_{12}}&=\frac12\Big(v_{12}\gamma_2 \otimes e_2 + v_{12}\otimes \gamma_2\Big)\,, &
  \dgal{\gamma_2,v_{21}}&=-\frac12 \Big(e_2 \otimes \gamma_2 v_{21} + \gamma_2 \otimes v_{21}\Big)\,, 
  \\
 \dgal{\gamma_1,v_{12}}&=-\frac12 \Big(e_1 \otimes \gamma_1 v_{12} + \gamma_1 \otimes v_{12}\Big)\,, &  \dgal{\gamma_1,v_{21}}&=\frac12\Big(v_{21}\gamma_1 \otimes e_1 + v_{21}\otimes \gamma_1\Big)\,.
\end{aligned}
\end{equation*}
Gathering these observations together, we obtain the first instance of Conjecture \ref{Conj:Main}:

\begin{prop}  \label{Pr:OneArrow}
Let $\mathcal{I}$ be the monochromatic graph with two vertices and one edge, with partition of the set of vertices $\{1\}\sqcup\{2\}$, and $B=\kk e_1\oplus\kk e_2$. We construct the associated Boalch algebra $\Bc(\mathcal{I})$ as above. 
We define  a $B$-linear double bracket on $\Bc(\mathcal{I})$ from \eqref{Eq:VdBfission} by using the cyclic antisymmetry \eqref{Eq:cycanti} and the Leibniz rule \eqref{Eq:outder}.
Furthermore, consider the element
\begin{equation*}
\Phi:=\gamma_1+\gamma_2.
\end{equation*}
Then the triple $\big(\Bc(\mathcal{I}),\lr{-,-},\Phi\big)$ is a Hamiltonian double quasi-Poisson algebra.
\end{prop}

If we want to understand the double quasi-Poisson bracket in terms of all the generators \eqref{eq:generators-B-11} of $\Bc(\mathcal{I})$, it suffices to perform some elementary computations. In particular, we will need the following identities that follow from \eqref{Eq:RelBA}:
\begin{equation*}
 w_{12}\gamma_2 w_{21}=v_{12}w_{21}=w_{12}v_{21}=e_1-\gamma_1\,, \quad 
 v_{21}v_{12}=\gamma_2 w_{21}v_{12}=v_{21}w_{12}\gamma_2=\gamma_2-e_2\,.
\end{equation*}

\begin{lem} \label{Lemma:OneArrow}
 The Boalch algebra $\Bc(\mathcal{I})$ equipped with its double quasi-Poisson bracket from Proposition \ref{Pr:OneArrow} is such that 
   \begin{align*}
      &\dgal{v_{12},v_{12}}=0=\dgal{v_{21},v_{21}}\,, 
      \\
  &\dgal{v_{21},v_{12}}=e_1\otimes e_2 + \frac12 \Big(v_{12}v_{21} \otimes e_2 +  e_1 \otimes v_{21}v_{12}\Big)\,,
   \\
  &  \dgal{w_{21},v_{21}} = \frac12 (w_{21}\otimes v_{21} + v_{21} \otimes w_{21}) \,, 
  \\
   &\dgal{w_{12}, v_{12} } = -\frac12 (w_{12}\otimes v_{12} + v_{12} \otimes w_{12}) \,,
    \\
   &\dgal{w_{21}, v_{12} } = \frac12 e_1 \otimes \gamma_2^{-1} +\frac12 \gamma_1\otimes e_2 \,, 
   \\
   &\dgal{w_{12},v_{21}} = -\frac12 \big(\gamma_2^{-1}\otimes e_1  + e_2 \otimes \gamma_1\big)\,,
    \\
      &\dgal{w_{12},w_{12}}=0=\dgal{w_{21},w_{21}}\,, 
      \\
 & \dgal{w_{21},w_{12}}=  \gamma_1 \otimes \gamma_2^{-1} -\frac12 \Big(w_{12}w_{21} \otimes e_2 + e_1 \otimes w_{21}w_{12}\Big)\, .
   \end{align*}
 \normalsize
\end{lem}
\begin{proof}
 The first three identities are just \eqref{Eq:VdBfission}. The next two are obtained by a direct computation using those identities and the moment map property for $\gamma_2$. For example, using the decomposition 
\begin{equation*}
 \dgal{w_{21},v_{21}} =\gamma_2^{-1}\ast \dgal{v_{21},v_{21}} -  \gamma_2^{-1} \ast \dgal{\gamma_2,v_{21}}\ast w_{21}\,.
\end{equation*}
 We then find $\dgal{w_{12},w_{12}}$ and $\dgal{w_{21},w_{21}}$ easily. 
 Next, we note that 
 \begin{equation*}
  \begin{aligned}
\dgal{w_{21}, v_{12} } =& \dgal{\gamma_2^{-1}v_{21},v_{12}} \\
=& e_1 \otimes \gamma_2^{-1}+\frac12 \Big(e_1 \otimes \underbrace{w_{21}v_{12}}_{e_2-\gamma_2^{-1}}-\underbrace{v_{12}w_{21}}_{e_1-\gamma_1}\otimes e_2\Big) 
=\frac12 e_1 \otimes \gamma_2^{-1} +\frac12 \gamma_1 \otimes e_2\,.
  \end{aligned}
 \end{equation*}
We can also find $\dgal{w_{12},v_{21}} $ in this way. Finally, we can get from these expressions
 \begin{equation*}
  \begin{aligned}
\dgal{w_{21},w_{12}}=& \dgal{w_{21},v_{12}}\gamma_2^{-1} - w_{12} \dgal{w_{21},\gamma_2}\gamma_2^{-1} \\
=& \frac12 \Big( e_1 \otimes \gamma_2^{-2} + \gamma_1 \otimes \gamma_2^{-1}\Big) -\frac12 \Big(w_{12}\gamma_2 w_{21}\otimes \gamma_2^{-1} + w_{12}w_{21}\otimes e_2\Big)\\
=&\frac12 e_1 \otimes \Big[\gamma_2^{-1}-w_{21}w_{12}\Big] + \frac12 \gamma_1 \otimes \gamma_2^{-1}  - \frac12 \Big[e_1-\gamma_1\Big]\otimes \gamma_2^{-1}  -\frac12 w_{12} w_{21}\otimes e_2\\
=&\gamma_1 \otimes \gamma_2^{-1} -\frac12 \Big(w_{12}w_{21} \otimes e_2 + e_1 \otimes w_{21}w_{12}\Big)\,,
  \end{aligned}
 \end{equation*}
 which is the last identity. 
\end{proof}

\begin{rem}
Combining Lemma \ref{Lem:MonochromaticConjecture} and Proposition \ref{Pr:OneArrow}, we get that if $\Upsilon$ is an arbitrary colored graph with the tautological coloring (that is, its edges have all different colors), then $\Bc(\Upsilon)$ carries a Hamiltonian double quasi-Poisson structure. By construction, it is obtained by fusion of the Hamiltonian double quasi-Poisson algebras associated with the  disjoint arrows of $\Upsilon$, hence this is equivalent to Van den Bergh's result \cite[Theorem 6.7.1]{VdB1}. 

Using the Kontsevich--Rosenberg principle, we then get a quasi-Poisson bracket and a group-valued moment map on representation schemes. This induces the Poisson structure on multiplicative quiver varieties (in the original sense of \cite{CBShaw,Ya}) uncovered by Van den Bergh \cite[Theorem 1.1]{VdB1}. 
\end{rem}

 \subsection{The monochromatic triangle \texorpdfstring{$\Delta$}{}}
 \label{sec:the-monochromatic-111}

 Consider  the monochromatic triangle $\Delta$ whose set of vertices $\{1,2,3\}$ has the partition $\{1\}\sqcup\{2\}\sqcup\{3\}$.
Following Boalch's convention as in \ref{ss:Gr}, we take the arrows such that $v_{23}:3\to 2$, $v_{13}:3\to 1$, $v_{12}:2\to 1$. Next, as explained in \ref{sec:fission-algebras} and depicted in Figure \ref{fig:Quiver111},
  \begin{enumerate}
  \item [\textup{(i)}]
We add the opposites $v_{ij}:j\to i$ for any $i\neq j$, giving the double quiver $\overline{\Delta}$.  
\item [\textup{(ii)}]
 We add the elements $w_{ij}:j\to i$ for all $j\neq i$ and $\gamma_i:i\to i$ for all $i$ to form $\widetilde \Delta$. 
 \end{enumerate}
This means that we can see the Boalch algebra $\Bc(\Delta)$ as being generated by the symbols
 \begin{align*}
e_1,\quad e_2,\quad e_3, \quad &v_{12}, \quad v_{21}, \quad v_{13},\quad v_{31}, \quad v_{23},\quad v_{32},
\\
\gamma^{\pm 1}_1,\quad \gamma^{\pm 1}_2,\quad \gamma^{\pm 1}_3, \quad & w_{12},\quad w_{21}, \quad w_{13},\quad w_{31}, \quad w_{23},\quad w_{32},
\end{align*}
 subject to the relations induced by the idempotent decomposition $1=e_1+e_2+e_3$,
\begin{align*}
 e_ie_k=\delta_{ik}e_i\,, \quad v_{ij}=e_i v_{ij}e_j,\,\,  w_{ij}=e_i w_{ij}e_j,\,\,  \gamma_{i}=e_i \gamma_{i}e_i,\,\; \,i,j,k\in \{1,2,3\},\, i\neq j\,, 
\end{align*}
 the invertibility conditions $\gamma_i\gamma_i^{-1}=e_i=\gamma_i^{-1}\gamma_i$ for $i\in \{1,2,3\}$,
 as well as the following relation obtained from \eqref{Eq:RelB}:
 \begin{equation}
\begin{aligned}
  (1+v_{21}+v_{31}+v_{32})&(1+v_{12}+v_{13}+v_{23}) \\
  =&(1+w_{12}+w_{13}+w_{23})(\gamma_1+\gamma_2+\gamma_3)(1+w_{21}+w_{31}+w_{32}).
 \label{eq:relation-Boalch-algebra-big}
\end{aligned}
 \end{equation}


Decomposing \eqref{eq:relation-Boalch-algebra-big} with respect to the idempotents, it is equivalent to the following nine identities in $\Bc(\Delta)$:
 \begin{subequations} \label{eq:Boalch-eq-111}
  \begin{align}
 e_1&=\gamma_1+w_{12}\gamma_2 w_{21} + w_{13}\gamma_3 w_{31}\, , \label{eq:Boalch-eq-111.a}
 \\
 v_{12}&=w_{12}\gamma_2 + w_{13} \gamma_3 w_{32}\, , \label{eq:Boalch-eq-111.b}
 \\
 v_{13}&=w_{13}\gamma_3\, ,  \label{eq:Boalch-eq-111.c}
 \\
v_{21}&=\gamma_2w_{21} + w_{23}\gamma_3w_{31}\, , \label{eq:Boalch-eq-111.d}
\\
e_2+v_{21}v_{12}&=\gamma_2+w_{23}\gamma_3w_{32}\, , \label{eq:Boalch-eq-111.e}
\\
 v_{21}v_{13}+v_{23}&=w_{23}\gamma_3\, , \label{eq:Boalch-eq-111.f}
\\
v_{31}&=\gamma_3w_{31}\, , \label{eq:Boalch-eq-111.g}
\\
 v_{32}+v_{31}v_{12}&=\gamma_3 w_{32}\, , \label{eq:Boalch-eq-111.h}
 \\
 e_3+v_{31}v_{13}+v_{32}v_{23}&=\gamma_3\, . \label{eq:Boalch-eq-111.i}
  \end{align}
 \end{subequations}
 
 \begin{rem}
Note that by Lemma \ref{lem:localisation-Bgamma}, we can conclude that $e_i$, $v_{ij}$, $v_{ji}$ and $\gamma^{-1}_i$ for $1\leq i<j\leq 3$ generate the Boalch algebra $\Bc(\Delta)$. Furthermore, we can observe from \eqref{eq:Boalch-eq-111} that the chain of localizations $\kk\overline{\Delta}\to \Bc(\Delta)$  is obtained by iteratively introducing and then localizing at the following elements: 
\begin{align*}
\gamma_3&=e_3+v_{31}v_{13}+v_{32}v_{23}\,, \\
\gamma_2&=e_2+v_{21}v_{12}-(v_{21}v_{13}+v_{23})\gamma_3^{-1}(v_{32}+v_{31}v_{12})\,, \\
\gamma_1&=e_1-v_{13}\gamma^{-1}_3 v_{31} 
-\big[v_{12}-v_{13}\gamma_3^{-1}(v_{32}+v_{31}v_{12})\big]\gamma_2^{-1} \big[v_{21}- (v_{21}v_{13}+v_{23})\gamma_3^{-1} v_{31}\big]\,.
\end{align*}
\label{rem:arrows-111-generators}
 \end{rem}

 \vspace{-0.5cm}
 
\begin{center}
 \rule{12cm}{0.1pt}
\end{center}

\begin{figure}[h]
\centering
   \begin{tikzpicture}[scale=0.8]
  \node[circle,thick,fill=black,inner sep=2pt]  (v1le) at (-12,-2) {};
  \node[circle,thick,fill=black,inner sep=2pt]  (v2le) at (-10,1) {};
   \node[circle,thick,fill=black,inner sep=2pt] (v3le) at (-8,-2) {};
     \node  (v1n) at (-12.3,-2.2) {$1$};
  \node  (v2n) at (-10,1.4) {$2$};
   \node (v3n) at (-7.7,-2.2) {$3$};
\node (NameLe) at (-10,-3.5) {\underline{Quiver $\Delta$}};
\draw[->,>=triangle 45] (v2le) -- node[left,font=\small]{$v_{12}$}   (v1le) ;
\draw[->,>=triangle 45] (v3le) -- node[above,font=\small]{$v_{13}$}  (v1le) ;
\draw[->,>=triangle 45] (v3le) -- node[right,font=\small]{$v_{23}$}  (v2le) ;
  \node[circle,thick,fill=black,inner sep=2pt]  (v1m) at (-6,-2) {};
  \node[circle,thick,fill=black,inner sep=2pt]  (v2m) at (-4,1) {};
   \node[circle,thick,fill=black,inner sep=2pt] (v3m) at (-2,-2) {};
     \node  (v1mn) at (-6.3,-2.2) {$1$};
  \node  (v2mn) at (-4,1.4) {$2$};
   \node (v3mn) at (-1.7,-2.2) {$3$};
\node (NameLe) at (-4,-3.5) {\underline{Quiver $\overline{\Delta}$}};
\path[->,>=latex] (v2m) edge [bend right=15]  node[left,font=\small]{$v_{12}$}   (v1m) ;
\path[->,>=latex] (v3m)  edge [bend right=6] node[above,font=\small]{$v_{13}$}  (v1m) ;
\path[->,>=latex] (v3m)  edge [bend right=15] node[right,font=\small]{$v_{23}$}  (v2m) ;
\path[->,>=latex] (v1m) edge [bend right=6]  node[right,font=\small]{$v_{21}$}   (v2m) ;
\path[->,>=latex] (v1m)  edge [bend right=15] node[below,font=\small]{$v_{31}$}  (v3m) ;
\path[->,>=latex] (v2m)  edge [bend right=6] node[left,font=\small]{$v_{32}$}  (v3m) ;
  \node[circle,thick,fill=black,inner sep=3pt]  (v1) at (0.7,-2.2) {};
  \node[circle,thick,fill=black,inner sep=3pt]  (v2) at (3,1.1) {};
   \node[circle,thick,fill=black,inner sep=3pt] (v3) at (5.3,-2.2) {};
     \node  (v1r) at (0.35,-2.35) {$1$};
  \node  (v2r) at (3,1.5) {$2$};
   \node (v3r) at (5.65,-2.35) {$3$};
\node (NameLe) at (3,-3.5) {\underline{Quiver $\widetilde{\Delta}$}};
\path[->,>=latex] (v2) edge [bend right=10]   (v1) ;
\path[->,>=latex] (v3) edge [bend right=6]    (v1) ;
\path[->,>=latex] (v3) edge [bend right=10]    (v2) ;
\path[->,>=latex] (v1) edge [bend right=6]     (v2) ;
\path[->,>=latex] (v1) edge [bend right=10]    (v3) ;
\path[->,>=latex] (v2) edge [bend right=6]    (v3) ;
\path[->,dashed,>=latex] (v2) edge [bend left=20]   (v1) ;
\path[->,dashed,>=latex] (v3) edge [bend left=25]    (v1) ;
\path[->,dashed,>=latex] (v3) edge [bend left=20]    (v2) ;
\path[->,dashed,>=latex] (v1) edge [bend left=25]     (v2) ;
\path[->,dashed,>=latex] (v1) edge [bend left=20]    (v3) ;
\path[->,dashed,>=latex] (v2) edge [bend left=25]    (v3) ;
\draw[->,>=latex] (v1) to[out=130,in=250,looseness=15] (v1);
\draw[->,>=latex] (v2) to[out=30,in=150,looseness=15] (v2);
\draw[->,>=latex] (v3) to[out=50,in=290,looseness=15] (v3);
   \end{tikzpicture}
 \caption{The quivers $\Delta$, $\overline{\Delta}$, and $\widetilde{\Delta}$ used to introduce the algebra $\Bc(\Delta)$. In $\widetilde{\Delta}$, 
a loop based at the vertex $i$ represents $\gamma_i$, while a plain (resp. dashed) arrow from the vertex $i$ to the vertex $j$ represents $v_{ji}$ (resp. $w_{ji}$).} \label{fig:Quiver111}
\end{figure}
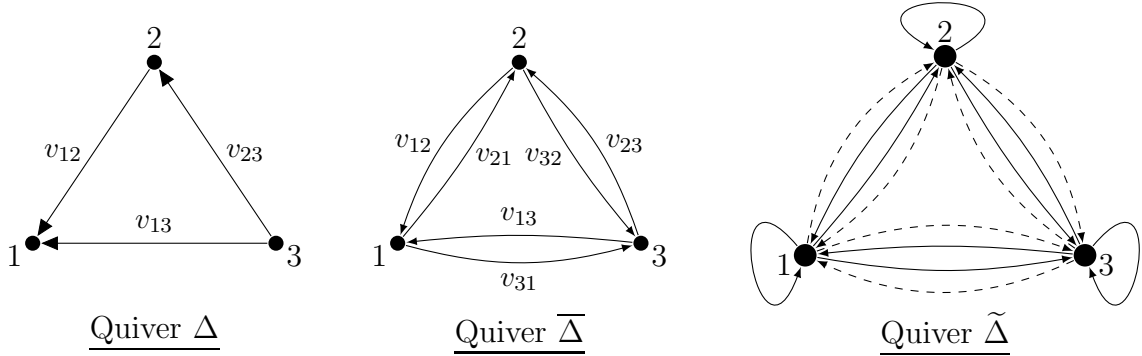

\vspace{-0.5cm}

\begin{center}
 \rule{12cm}{0.1pt}
\end{center}

Let $B=\kk e_1\oplus\kk e_2\oplus\kk e_3$.
Building on Remark \ref{rem:arrows-111-generators}, we can introduce a $B$-linear double bracket $\lr{-,-}$ on $\Bc (\Delta)$ by first specifying its expression on the arrows $v_{ij}$, for $i,j\in\{1,2,3\}$ and $i\neq j$. We start by setting 
\begin{small}
\begin{equation}
\label{double-bracket-111}
\begin{split}
\lr{v_{ij},v_{ij}}&=  \; \;\, 0\, ,			  
\\
\lr{v_{12},v_{13}}&= \; \;\,  \frac{1}{2}    v_{12}\otimes v_{13}\, ,
\\
\lr{v_{12},v_{32}}&=\; \;\,	 \frac{1}{2}    v_{32}\otimes v_{12}\, ,
\\
\lr{v_{21},v_{31}}&= -\frac{1}{2}    v_{31}\otimes v_{21}\, ,
\\
\lr{v_{21},v_{23}}&= -\frac{1}{2}    v_{21}\otimes v_{23}\, ,
\\
\lr{v_{13},v_{23}}&= \; \;\, \frac{1}{2}    v_{23}\otimes v_{13}\, ,
\\
\lr{v_{13},v_{32}}&=-\frac{1}{2}    e_3\otimes v_{13}v_{32}\, ,
\\
\lr{v_{31},v_{23}}&= \; \;\, \frac{1}{2}    v_{23}v_{31}\otimes e_3\, ,
\\
\lr{v_{31},v_{32}}&=-\frac{1}{2}    v_{31}\otimes v_{32}\, ,
\end{split}
\quad 
\begin{split}
\lr{v_{12},v_{21}}&=  -e_2\otimes e_1- \frac{1}{2}\Big(e_2\otimes v_{12}v_{21}+   v_{21}v_{12}\otimes e_1\Big)\, ,
\\
\lr{v_{12},v_{31}}&=  -\frac{1}{2}    v_{31}v_{12}\otimes e_1-v_{32}\otimes e_1\, ,
\\
\lr{v_{12},v_{23}}&=  - \frac{1}{2}    e_2\otimes v_{12}v_{23}+e_2\otimes v_{13}\, ,
\\
\lr{v_{21},v_{13}}&= \; \;\, \frac{1}{2}    e_1\otimes v_{21}v_{13}+e_1\otimes v_{23}\, ,
\\
\lr{v_{21},v_{32}}&= \; \;\, \frac{1}{2}    v_{32}v_{21}\otimes e_2-v_{31}\otimes e_2\, ,
\\
\lr{v_{13},v_{31}}&= -e_3\otimes e_1-v_{32}v_{23}\otimes e_1
\\
&\qquad -\frac{1}{2}\Big(e_3\otimes v_{13}v_{31}+  v_{31}v_{13}\otimes e_1\Big),
\\
\lr{v_{23},v_{32}}&=-e_3\otimes e_2-\frac{1}{2} \Big(   e_3\otimes v_{23}v_{32}+  v_{32}v_{23}\otimes e_3\Big),
\end{split}
\end{equation}
\end{small}%
so that using the cyclic antisymmetry \eqref{Eq:cycanti} and the identities \eqref{double-bracket-111}, we can define $\lr{-,-}$ on all the remaining pairs of arrows $v_{ij}$, which we can then extend to $\kk \overline{\Delta}$ by the Leibniz rules \eqref{Eq:outder}--\eqref{Eq:inder}.  
Note that \eqref{double-bracket-111} descends to the Boalch algebra $\Bc(\Delta)$ by Lemma \ref{lem:localisation-Bgamma}. 
Therefore, gathering the identities \eqref{eq:Boalch-eq-111} and  \eqref{Eq:cycanti}--\eqref{Eq:inder}, we can obtain the rest of the  expressions for the double bracket $\lr{-,-}$ when evaluated on the other elements of $\Bc(\Delta)$.
For the reader's convenience, we collect them in Appendix \ref{sec:the-complete-list}.

\begin{thm}
\label{thm:triangle}
Let $\Delta$ be the monochromatic triangle, with partition of the set of vertices $\{1\}\sqcup\{2\}\sqcup\{3\}$, and $B=\kk e_1\oplus\kk e_2\oplus\kk e_3$. We construct the associated Boalch algebra $\Bc(\Delta)$ as above, and define a $B$-linear double bracket on $\Bc(\Delta)$ from \eqref{double-bracket-111} by using 
the cyclic antisymmetry \eqref{Eq:cycanti} and the Leibniz rule \eqref{Eq:outder}. 
Furthermore, we consider the element
\begin{equation}
\Phi:=\gamma_1+\gamma_2+\gamma_3.
\label{eq:moment-map-thm-111}
\end{equation}
Then the triple $\big(\Bc(\Delta),\lr{-,-},\Phi\big)$ is a Hamiltonian double quasi-Poisson algebra.
\end{thm}

\begin{proof}
We will prove that the double bracket \label{eq:double-bracket-111} is quasi-Poisson as part of a more general result, as described in Section \ref{sec:proof-quasi-Poisson-property}. To prove that \eqref{eq:moment-map-thm-111} is a moment map, we need to check that the multiplicative moment map condition~\eqref{Phim} holds by using~\eqref{eq:Boalch-eq-111}. Moreover, using the Leibniz rule \eqref{Eq:outder}, it is enough to show that~\eqref{Phim} is satisfied for generators of $\Bc(\Delta)$, which we take to be the arrows $v_{ij}$ by Lemma \ref{lem:localisation-Bgamma}.

Firstly, we shall prove that~\eqref{Phim} is satisfied for $v_{12}$.
By~\eqref{eq:Boalch-eq-111.i} and \eqref{Eq:inder}, 
\begin{align*}
\lr{\gamma_3,v_{12}}&=\lr{\big(e_3+v_{31}v_{13}+v_{32}v_{23}\big),v_{12}}
\\
&=\lr{v_{31},v_{12}}*v_{13}+v_{31}*\lr{v_{13},v_{12}}+\lr{v_{32},v_{12}}*v_{23}+v_{32}*\lr{v_{23},v_{12}}
\\
&=\frac{1}{2}v_{13}\otimes v_{31}v_{12}+v_{13}\otimes v_{32}-\frac{1}{2}v_{13}\otimes v_{31}v_{12}
\\
&\quad -\frac{1}{2}v_{12}v_{23}\otimes v_{32}+\frac{1}{2}v_{12}v_{23}\otimes v_{32}-v_{13}\otimes v_{32}=0.
\end{align*}
This agrees with \eqref{Phim} since $e_3 v_{12}=0=v_{12}e_3$. 
Similarly, using that  $\lr{v_{12},v_{12}}=0$ and $\lr{\gamma_3,v_{12}}=0$, by \eqref{eq:Boalch-eq-111.e} we have
\begin{align*}
\lr{\gamma_2,v_{12}}&=\lr{\big(e_2+v_{21}v_{12}-w_{23}\gamma_3w_{32}\big),v_{12}}
\\
&=\lr{v_{21},v_{12}}*v_{12}-w_{23}\gamma_3*\lr{w_{32},v_{12}}-\lr{w_{23},v_{12}}*\gamma_3w_{32}
\\
&=v_{12}\otimes e_2+\frac{1}{2}v_{12}\otimes\big(v_{21}v_{12}-w_{23}\gamma_3w_{32}\big)+\frac{1}{2}v_{12}\big(v_{21}v_{12}-w_{23}\gamma_3w_{32}\big)\otimes e_2
\\
&=\frac{1}{2}\big(v_{12}\otimes \gamma_2+v_{12}\gamma_2\otimes e_2\big),
\end{align*}
which is \eqref{Phim}. Finally, by \eqref{eq:Boalch-eq-111.a},
\begin{align*}
\lr{\gamma_1,v_{12}}&=\lr{\big(e_1-w_{12}\gamma_2w_{21}-w_{13}\gamma_3w_{31}\big),v_{12}}
\\
&=-w_{12}\gamma_2*\lr{w_{21},v_{12}}-w_{12}*\lr{\gamma_2,v_{12}}*w_{21}-\lr{w_{12},v_{21}}*\gamma_2w_{21}
\\
&\quad -w_{13}\gamma_3*\lr{w_{31},v_{12}}-\lr{w_{13},v_{12}}*\gamma_3w_{31}
\\
&=-e_1\otimes \big(w_{12}\gamma_2+w_{13}\gamma_3w_{32}\big)+\frac{1}{2}e_1\otimes \big(w_{12}\gamma_2w_{21}+w_{13}\gamma_3w_{31}\big)v_{12}+
\\
&\quad+ \frac{1}{2}\big(w_{12}\gamma_2w_{21}+w_{13}\gamma_3w_{31}\big)\otimes v_{12}
\\
&=\frac{1}{2}\big(v_{12}\otimes \gamma_2+v_{12}\gamma_2\otimes e_2\big),
\end{align*}
where we used \eqref{eq:Boalch-eq-111.b} and \eqref{eq:Boalch-eq-111.e}, showing \eqref{Phim} applied to $v_{12}$ holds. Since the proof of \eqref{Phim} applied to $v_{21}$ is quite similar, we leave it to the reader.


Next, we focus on the generator $v_{13}$. Since $\lr{v_{13},v_{13}}=0$, by~\eqref{eq:Boalch-eq-111.i} and \eqref{Eq:inder},
\begin{align*}
 \lr{\gamma_3,v_{13}}&=\lr{v_{31},v_{13}}*v_{13}+v_{32}*\lr{v_{23},v_{13}}+\lr{v_{32},v_{13}}*v_{23}
 \\
 &=v_{13}\otimes e_3+\frac{1}{2}v_{13}\big(v_{31}v_{13}+v_{32}v_{23}\big)\otimes e_3
 +\frac{1}{2}v_{13}\otimes\big( v_{31}v_{13}+ v_{32}v_{23}\big)
 \\
 &=\frac{1}{2}\big(v_{13}\otimes\gamma_3+v_{13}\gamma_3\otimes e_3\big).
\end{align*}

Similarly, by~\eqref{eq:Boalch-eq-111.e} and~\eqref{eq:Boalch-eq-111.f},
\begin{align*}
 \lr{\gamma_2,v_{13}}&=\lr{v_{21},v_{13}}*v_{12}+v_{21}*\lr{v_{12},v_{13}}
\\
 &\quad -w_{23}\gamma_3*\lr{w_{32},v_{13}}-w_{23}*\lr{\gamma_3,v_{13}}*w_{32}-\lr{w_{23},v_{13}}*\gamma_3w_{32}
 \\
 &=v_{12}\otimes v_{21}v_{13}+v_{12}\otimes v_{23}-v_{12}\otimes w_{23}\gamma_3
 \\
&=v_{12}\otimes v_{21}v_{13}+v_{12}\otimes (w_{23}\gamma_3-v_{21}v_{13})-v_{12}\otimes w_{23}\gamma_3=0.
\end{align*}

Finally, we state the multiplicative moment map condition for $\gamma_1$ and $v_{13}$. Applying ~\eqref{eq:Boalch-eq-111.a} and \eqref{eq:Boalch-eq-111.c},
\begin{align*}
 \lr{\gamma_1,v_{13}}&=-w_{12}\gamma_2*\lr{w_{21},v_{13}}-\lr{w_{12},v_{13}}*\gamma_2w_{21}
 \\
 &\quad -w_{13}\gamma_3*\lr{w_{31},v_{13}}-w_{13}*\lr{\gamma_3,v_{13}}*w_{31}-\lr{w_{13},v_{13}}*\gamma_3w_{31}
 \\
 &=\frac{1}{2}e_1\otimes w_{12}\gamma_2w_{21}v_{13}+\frac{1}{2}w_{12}\gamma_2w_{21}\otimes v_{13}-e_1\otimes w_{13}\gamma_3
 \\
 &\quad + \frac{1}{2}e_1\otimes w_{13}\gamma_3w_{31}v_{13} +\frac{1}{2}w_{13}\gamma_3w_{31}\otimes v_{13}
 \\
 &=-e_1\otimes v_{13}+\frac{1}{2}e_1\otimes\big(w_{12}\gamma_2w_{21}+w_{13}\gamma_3w_{31}\big)v_{13}+\frac{1}{2}\big(w_{12}\gamma_2w_{21}+w_{13}\gamma_3w_{31}\big)\otimes v_{13}
 \\
 &=-\frac{1}{2}\big(e_1\otimes \gamma_1v_{13}+\gamma_1\otimes v_{13}\big),
\end{align*}
as we wished.
Since the proof of \eqref{Phim} concerning $v_{31}$ is similar to the case that we just showed, we leave it to the reader.

Now, we shall prove~\eqref{Phim} for $v_{23}$. To start with, by~\eqref{eq:Boalch-eq-111.i},
\begin{align*}
 \lr{\gamma_3,v_{23}}&=\lr{v_{31},v_{23}}*v_{13}+v_{31}*\lr{v_{13},v_{23}}+\lr{v_{32},v_{23}}*v_{23}
 \\
 &=v_{23}\otimes e_3+\frac{1}{2}v_{23}\otimes\big(v_{31}v_{13}+v_{32}v_{23}\big)+\frac{1}{2}v_{23}\big(v_{31}v_{13}+v_{32}v_{23}\big)\otimes e_3
 \\
 &=\frac{1}{2}\big(v_{23}\otimes\gamma_3+v_{23}\gamma_3\otimes e_3\big).
\end{align*}

Next, by~\eqref{eq:Boalch-eq-111.e} and~\eqref{eq:Boalch-eq-111.f},
\begin{align*}
& \lr{\gamma_2,v_{23}}
 \\
 &=v_{21}*\lr{v_{12},v_{23}}+\lr{v_{21},v_{23}}*v_{12}
 \\
 &\quad - w_{23}\gamma_3*\lr{w_{32},v_{23}}-w_{23}*\lr{\gamma_3,v_{23}}*w_{32}-\lr{w_{23},v_{23}}*\gamma_3w_{32}
 \\
 &=e_2\otimes (v_{21}v_{13}-w_{23}\gamma_3)+\frac{1}{2}e_2\otimes\big(w_{23}\gamma_3w_{32}-v_{21}v_{12}\big)v_{23}+\frac{1}{2}\big(w_{23}\gamma_3w_{32}-v_{21}v_{12}\big)\otimes v_{23}
 \\
 &=-\frac{1}{2}\big(e_2\otimes\gamma_2v_{23}+\gamma_2\otimes v_{23}\big).
\end{align*}

Finally, the last case to study deals with $\gamma_1$ and $v_{23}$:
\begin{align*}
 \lr{\gamma_1,v_{23}}&=-w_{12}\gamma_2*\lr{w_{21},v_{23}}-w_{12}*\lr{\gamma_2,v_{23}}*w_{21}-\lr{w_{12},v_{23}}*\gamma_2w_{21}
 \\
 &\quad -w_{13}\gamma_3*\lr{w_{31},v_{23}}-w_{13}*\lr{\gamma_3,v_{23}}*w_{31}-\lr{w_{13},v_{23}}*\gamma_3w_{31}
 \\
 &=-\frac{1}{2}w_{21}\otimes w_{12}\gamma_2v_{23}+\frac{1}{2}w_{21}\otimes w_{12}\gamma_2v_{23}+\frac{1}{2}\gamma_2w_{21}\otimes w_{12}v_{23}
 \\
 &\quad -\frac{1}{2}\gamma_2w_{21}\otimes w_{12}v_{23}+\frac{1}{2}v_{23}w_{31}\otimes w_{13}\gamma_3
 \\
 &\quad -\frac{1}{2} v_{23}w_{31}\otimes w_{13}\gamma_3-\frac{1}{2}v_{23}\gamma_3w_{31}\otimes w_{13}+\frac{1}{2}v_{23}\gamma_3w_{31}\otimes w_{13}
 \\
 &=0.
\end{align*}
Due to the similarities, the reader can check that the multiplicative moment map condition~\eqref{Phim} applied to $v_{32}$ holds.
\end{proof}

Now, under the notations of Theorem \ref{thm:triangle}, and 
given a dimension vector $d=(d_1,d_2,d_3)$ with $N:=d_1+d_2+d_3$ as in \ref{sec:KR-principle}, we consider the representation scheme $\Rep\big(\Bc(\Delta),d\big)$, which is acted on by the group $G_{d}=\prod_{s\in I}\GL_{d_s}(\kk)$, with $\kk$ an algebraically closed field of characteristic zero.  
Combining Theorem \ref{thm:triangle} with Proposition \ref{prop:double-quasi-Ham-red} and Theorem \ref{thm:KR-theorem}, we obtain the following two interesting corollaries, which match with Boalch's results:

\begin{cor}
Using the notation introduced in \ref{sec:fission-algebras}, the following holds:
\begin{enumerate}
\item[\textup{(i)}]
The fission algebra $\mathcal{F}^q(\Delta)$ admits an $H_0$-Poisson structure (in the sense of Definition~\ref{def:H0-Poisson-str}).
\item[\textup{(ii)}]
The $H_0$-Poisson structure on $\mathcal{F}^q(\Delta)$ induces a Poisson structure
on $\Big(\mathcal{F}^q(\Delta)\Big)^{\GL_{d}}_d$, the coordinate ring of the colored multiplicative quiver variety attached to $\Delta$.
\end{enumerate}
\label{cor:1}
\end{cor}

\begin{cor}
The Hamiltonian double quasi-Poisson algebra $\big(\mathcal{B}(\Delta),\lr{-,-},\Phi\big)$ induces a  Hamiltonian quasi-Poisson $\GL_d$-structure on $\Rep(\Bc(\Delta), d)$. 
\label{cor:2}
\end{cor}

Note that Boalch \cite[Corollary 5.7]{B15} proves in particular that  $\Rep(\Bc(\Delta), d)$ carries a quasi-Hamiltonian $\GL_{d}$-structure (our $\Rep(\Bc(\Delta), d)$ corresponds to his ``space of invertible representations'') in the sense of \cite[Definition 2.2]{AMM98} (after complexification), that is, a triple consisting of a $\GL_{d}$-variety, an invariant 2-form, and a  $\GL_{d}$-valued moment map.
Hence, as emphasized in the introduction, we expect that the double Poisson bracket $\lr{-,-}$ on $\mathcal{B}(\Delta)$ obtained in Theorem \ref{thm:triangle} to be nondegenerate (see \cite[Theorem 7.1]{VdB2}), thus giving rise to a quasi-bisymplectic algebra \cite[\S6]{VdB2} that induces the quasi-Hamiltonian $\GL_{d}$-structure of  \cite[Corollary 5.7]{B15}, via \cite[Proposition 6.1]{VdB2}.
Indeed, we expect that Conjecture \ref{Conj:Main} can be updated by stating that for each colored quiver $\Upsilon$, the Boalch algebra $\Bc(\Upsilon)$ carries a quasi-bisymplectic structure (cf. also Remark~4.4 in \cite{Bo14}). 
We will explore this line of research in a separate work.

Further, another consequence of Theorem \ref{thm:triangle} involves pre-Calabi--Yau algebras. A pre-Calabi--Yau structure on a graded vector space is a solution of a certain Maurer--Cartan equation formulated in terms of the so-called (generalized) necklace bracket (see \cite{KTV}). Geometrically, a pre-Calabi-Yau structure can be seen as a noncommutative shifted Poisson structure, since Yeung \cite{Ye} proved that they induce shifted Poisson structures on the derived moduli stack of representations in analogy with the Kontsevich--Rosenberg principle.
Furthermore, Theorem 5.1 of \cite{FH20} states that double quasi-Poisson algebras induce pre-Calabi--Yau algebra structures.
Hence, this result, combined with Theorem \ref{thm:triangle}, gives a pre-Calabi-Yau algebra structure on the Boalch algebra
$\mathcal{B}(\Delta)$, which we expect to be nondegenerate, giving rise to a (right) Calabi-Yau structure. 
It would be interesting to show how this structure descends to the fission algebra $\mathcal{F}^q(\Delta)$. 
In this direction, let us note that some recent works have investigated (left) Calabi-Yau structures \cite{BCS21} and the 2-Calabi-Yau property \cite{KS} for multiplicative preprojective algebras and their differential graded versions. Since by Lemma \ref{L:AlgIso}, they are particular examples of fission algebras, it seems natural to search for analogous results in the case of colored quivers, with $\mathcal{F}^q(\Delta)$ the first new case to investigate.

\section{Proof of the quasi-Poisson property in Theorem \ref{thm:triangle}}
\label{sec:proof-quasi-Poisson-property}

The algebra $\Bc(\Delta)$ is obtained by localization of $\kk \overline{\Delta}$ due to Lemma \ref{lem:localisation-Bgamma}. Furthermore, the double bracket given in Theorem \ref{thm:triangle} can be directly defined on $\kk \overline{\Delta}$---see \eqref{double-bracket-111}. Thus, if we show that this double bracket on $\kk \overline{\Delta}$ is quasi-Poisson, it will also be the case on $\Bc(\Delta)$ by localization and we are done. We will prove the quasi-Poisson property on $\kk \overline{\Delta}$ in \ref{sec:proof-quasi-Poisson-Particular} as a particular case of a general construction that is explained in \ref{sec:proof-quasi-Poisson-General}. We expect that the general construction that is carried out in \ref{sec:proof-quasi-Poisson-General}--\ref{ss:strat4} can be useful to prove Conjecture \ref{Conj:Main} in the case of the monochromatic complete $n$-partite graph on $n$ vertices. 

\subsection{General conditions for a double quasi-Poisson bracket} \label{sec:proof-quasi-Poisson-General}

Fix $n\geq 2$ and let $K_n$ be the complete $n$-partite graph over $n$ vertices. Following \ref{ss:Gr}, we fix a total order on the vertices which allows us to identify them with $I_n=\{1,\ldots,n\}$. The induced colored quiver $Q_n$ has for double $\overline{Q}_n$, which can be seen as the quiver over the vertex set $I_n$ with arrows $v_{ij}:j\to i$ for each $i\neq j$. 

We define a double bracket on $\kk\overline{Q}_n$ as follows. For each $i=1,\ldots,n$, we introduce two skewsymmetric matrices $\alpha^{(i)},\beta^{(i)}$ whose entries along the $i$-th row and the $i$-th column are zero. That is, the entries satisfy 
\begin{align*}
 &\alpha^{(i)}_{jk}=-\alpha^{(i)}_{kj},\quad \alpha^{(i)}_{i k}=0=\alpha^{(i)}_{ji}\,, \quad \text{ for all }\ 1\leq j,k \leq n\,, \\
&\beta^{(i)}_{jk}=-\beta^{(i)}_{kj},\quad \beta^{(i)}_{i k}=0=\beta^{(i)}_{ji}\,, \quad \text{ for all }\ 1\leq j,k \leq n\,.
 \end{align*}
We also introduce for each $i=1,\ldots,n$ two matrices $\mu^{(i)},\nu^{(i)}$ whose entries along the diagonal, the $i$-th row and the $i$-th column are zero. 
 That is, the entries satisfy  
\begin{align*}
 &\mu^{(i)}_{jj}=0,\ \mu^{(i)}_{i j}=0=\mu^{(i)}_{ji}\,, \quad
 \nu^{(i)}_{jj}=0,\ \nu^{(i)}_{i j}=0=\nu^{(i)}_{ji}\,, \quad  \text{ for all }\ 1\leq j \leq n\,. 
 \end{align*} 
Finally, for any triple of strictly decreasing indices $i>j>k$, we choose an arbitrary $\kappa_j^{(i,k)}\in \kk$. We let  $\kappa_j^{(i,k)}=0$ whenever the condition $i>j>k$ is not satisfied. We then define for $i,j,k,l\in I_n$
\begin{align}
 \lr{v_{ij},v_{ij}}&=0, \label{vv0}\\
 \lr{v_{ij},v_{kl}}&=0,\quad \text{ for }\{i,j\}\cap\{k,l\}=\emptyset\,, \\
 \lr{v_{ij},v_{kj}}&=\alpha^{(j)}_{ik} v_{kj}\otimes v_{ij} \,, \\
 \lr{v_{ij},v_{il}}&=\beta^{(i)}_{jl} v_{ij} \otimes v_{il}\,, \\
 \lr{v_{ij},v_{jl}}&\stackrel{i\neq l}{=} \mu_{il}^{(j)} e_j \otimes v_{ij}v_{jl} + \nu_{il}^{(j)} e_j \otimes v_{il}, \\
 \lr{v_{ij},v_{ki}}&\stackrel{k\neq j}{=} -\mu_{kj}^{(i)} v_{ki}v_{ij} \otimes e_i - \nu_{kj}^{(i)} v_{kj} \otimes e_i \,, \\
 \lr{v_{ij},v_{ji}}&\stackrel{i> j}{=} e_j \otimes e_i + \frac12 v_{ji}v_{ij}\otimes e_i + \frac12 e_j \otimes v_{ij}v_{ji} + \sum_{i>a>j} \kappa_a^{(i,j)} e_j \otimes v_{ia}v_{ai}  \,, \\
 \lr{v_{ij},v_{ji}}&\stackrel{i< j}{=} -e_j \otimes e_i - \frac12 v_{ji}v_{ij}\otimes e_i - \frac12 e_j \otimes v_{ij}v_{ji} - \sum_{i<b<j} \kappa_b^{(j,i)} v_{jb}v_{bj} \otimes e_i \,, \label{vvopp}
\end{align}
which can be checked to be a double bracket on $\kk\overline{Q}_n$. In particular, if $v_{ab}$ appears on the left hand side with $a=b$, the right hand side vanishes; this is consistent with the fact that there is no generator $v_{aa}$. For later use, we note that the last two expressions can be gathered together as 
\begin{equation*}
 \begin{aligned}
  \lr{v_{ij},v_{ji}}=& \sgn(i-j)\left[ e_j \otimes e_i + \frac12 v_{ji}v_{ij}\otimes e_i + \frac12 e_j \otimes v_{ij}v_{ji}\right] \\
  &+ \sum_{i>a>j} \kappa_a^{(i,j)} e_j \otimes v_{ia}v_{ai} - \sum_{i<b<j} \kappa_b^{(j,i)} v_{jb}v_{bj} \otimes e_i\,,
\end{aligned}
\end{equation*}
where $\sgn$ is the sign function and we follow the convention that a sum over an empty set vanishes.

We want to compute the triple bracket \eqref{Eq:TripBr} on generators, which is given by 
\begin{align*}
\lr{v_{ij},v_{kl},v_{pq}} = \lr{v_{ij},\lr{v_{kl},v_{pq}}}_L+\tau_{(123)}\lr{v_{kl},\lr{v_{pq},v_{ij}}}_L+\tau_{(132)}\lr{v_{pq},\lr{v_{ij},v_{kl}}}_L \,.
\end{align*}
For the double bracket to be quasi-Poisson, we need to impose that the triple bracket coincides with \eqref{eq:triple-bracket-E3} for all indices, which will impose conditions on the coefficients in \eqref{vv0}--\eqref{vvopp}. Thus, we will compute the above triple bracket using \eqref{vv0}--\eqref{vvopp}, and equate it to the desired triple bracket 
\begin{equation*}
   \begin{aligned} \label{qPabc-vvv}
    \lr{v_{ij},v_{kl},v_{pq}}_{\qP}=&
\frac14 \sum_{s\in I} \Big(v_{pq} e_s v_{ij} \otimes e_s v_{kl} \otimes e_s  - v_{pq} e_s v_{ij} \otimes e_s \otimes v_{kl} e_s \Big) \\
&-\frac14 \sum_{s\in I} \Big( v_{pq} e_s \otimes v_{ij} e_s v_{kl} \otimes e_s - v_{pq} e_s \otimes v_{ij} e_s \otimes v_{kl} e_s\Big) \\
&-\frac14 \sum_{s\in I} \Big( e_s v_{ij} \otimes e_s v_{kl} \otimes e_s v_{pq} - e_s v_{ij} \otimes e_s \otimes v_{kl} e_s v_{pq}\Big)  \\
&+\frac14 \sum_{s\in I} \Big( e_s \otimes v_{ij} e_s v_{kl} \otimes e_s v_{pq} - e_s \otimes v_{ij} e_s \otimes v_{kl} e_s v_{pq} \Big)\,.
  \end{aligned}
\end{equation*}
We note from this expression that $\lr{v_{ij},v_{kl},v_{pq}}_{\qP}=0$ trivially if there is no index $s$ appearing simultaneously in the index sets $\{i,j\}$, $\{k,l\}$ and $\{p,q\}$. We also remark that since $v_{ab}=0$ if $a=b$, we only need to consider the cases where $i\neq j$, $k\neq l$ and $p\neq q$. For example, if $i=j$, we directly get $\lr{v_{ij},v_{kl},v_{pq}}=0$ and $\lr{v_{ij},v_{kl},v_{pq}}_{\qP}=0$. 

We will study the quasi-Poisson property on $\lr{v_{ij},v_{kl},v_{pq}}$ according to different cases. They depend on the cardinality of the intersection 
of the sets of ``first'' and ``second'' indices 
\begin{equation} \label{Eq:S}
 S=\{i,k,p\}\cap \{j,l,q\}
\end{equation}
as follows : 
\begin{enumerate}
 \item[]\underline{Case 1.} $S=\emptyset$ (considered in \ref{ss:strat1});
 \item[]\underline{Case 2.} $S=\{\star\}$ has cardinality $1$ (considered in \ref{ss:strat2}); 
 \item[]\underline{Case 3.} $S=\{\star,\ast\}$ has cardinality $2$ (considered in \ref{ss:strat3}); 
 \item[]\underline{Case 4.} $S=\{\star,\ast,\bullet\}$ has cardinality $3$ (considered in \ref{ss:strat4}).
\end{enumerate}
We recall that we will always assume that $i\neq j$, $k\neq l$ and $p\neq q$, though we will not always write these three conditions.

\subsection{Conditions obtained from Case 1} \label{ss:strat1}

\begin{lem} \label{Lem:strategy-1}
When $S$ given by \eqref{Eq:S} is empty, the quasi-Poisson property holds for $\lr{v_{ij},v_{kl},v_{pq}}$  if and only if the following two conditions are satisfied:
\begin{enumerate}
 \item[(i)] either $j=l=q$, or we have $i=k=p$ and 
 \begin{equation}  \label{Eq:strategy-1i}
  \beta_{jl}^{(i)}\beta_{lq}^{(i)}+\beta_{lq}^{(i)}\beta_{qj}^{(i)}+\beta_{qj}^{(i)}\beta_{jl}^{(i)}=-\frac14\,;
 \end{equation} 
  \item[(ii)] either  $i=k=p$, or we have $j=l=q$ and 
  \begin{equation} \label{Eq:strategy-1ii}
\alpha_{ik}^{(j)}\alpha_{kp}^{(j)}+\alpha_{kp}^{(j)}\alpha_{pi}^{(j)}+\alpha_{pi}^{(j)}\alpha_{ik}^{(j)}=-\frac14\,.
  \end{equation}
\end{enumerate}
\end{lem}  
\begin{proof}
We first compute that 
\begin{align*}
 \lr{v_{ij},\lr{v_{kl},v_{pq}}}_L=&\,\delta_{kp}\beta_{lq}^{(k)}\Big[ \delta_{ik} \beta_{jl}^{(k)} v_{ij}\otimes v_{kl}\otimes v_{pq} + \delta_{jl} \alpha_{ik}^{(l)} v_{kl} \otimes v_{ij} \otimes v_{pq} \Big] \\
 &+\delta_{lq}\alpha_{kp}^{(l)}\Big[ \delta_{ip} \beta_{jq}^{(i)} v_{ij}\otimes v_{pq}\otimes v_{kl} + \delta_{jq} \alpha_{ip}^{(j)} v_{pq} \otimes v_{ij} \otimes v_{kl} \Big]\,, \\
 \tau_{(123)}\lr{v_{kl},\lr{v_{pq},v_{ij}}}_L=& 
\, \delta_{ip}\beta_{qj}^{(i)}\Big[ \delta_{kp} \beta_{lq}^{(k)} v_{ij}\otimes v_{kl}\otimes v_{pq} + \delta_{lq} \alpha_{kp}^{(l)} v_{ij} \otimes v_{pq} \otimes v_{kl} \Big] \\
 &+\delta_{qj}\alpha_{pi}^{(j)}\Big[ \delta_{ki} \beta_{lj}^{(i)} v_{pq}\otimes v_{kl}\otimes v_{ij} + \delta_{lj} \alpha_{ki}^{(j)} v_{pq} \otimes v_{ij} \otimes v_{kl} \Big]\,, \\
\tau_{(132)}\lr{v_{pq},\lr{v_{ij},v_{kl}}}_L=&
\, \delta_{ik}\beta_{jl}^{(i)}\Big[ \delta_{pi} \beta_{qj}^{(i)} v_{ij}\otimes v_{kl}\otimes v_{pq} + \delta_{qj} \alpha_{pi}^{(j)} v_{pq} \otimes v_{kl} \otimes v_{ij} \Big] \\
 &+\delta_{jl}\alpha_{ik}^{(j)}\Big[ \delta_{pk} \beta_{ql}^{(k)} v_{kl}\otimes v_{ij}\otimes v_{pq} + \delta_{ql} \alpha_{pk}^{(l)} v_{pq} \otimes v_{ij} \otimes v_{kl} \Big]\,.
\end{align*}
 After some simplifications relying on the skewsymmetry rules $\alpha^{(a)}_{bc}=-\alpha^{(a)}_{cb}$ and $\beta^{(a)}_{bc}=-\beta^{(a)}_{cb}$, we get that 
 \begin{align*}
\lr{v_{ij},v_{kl},v_{pq}}=&\,\delta_{ik}\delta_{kp}  [\beta_{jl}^{(i)}\beta_{lq}^{(i)}+\beta_{lq}^{(i)}\beta_{qj}^{(i)}+\beta_{qj}^{(i)}\beta_{jl}^{(i)}]\, v_{ij} \otimes v_{kl} \otimes v_{pq} \\
&-\delta_{jl}\delta_{lq} [\alpha_{ik}^{(j)}\alpha_{kp}^{(j)}+\alpha_{kp}^{(j)}\alpha_{pi}^{(j)}+\alpha_{pi}^{(j)}\alpha_{ik}^{(j)}]\, 
v_{pq}\otimes v_{ij} \otimes v_{kl} \,.
 \end{align*}
We also compute 
\begin{align*}
 \lr{v_{ij},v_{kl},v_{pq}}_{\qP} = \frac14 \left(\delta_{jl}\delta_{lq} \,v_{pq}\otimes v_{ij} \otimes v_{kl} - \delta_{ik}\delta_{kp} \,v_{ij} \otimes v_{kl} \otimes v_{pq}   \right)\,.
\end{align*}
Both triple brackets vanish when $i=k=p$ and $j=l=q$, while in the remaining cases it suffices to equate the coefficients of the two triple brackets in order to find the claimed conditions. 
\end{proof}

\begin{rem}
As a special case of Lemma \ref{Lem:strategy-1}, we recover the easy result that there are no conditions to verify for $\lr{v_{ij},v_{ij},v_{ij}}=0=\lr{v_{ij},v_{ij},v_{ij}}_{\qP}$.
\end{rem}

\subsection{Conditions obtained from Case 2}  \label{ss:strat2}

The single element $\star$ appearing in the intersection can occur either 
\begin{enumerate}
 \item[]\underline{Case 2.1.} once in $\{i,k,p\}$ and once in  $\{j,l,q\}$;
 \item[]\underline{Case 2.2.} once in $\{i,k,p\}$ and twice in  $\{j,l,q\}$; 
 \item[]\underline{Case 2.3.} twice in $\{i,k,p\}$ and once in  $\{j,l,q\}$.
\end{enumerate}
As we assume that $i\neq j$, $k\neq l$ and $p\neq q$, there is no other case  because all the other possibilities would lead to having an element  $v_{ij},v_{kl},v_{pq}$ of the form $v_{\star\star}$. 

\subsubsection{Case 2.1.} From the cyclicity of the triple bracket given as (see \ref{sec:sec-double-quasi-Ham})
\begin{align*}
 \lr{-,-,-}=\tau_{(123)}\circ \lr{-,-,-}\circ \tau_{(123)}^{-1}\,,
\end{align*}
we can assume without loss of generality that either $j=k=\star$, or $j=p=\star$. 

\begin{lem} \label{Lem:strategy-2.1.a}
When $S$ given by \eqref{Eq:S} is such that $S=\{\star\}$ with $\star=j=k$, the quasi-Poisson property always holds for $\lr{v_{ik},v_{kl},v_{pq}}$. 
\end{lem}
\begin{proof}
The condition on $S$ implies that $i,p\neq k$, with  $l\neq i,p,k$ and $q \neq i,p,k$. 
Thus, we compute that 
 \begin{equation*}
  \lr{v_{ik},\lr{v_{kl},v_{pq}}}_L=-\tau_{(123)}\lr{v_{kl},\lr{v_{pq},v_{ik}}}_L =\delta_{ip}\delta_{lq} \alpha_{kp}^{(l)}\beta_{kq}^{(i)} \, v_{ik}\otimes v_{pq} \otimes v_{kl}  \,,
 \end{equation*}
and $\lr{v_{pq},\lr{v_{ik},v_{kl}}}_L=0$. 
Hence $\lr{v_{ik},v_{kl},v_{pq}}=0$. Since $\lr{v_{ij},v_{kl},v_{pq}}_{\qP}=0$, both triple brackets vanish. 
\end{proof}

\begin{lem} \label{Lem:strategy-2.1.b}
When $S$ given by \eqref{Eq:S} is such that $S=\{\star\}$ with $\star=j=p$, the quasi-Poisson property holds for $\lr{v_{ip},v_{kl},v_{pq}}$ 
if and only if the following condition is satisfied:
 \begin{equation}  \label{Eq:strategy-2.1.bi}
  \nu_{iq}^{(p)}\Big[\delta_{lq} (\alpha_{kp}^{(l)}+\alpha_{ki}^{(l)})+\delta_{ik} (\beta_{ql}^{(i)}+\beta_{lp}^{(i)})\Big]=0\,.
 \end{equation} 
\end{lem}
\begin{proof}
The condition on $S$ implies that $i,k\neq p$, with  $l\neq i,k,p$ and $q \neq i,k,p$. We compute 
\begin{align*}
\lr{v_{ip},\lr{v_{kl},v_{pq}}}_L=&
\delta_{lq}\alpha_{kp}^{(l)}\mu_{iq}^{(p)}\, e_p\otimes v_{ip}v_{pq}\otimes v_{kl}  
+\delta_{lq}\alpha_{kp}^{(l)}\nu_{iq}^{(p)}\, e_p\otimes v_{iq}\otimes v_{kl}\,, \\
\tau_{(123)}\lr{v_{kl},\lr{v_{pq},v_{ip}}}_L=& 
-\delta_{lq}\alpha_{kp}^{(l)}\mu_{iq}^{(p)}\, e_p\otimes v_{ip}v_{pq}\otimes v_{kl}  
-\delta_{lq}\alpha_{ki}^{(l)}\nu_{iq}^{(p)}\, e_p\otimes v_{iq}\otimes v_{kl} \\
&-\delta_{ik}\beta_{lp}^{(i)}\mu_{iq}^{(p)}\, e_p\otimes  v_{kl} \otimes v_{ip}v_{pq} 
-\delta_{ik}\beta_{lq}^{(i)}\nu_{iq}^{(p)}\, e_p\otimes v_{kl}\otimes v_{iq} \,,\\
\tau_{(132)}\lr{v_{pq},\lr{v_{ip},v_{kl}}}_L=&
-\delta_{ik}\beta_{pl}^{(i)}\mu_{iq}^{(p)}\, e_p\otimes  v_{kl} \otimes v_{ip}v_{pq} 
-\delta_{ik}\beta_{pl}^{(i)}\nu_{iq}^{(p)}\, e_p\otimes v_{kl}\otimes v_{iq} \,.
\end{align*}
After easy cancellations using the skewsymmetry of $\beta^{(i)}$, we obtain 
$$\lr{v_{ip},v_{kl},v_{pq}}= \nu_{iq}^{(p)}\Big[\delta_{lq} (\alpha_{kp}^{(l)}+\alpha_{ki}^{(l)})+\delta_{ik} (\beta_{ql}^{(i)}+\beta_{lp}^{(i)})\Big] \, e_p\otimes v_{il}\otimes v_{kq}\,. $$
Since $\lr{v_{ip},v_{kl},v_{pq}}_{\qP}=0$, the quasi-Poisson property holds if and only if \eqref{Eq:strategy-2.1.bi} is satisfied. 
\end{proof}

\subsubsection{Case 2.2.} Using the cyclicity of the triple bracket, we can assume without loss of generality that $i=l=q=\star$. 

\begin{lem} \label{Lem:strategy-2.2}
When $S$ given by \eqref{Eq:S} is such that $S=\{\star\}$ with $\star=i=l=q$, the quasi-Poisson property holds for $\lr{v_{ij},v_{ki},v_{pi}}$ 
if and only if the following two conditions are satisfied:
\begin{align}
 \label{Eq:strategy-2.2i}
  \nu_{pj}^{(i)}\Big[\alpha_{kp}^{(i)} + \mu_{kj}^{(i)}+\delta_{kp} \beta_{ij}^{(k)} \Big]&=0\,,\\
   \label{Eq:strategy-2.2ii}
  \alpha_{kp}^{(i)}\mu_{kj}^{(i)}-\alpha_{kp}^{(i)}\mu_{pj}^{(i)}-\mu_{pj}^{(i)}\mu_{kj}^{(i)}&=-\frac14\,.
\end{align} 
\end{lem}
\begin{proof}
The condition on $S$ implies that $j,k,p\neq i$, with  $j\neq k,p$. We compute 
\begin{align*}
\lr{v_{ij},\lr{v_{ki},v_{pi}}}_L=& 
-\alpha_{kp}^{(i)}\mu_{pj}^{(i)}\, v_{pi}v_{ij}\otimes e_i \otimes v_{ki}  
-\alpha_{kp}^{(i)}\nu_{pj}^{(i)}\, v_{pj}\otimes e_i \otimes v_{ki}\,, \\
\tau_{(132)}\lr{v_{pi},\lr{v_{ij},v_{ki}}}_L=& 
-\alpha_{pk}^{(i)}\mu_{kj}^{(i)}\, v_{pi}v_{ij}\otimes e_i \otimes v_{ki}  
-\mu_{pj}^{(i)}\mu_{kj}^{(i)}\, v_{pi}v_{ij}\otimes e_i \otimes v_{ki}  \\
&-\mu_{kj}^{(i)}\nu_{pj}^{(i)}\, v_{pj}\otimes e_i \otimes v_{ki}
-\delta_{kp}\beta_{ij}^{(k)}\nu_{kj}^{(i)}\, v_{kj}\otimes e_i \otimes v_{pi} \,, 
\end{align*}
while $\lr{v_{ki},\lr{v_{pi},v_{ij}}}_L=0$. Summing the above terms, we get $\lr{v_{ij},v_{ki},v_{pi}}$. Meanwhile, we can see that 
$$\lr{v_{ij},v_{ki},v_{pi}}_{\qP}=-\frac14 v_{pi}v_{ij}\otimes e_i \otimes v_{ki}\,.$$
Thus the two triple brackets are equal if and only if \eqref{Eq:strategy-2.2i} and \eqref{Eq:strategy-2.2ii} hold because these are the coefficients 
of the terms $v_{pj}\otimes e_i \otimes v_{ki}$ and $v_{pi}v_{ij}\otimes e_i \otimes v_{ki}$, respectively. 
\end{proof}

\subsubsection{Case 2.3.} Using the cyclicity of the triple bracket, we can assume without loss of generality that $j=k=p=\star$. 

\begin{lem} \label{Lem:strategy-2.3}
When $S$ given by \eqref{Eq:S} is such that $S=\{\star\}$ with $\star=j=k=p$, the quasi-Poisson property holds for $\lr{v_{ij},v_{jl},v_{jq}}$ 
if and only if the following two conditions are satisfied:
\begin{align}
 \label{Eq:strategy-2.3i}
  \nu_{il}^{(j)}\Big[\beta_{lq}^{(j)} + \mu_{iq}^{(j)}+\delta_{lq} \alpha_{ij}^{(l)} \Big]&=0\,,\\
   \label{Eq:strategy-2.3ii}
  \beta_{lq}^{(j)}\mu_{il}^{(j)}-\beta_{lq}^{(j)}\mu_{iq}^{(j)}+\mu_{iq}^{(j)}\mu_{il}^{(j)}&=\frac14\,.
\end{align} 
\end{lem}
\begin{proof}
The condition on $S$ implies that $i,l,q\neq j$, with  $i\neq l,q$. We compute 
\begin{align*}
\lr{v_{ij},\lr{v_{jl},v_{jq}}}_L=&
\beta_{lq}^{(j)} \mu_{il}^{(j)} e_j\otimes v_{ij}v_{jl}\otimes v_{jq} 
+\beta_{lq}^{(j)} \nu_{il}^{(j)}  e_j\otimes v_{il}\otimes v_{jq}  \,, \\
\tau_{(123)}\lr{v_{jl},\lr{v_{jq},v_{ij}}}_L=&  
\mu_{iq}^{(j)}(\mu_{il}^{(j)}-\beta_{lq}^{(j)}) e_j\otimes v_{ij}v_{jl}\otimes v_{jq} 
+\nu_{il}^{(j)} (\mu_{iq}^{(j)}-\delta_{lq}\alpha_{ji}^{(l)} )  e_j\otimes v_{il}\otimes v_{jq}  \,, 
\end{align*}
while $\lr{v_{jq},\lr{v_{ij},v_{jl}}}_L=0$. Summing all these terms, we get $\lr{v_{ij},v_{jl},v_{jq}}$. 
Meanwhile, we can see that 
$$\lr{v_{ij},v_{jl},v_{jq}}_{\qP}=\frac14 e_j \otimes v_{ij}v_{jl}\otimes  v_{jq}\,.$$
Thus the two triple brackets are equal if and only if \eqref{Eq:strategy-2.3i} and \eqref{Eq:strategy-2.3ii} hold by matching the coefficients   
of the terms $e_j\otimes v_{il}\otimes v_{jq}$ and $e_j \otimes v_{ij}v_{jl}\otimes  v_{jq}$. 
\end{proof}

\subsection{Conditions obtained from Case 3} \label{ss:strat3}

The distinct elements $\star,\ast$ appearing in the intersection $S$ can appear in the sets $\{i,k,p\}$ and $\{j,l,q\}$ according to the following cases:
\begin{enumerate}
 \item[\underline{Case 3.1.}] $\star$ and $\ast$ both appear exactly once in each set;
 \item[\underline{Case 3.2.}] $\star$ appears twice in $\{i,k,p\}$ and once in  $\{j,l,q\}$, while $\ast$ appears once in each set; 
 \item[\underline{Case 3.3.}] $\star$ appears once in $\{i,k,p\}$ and twice in  $\{j,l,q\}$, while $\ast$ appears once in each set; 
 \item[\underline{Case 3.4.}] $\star$ appears twice in $\{i,k,p\}$ and once in  $\{j,l,q\}$, while $\ast$ appears once in in $\{i,k,p\}$ and twice in  $\{j,l,q\}$. 
\end{enumerate}
As we assume that $i\neq j$, $k\neq l$ and $p\neq q$, there is no other case to consider. We derive the conditions obtained in these different cases in the next subsections.

\subsubsection{Case 3.1.}  \label{sss:strat3-1}
Because both $\star$ and $\ast$ appear once in each set, one of the elements must be of the form $v_{\star \ast}$. 
Using the cyclicity of the triple bracket, we can assume without loss of generality that we are considering one of the following cases: 
\begin{enumerate}
 \item[]\underline{Case 3.1.a.} $\lr{v_{ij},v_{ji},v_{pq}}$ with $p,q\neq i,j$;
 \item[]\underline{Case 3.1.b.} $\lr{v_{ij},v_{ki},v_{jq}}$ with  $k,q\neq i,j$ and $k\neq q$; 
 \item[]\underline{Case 3.1.c.} $\lr{v_{ij},v_{jl},v_{pi}}$ with  $p,l\neq i,j$ and $p\neq l$.  
\end{enumerate}

\begin{lem} \label{Lem:strategy-3.1.a}
In the \emph{\underline{Case 3.1.a}}, the quasi-Poisson property holds for $\lr{v_{ij},v_{ji},v_{pq}}$ 
if and only if the conditions given in one of the following three cases are satisfied:
\begin{enumerate}
 \item[\textup{(i)}] when $i<p<j<q$ or $q<i<p<j$, either $\kappa_p^{(j,i)}=0$ or 
 $$\mu_{jq}^{(p)}-\beta_{qj}^{(p)}=0\,, \quad \nu_{jq}^{(p)}=0\,;$$
 \item[\textup{(ii)}] when $i<q<j<p$ or $p<i<q<j$, either $\kappa_q^{(j,i)}=0$ or 
 $$\mu_{pj}^{(q)}+\alpha_{pj}^{(q)}=0\,, \quad \nu_{pj}^{(q)}=0\,;$$
 \item[\textup{(iii)}] when $i<p,q<j$ with $p\neq q$, we have 
  $$\kappa_q^{(j,i)}(\mu_{pj}^{(q)}+\alpha_{pj}^{(q)})=0\,, \quad \kappa_p^{(j,i)}(\mu_{jq}^{(p)}-\beta_{qj}^{(p)})=0\,, \quad 
  \kappa_p^{(j,i)}\nu_{jq}^{(p)}-\kappa_q^{(j,i)}\nu_{pj}^{(q)}=0\,.$$
\end{enumerate} 
\end{lem}
\begin{proof}
 It is easy to notice that $\lr{v_{ij},\lr{v_{ji},v_{pq}}}_L=0$ and $\lr{v_{ji},\lr{v_{pq},v_{ij}}}_L=0$ because $i,j,p,q$ are pairwise distinct by assumption. Thus 
 \begin{align*}
&\lr{v_{ij},v_{ji},v_{pq}}=   \tau_{(132)}\lr{v_{pq},\lr{v_{ij},v_{ji}}}_L 
=-\tau_{(132)} \sum_{i<b<j} \kappa_b^{(j,i)} \, \lr{v_{pq},v_{jb}v_{bj}}\otimes e_i\\
=&\,\delta_{(i<p<j)}\kappa_p^{(j,i)}(\mu_{jq}^{(p)}-\beta_{qj}^{(p)}) \, v_{pj}\otimes e_i \otimes v_{jp}v_{pq}
+\delta_{(i<p<j)}\kappa_p^{(j,i)}\nu_{jq}^{(p)} \, v_{pj}\otimes e_i \otimes v_{jq} \\
&-\delta_{(i<q<j)} \kappa_q^{(j,i)} (\mu_{pj}^{(q)}+\alpha_{pj}^{(q)})\, v_{pq}v_{qj}\otimes e_i \otimes v_{jq}
-\delta_{(i<q<j)} \kappa_q^{(j,i)} \nu_{pj}^{(q)}\, v_{pj}\otimes e_i \otimes v_{jq}\,.
 \end{align*}
(Here and below, the symbol $\delta_{(a<b<c)}$ equals $+1$ if $a<b<c$, and is zero otherwise). At the same time, we easily get that 
$\lr{v_{ij},v_{ji},v_{pq}}_{\qP}=0$. The vanishing of the different coefficients in the above expansion for $\lr{v_{ij},v_{ji},v_{pq}}$ gives the claimed conditions. 
\end{proof}

\begin{lem} \label{Lem:strategy-3.1.b}
In the \emph{\underline{Case 3.1.b}}, the quasi-Poisson property holds for $\lr{v_{ij},v_{ki},v_{jq}}$ 
if and only if the following conditions are satisfied: 
\begin{enumerate}
 \item[\textup{(i)}] either $\nu_{iq}^{(j)}=0$, or $\mu_{kj}^{(i)}-\mu_{kq}^{(i)}=0$;
 \item[\textup{(ii)}]  either $\nu_{kj}^{(i)}=0$, or $\mu_{kq}^{(j)}-\mu_{iq}^{(j)}=0$;
 \item[\textup{(iii)}]  $ \nu_{kj}^{(i)}\nu_{kq}^{(j)}-\nu_{kq}^{(i)}\nu_{iq}^{(j)}=0$.
\end{enumerate} 
\end{lem}
\begin{proof}
 As the indices $i,j,k,q$ are pairwise distinct, $\lr{v_{ij},\lr{v_{ki},v_{jq}}}_L=0$. Next, we compute 
 \begin{align*}
\tau_{(123)}  \lr{v_{ki},\lr{v_{jq},v_{ij}}}_L=& 
-\mu_{iq}^{(j)} \mu_{kj}^{(i)} \, e_j\otimes e_i \otimes v_{ki}v_{ij}v_{jq}
-\mu_{iq}^{(j)} \nu_{kj}^{(i)}\,  e_j\otimes e_i \otimes v_{kj}v_{jq} \\
&-\nu_{iq}^{(j)} \mu_{kq}^{(i)} \, e_j\otimes e_i \otimes v_{ki}v_{iq} 
-\nu_{iq}^{(j)} \nu_{kq}^{(i)} \, e_j\otimes e_i \otimes v_{kq}\,, \\
\tau_{(132)}\lr{v_{jq},\lr{v_{ij},v_{ki}}}_L=& 
\,\mu_{iq}^{(j)} \mu_{kj}^{(i)} \, e_j\otimes e_i \otimes v_{ki}v_{ij}v_{jq}
+\mu_{kq}^{(j)} \nu_{kj}^{(i)} \, e_j\otimes e_i \otimes v_{kj}v_{jq} \\
&+\nu_{iq}^{(j)} \mu_{kj}^{(i)}\,  e_j\otimes e_i \otimes v_{ki}v_{iq} 
+\nu_{kq}^{(j)} \nu_{kj}^{(i)} \, e_j\otimes e_i \otimes v_{kq}\,.
 \end{align*}
Summing these terms, we get  
\begin{align*}
 \lr{v_{ij},v_{ki},v_{jq}}=&
 \nu_{iq}^{(j)} (\mu_{kj}^{(i)} - \mu_{kq}^{(i)})\,  e_j\otimes e_i \otimes v_{ki}v_{iq} 
 +\nu_{kj}^{(i)} (\mu_{kq}^{(j)} - \mu_{iq}^{(j)})\, e_j\otimes e_i \otimes v_{kj}v_{jq} \\
 &+ (\nu_{kq}^{(j)} \nu_{kj}^{(i)}-\nu_{iq}^{(j)} \nu_{kq}^{(i)}) \, e_j\otimes e_i \otimes v_{kq}\,.
\end{align*}
We want the latter expression to be equal to $\lr{v_{ij},v_{ki},v_{jq}}_{\qP}=0$,  which is easily seen to be equivalent to the conditions (i)--(iii). 
\end{proof}

\begin{lem} \label{Lem:strategy-3.1.c}
In the \emph{\underline{Case 3.1.c}}, the quasi-Poisson property always holds for $\lr{v_{ij},v_{jl},v_{pi}}$. 
\end{lem}
\begin{proof}
By assumption, $i,j,l,p$ are distinct, and it is easy to see that $\lr{v_{ik},v_{kl},v_{pq}}=0$ as well as $\lr{v_{ij},v_{kl},v_{pq}}_{\qP}=0$. 
\end{proof}

\subsubsection{Case 3.2.}   \label{sss:strat3-2}
Recall that $\star$ appears twice in $\{i,k,p\}$. Using the cyclicity of the triple bracket, we can assume without loss of generality that we are considering one of the following cases : 
\begin{enumerate}
 \item[]\underline{Case 3.2.a.} $\lr{v_{ij},v_{ik},v_{ji}}$ with $i,j,k$ distinct;
 \item[]\underline{Case 3.2.b.} $\lr{v_{ij},v_{ji},v_{ik}}$ with  $i,j,k$ distinct.  
\end{enumerate}

\begin{lem} \label{Lem:strategy-3.2.a}
In the \emph{\underline{Case 3.2.a}}, the quasi-Poisson property holds for $\lr{v_{ij},v_{ik},v_{ji}}$
if and only if the conditions given in one of the following three cases are satisfied: 
\begin{enumerate}
 \item[\textup{(i)}] when $i<j$ and $k\neq i,j$, we have 
 \begin{align*}
  \kappa_b^{(j,i)} (\mu_{jk}^{(i)} + \beta_{jk}^{(i)})&=0\,, \quad \text{ for all }i<b<j\,, \\
  \frac12 (\mu_{jk}^{(i)} + \beta_{jk}^{(i)})- \mu_{jk}^{(i)} \beta_{jk}^{(i)}&=\frac14\,, \\
  \nu_{jk}^{(i)} \left(\mu_{ik}^{(j)}+\frac12\right)&=0\,, \\
  \mu_{jk}^{(i)} + \beta_{jk}^{(i)} - \nu_{jk}^{(i)}\nu_{ik}^{(j)}&=0\,;
 \end{align*}
 \item[\textup{(ii)}]  when $k<j<i$ or $j<i<k$, we have 
  \begin{align*}
  \kappa_a^{(i,j)} (\mu_{jk}^{(i)} + \mu_{ak}^{(i)})&=0\,, \quad \text{ for all }j<a<i\,, \\
  \kappa_a^{(i,j)} (\beta_{jk}^{(i)} + \beta_{ka}^{(i)})&=0\,, \quad \text{ for all }j<a<i\,, \\
  \kappa_a^{(i,j)} \nu_{ak}^{(i)}&=0\,, \quad \text{ for all }j<a<i\,, \\
  \frac12 (\mu_{jk}^{(i)} + \beta_{jk}^{(i)})+ \mu_{jk}^{(i)} \beta_{jk}^{(i)}&=-\frac14\,, \\
  \nu_{jk}^{(i)} \left(\mu_{ik}^{(j)}-\frac12\right)&=0\,, \\
  \mu_{jk}^{(i)} + \beta_{jk}^{(i)} + \nu_{jk}^{(i)}\nu_{ik}^{(j)}&=0\,;
 \end{align*}
 \item[\textup{(iii)}] when $j<k<i$, we have 
   \begin{align*}
  \kappa_a^{(i,j)} (\mu_{jk}^{(i)} + \mu_{ak}^{(i)}+\frac12 \delta_{ak})&=0\,, \quad \text{ for all }j<a<i\,, \\
  \kappa_a^{(i,j)} (\beta_{jk}^{(i)} + \beta_{ka}^{(i)}+\frac12 \delta_{ak})&=0\,, \quad \text{ for all }j<a\leq k\,, \\
  \kappa_a^{(i,j)} (\beta_{jk}^{(i)} + \beta_{ka}^{(i)})+\kappa_a^{(i,k)}\kappa_{k}^{(i,j)}&=0\,, \quad \text{ for all }k<a<i\,, \\
  \kappa_a^{(i,j)} \nu_{ak}^{(i)}&=0\,, \quad \text{ for all }j<a<i\,, \\
  \frac12 (\mu_{jk}^{(i)} + \beta_{jk}^{(i)})+ \mu_{jk}^{(i)} \beta_{jk}^{(i)}&=-\frac14\,, \\
  \nu_{jk}^{(i)} \left(\mu_{ik}^{(j)}-\frac12\right)&=0\,, \\
  \mu_{jk}^{(i)} + \beta_{jk}^{(i)} + \nu_{jk}^{(i)}\nu_{ik}^{(j)}+\kappa_k^{(i,j)}&=0\,. 
 \end{align*}
\end{enumerate}
\end{lem}
\begin{proof}
 As $i,j,k$ are pairwise distinct, we can compute that 
 \begin{align*}
&\lr{v_{ij},\lr{v_{ik},v_{ji}}}_L=-\frac12 \sgn(i-j) \mu_{jk}^{(i)} e_j \otimes v_{ij}v_{ji}v_{ik}\otimes e_i
-\mu_{jk}^{(i)} \sum_{i>a>j} \kappa_a^{(i,j)} e_j \otimes v_{ia}v_{ai}v_{ik} \otimes e_i \\
&\qquad -\mu_{jk}^{(i)}\left(\frac12 \sgn(i-j) + \beta_{jk}^{(i)}\right) v_{ji}v_{ij} \otimes v_{ik}\otimes e_i 
+\mu_{jk}^{(i)} \sum_{i<b<j} \kappa_b^{(j,i)} v_{jb}v_{bj} \otimes v_{ik} \otimes e_i \\
&\qquad -\nu_{jk}^{(i)}\mu_{ik}^{(j)} e_j \otimes v_{ij}v_{jk} \otimes e_i 
-\left(\sgn(i-j)\mu_{jk}^{(i)} + \nu_{jk}^{(i)}\nu_{ik}^{(j)} \right) e_j \otimes v_{ik} \otimes e_i\,, 
 \end{align*}
 \begin{align*}
\tau_{(123)}&\lr{v_{ik},\lr{v_{ji},v_{ij}}}_L\\
=& -\frac12 \sgn(j-i) \mu_{jk}^{(i)} e_j \otimes v_{ij}v_{ji}v_{ik}\otimes e_i
- \sum_{i>a>j}\kappa_a^{(i,j)} \mu_{ak}^{(i)} e_j \otimes v_{ia}v_{ai}v_{ik} \otimes e_i \\
&-\frac12 \delta_{(j<k<i)}\kappa_k^{(i,j)} e_j\otimes v_{ik}v_{ki}v_{ik} \otimes e_i 
+\frac12 \sgn(j-i) \beta_{kj}^{(i)} e_j\otimes v_{ik} \otimes v_{ij}v_{ji} \\
&- \sum_{i>a>j} \kappa_a^{(i,j)} \beta_{ka}^{(i)} e_j \otimes v_{ik} \otimes v_{ia}v_{ai} 
-\frac12 \delta_{(j<k<i)}\kappa_k^{(i,j)}  e_j\otimes v_{ik} \otimes v_{ik}v_{ki} \\
&-\delta_{(j<k<i)}\kappa_k^{(i,j)} \sum_{i>c>k} \kappa_c^{(i,k)} e_j\otimes v_{ik} \otimes v_{ic}v_{ci}
-\frac12 \sgn(j-i) \nu_{jk}^{(i)} e_j\otimes v_{ij}v_{jk}\otimes e_i \\
&-\sum_{i>a>j} \kappa_a^{(i,j)} \nu_{ak}^{(i)} e_j \otimes v_{ia}v_{ak} \otimes e_i 
- \delta_{(j<k<i)}\kappa_k^{(i,j)} e_j \otimes v_{ik}\otimes e_i\,, 
 \end{align*}
 \begin{align*}
\tau_{(132)}&\lr{v_{ji},\lr{v_{ij},v_{ik}}}_L\\
=& \,\frac12\sgn(j-i) \beta_{jk}^{(i)} e_j \otimes v_{ik} \otimes v_{ij}v_{ji} 
-\beta_{jk}^{(i)} \sum_{i>a>j} \kappa_a^{(i,j)} e_j \otimes v_{ik} \otimes v_{ia}v_{ai} \\
&+\frac12 \sgn(j-i) \beta_{jk}^{(i)} v_{ji}v_{ij} \otimes v_{ik} \otimes e_i 
+\beta_{jk}^{(i)} \sum_{j>b>i} \kappa_b^{(j,i)} v_{jb}v_{bj} \otimes v_{ik}\otimes e_i \\  
&+\sgn(j-i) \beta_{jk}^{(i)} e_j \otimes v_{ik} \otimes e_i \,.
 \end{align*}
Summing the expressions when $i<j$, we can get 
\begin{align*}
&\lr{v_{ij},v_{ik},v_{ji}} \qquad \qquad \qquad [\text{ subcase (i) where }i<j]\\
=&\sum_{i<b<j}\kappa_b^{(j,i)} (\mu_{jk}^{(i)} + \beta_{jk}^{(i)}) v_{jb}v_{bj} \otimes v_{ik} \otimes e_i 
  +\left(\frac12 (\mu_{jk}^{(i)} + \beta_{jk}^{(i)})- \mu_{jk}^{(i)} \beta_{jk}^{(i)}\right) v_{ji}v_{ij} \otimes v_{ik}\otimes e_i \\
&-\nu_{jk}^{(i)} \left(\mu_{ik}^{(j)}+\frac12\right) e_j \otimes v_{ij}v_{jk} \otimes e_i 
  +(\mu_{jk}^{(i)} + \beta_{jk}^{(i)} - \nu_{jk}^{(i)}\nu_{ik}^{(j)}) e_j \otimes v_{ik}\otimes e_i\,.
\end{align*}
Doing the same for $i>j$ when $k\notin \{j+1,\ldots,i-1\}$, 
\begin{align*}
&\lr{v_{ij},v_{ik},v_{ji}} \qquad \qquad \qquad [\text{ subcase (ii) where }k<j<i\text{ or }j<i<k] \\
=& 
-\sum_{j<a<i} \kappa_a^{(i,j)} (\mu_{jk}^{(i)} + \mu_{ak}^{(i)}) e_j \otimes v_{ia} v_{ai}v_{ik} \otimes e_i
 -\sum_{j<a<i} \kappa_a^{(i,j)} (\beta_{jk}^{(i)} + \beta_{ka}^{(i)}) e_j \otimes v_{ik} \otimes v_{ia}v_{ai} \\
&-\sum_{j<a<i} \kappa_a^{(i,j)} \nu_{ak}^{(i)} e_j \otimes v_{ia}v_{ak} \otimes e_i 
-\left(\frac12 (\mu_{jk}^{(i)} + \beta_{jk}^{(i)})+ \mu_{jk}^{(i)} \beta_{jk}^{(i)}\right) v_{ji}v_{ij} \otimes v_{ik} \otimes e_i \\
&-\nu_{jk}^{(i)} \left(\mu_{ik}^{(j)}-\frac12\right) e_j \otimes v_{ij}v_{jk} \otimes e_i
 -( \mu_{jk}^{(i)} + \beta_{jk}^{(i)} + \nu_{jk}^{(i)}\nu_{ik}^{(j)}) e_j \otimes v_{ik} \otimes e_i\,.
\end{align*}
In the remaining case, we can write that 
\begin{align*}
&\lr{v_{ij},v_{ik},v_{ji}} \qquad \qquad \qquad [\text{ subcase (iii) where }j<k<i] \\
=& 
-\sum_{j<a<i} \kappa_a^{(i,j)} (\mu_{jk}^{(i)} + \mu_{ak}^{(i)}+\frac12 \delta_{ak}) e_j \otimes v_{ia} v_{ai}v_{ik} \otimes e_i \\
& -\sum_{j<a\leq k} \kappa_a^{(i,j)} (\beta_{jk}^{(i)} + \beta_{ka}^{(i)}+\frac12 \delta_{ak}) e_j \otimes v_{ik} \otimes v_{ia}v_{ai}  \\
&  -\sum_{k<a<i} \big(\kappa_a^{(i,j)} (\beta_{jk}^{(i)} + \beta_{ka}^{(i)})+\kappa_a^{(i,k)}\kappa_k^{(i,j)} \big) e_j \otimes v_{ik} \otimes v_{ia}v_{ai}\\
&-\sum_{j<a<i} \kappa_a^{(i,j)} \nu_{ak}^{(i)} e_j \otimes v_{ia}v_{ak} \otimes e_i 
-\left(\frac12 (\mu_{jk}^{(i)} + \beta_{jk}^{(i)})+ \mu_{jk}^{(i)} \beta_{jk}^{(i)}\right) v_{ji}v_{ij} \otimes v_{ik} \otimes e_i \\
&-\nu_{jk}^{(i)} \left(\mu_{ik}^{(j)}-\frac12\right) e_j \otimes v_{ij}v_{jk} \otimes e_i
 -( \mu_{jk}^{(i)} + \beta_{jk}^{(i)} + \nu_{jk}^{(i)}\nu_{ik}^{(j)}+\kappa_k^{(i,j)}) e_j \otimes v_{ik} \otimes e_i\,.
\end{align*}
Putting these different expressions to be equal to 
\begin{align*}
 \lr{v_{ij},v_{ik},v_{ji}}_{\qP}=\frac14 \,v_{ji}v_{ij}\otimes v_{ik}\otimes e_i\,,
\end{align*}
is equivalent to the claimed conditions.
\end{proof}

\begin{lem} \label{Lem:strategy-3.2.b}
In the \emph{\underline{Case 3.2.b}}, the quasi-Poisson property holds for $\lr{v_{ij},v_{ji},v_{ik}}$
if and only if the conditions given in one of the following three cases are satisfied: 
\begin{enumerate}
 \item[\textup{(i)}] when $i>j$ and $k\neq i,j$, we have 
 \begin{align*}
&\frac12 (\mu_{jk}^{(i)} + \beta_{jk}^{(i)})+ \mu_{jk}^{(i)} \beta_{jk}^{(i)}=-\frac14\,, \quad \text{ and }\quad 
\nu_{jk}^{(i)} \left(\beta_{jk}^{(i)}+\frac12\right)=0\,;
 \end{align*}
 \item[\textup{(ii)}]   when $k<i<j$ or $i<j<k$, we have 
 \begin{align*}
&\frac12 (\mu_{jk}^{(i)} + \beta_{jk}^{(i)})- \mu_{jk}^{(i)} \beta_{jk}^{(i)}=\frac14\,, \quad \text{ and }\quad 
  \nu_{jk}^{(i)} \left(\beta_{jk}^{(i)}-\frac12\right)=0\,;
 \end{align*}
 \item[\textup{(iii)}] when $i<k<j$, we have 
 \begin{align*}
\frac12 (\mu_{jk}^{(i)} + \beta_{jk}^{(i)})- \mu_{jk}^{(i)} \beta_{jk}^{(i)}&=\frac14\,, \\
  \nu_{jk}^{(i)} \left(\beta_{jk}^{(i)}-\frac12\right)+\kappa_k^{(j,i)}\nu_{ij}^{(k)}&=0\,,\\
  \kappa_k^{(j,i)} (\alpha_{ij}^{(k)}+\mu_{ij}^{(k)})&=0\,.
 \end{align*}
\end{enumerate}
\end{lem}
\begin{proof}
  As $i,j,k$ are pairwise distinct, we can compute that $\lr{v_{ij},\lr{v_{ji},v_{ik}}}_L=0$ and then 
  \begin{align*}
\tau_{(123)}\lr{v_{ji},\lr{v_{ik},v_{ij}}}_L=&
\beta_{kj}^{(i)} \mu_{jk}^{(i)} v_{ij}\otimes e_i \otimes v_{ji}v_{ik}
+\beta_{kj}^{(i)} \nu_{jk}^{(i)} v_{ij}\otimes e_i \otimes v_{jk}\,, \\
\tau_{(132)}\lr{v_{ik},\lr{v_{ij},v_{ji}}}_L=&\frac12 \sgn(i-j) (\beta_{kj}^{(i)}-\mu_{jk}^{(i)}) v_{ij}\otimes e_i \otimes v_{ji}v_{ik} \\
&-\delta_{(i<k<j)}\kappa_k^{(j,i)} (\alpha_{ij}^{(k)}+\mu_{ij}^{(k)}) v_{ik}v_{kj} \otimes e_i \otimes v_{jk} \\
&-\left( \frac12 \sgn(i-j)\nu_{jk}^{(i)}+\delta_{(i<k<j)}\kappa_k^{(j,i)}\nu_{ij}^{(k)}\right) v_{ij}\otimes e_i\otimes v_{jk}\,.
  \end{align*}
  We thus get that 
  \begin{align*}
\lr{v_{ij},v_{ji},v_{ik}} =& 
\left( \frac12 \sgn(j-i) (\mu_{jk}^{(i)}+\beta_{jk}^{(i)})-\mu_{jk}^{(i)}\beta_{jk}^{(i)} \right) v_{ij}\otimes e_i \otimes v_{ji}v_{ik} \\
&-\delta_{(i<k<j)}\kappa_k^{(j,i)} (\alpha_{ij}^{(k)}+\mu_{ij}^{(k)}) v_{ik}v_{kj} \otimes e_i \otimes v_{jk} \\
&-\left( \nu_{jk}^{(i)} \Big(\beta_{jk}^{(i)}+ \frac12 \sgn(i-j) \Big)+\delta_{(i<k<j)}\kappa_k^{(j,i)}\nu_{ij}^{(k)}\right) v_{ij}\otimes e_i\otimes v_{jk}\,.
  \end{align*}
We then see that this triple bracket coincides with 
$$\lr{v_{ij},v_{ji},v_{ik}}_{\qP} = \frac14 \,v_{ij}\otimes e_i \otimes v_{ji}v_{ik}\,,$$
if and only if the claimed conditions are satisfied. 
\end{proof}

\subsubsection{Case 3.3.}  \label{sss:strat3-3}
Recall that $\star$ appears twice in $\{j,l,q\}$. Using the cyclicity of the triple bracket, we can assume without loss of generality that we are considering one of the following cases: 
\begin{enumerate}
 \item[]\underline{Case 3.3.a.} $\lr{v_{ij},v_{kj},v_{ji}}$ with $i,j,k$ distinct;
 \item[]\underline{Case 3.3.b.} $\lr{v_{ij},v_{ji},v_{kj}}$ with  $i,j,k$ distinct.  
\end{enumerate}

\begin{lem} \label{Lem:strategy-3.3.a}
In the \emph{\underline{Case 3.3.a}}, the quasi-Poisson property holds for $\lr{v_{ij},v_{kj},v_{ji}}$ 
if and only if the conditions given in one of the following three cases are satisfied: 
\begin{enumerate}
 \item[\textup{(i)}] when $j>i$ and $k\neq i,j$, we have 
 \begin{align*}
&\frac12 (\mu_{ki}^{(j)} + \alpha_{ki}^{(j)})+ \mu_{ki}^{(j)} \alpha_{ki}^{(j)}=-\frac14\,, \quad \text{ and }\quad 
\nu_{ki}^{(j)} \left(\alpha_{ki}^{(j)}+\frac12\right)=0\,;
 \end{align*}
 \item[\textup{(ii)}]   when $k<j<i$ or $j<i<k$, we have 
 \begin{align*}
&\frac12 (\mu_{ki}^{(j)} + \alpha_{ki}^{(j)})- \mu_{ki}^{(j)} \alpha_{ki}^{(j)}=\frac14\,, \quad \text{ and }\quad 
\nu_{ki}^{(j)} \left(\alpha_{ki}^{(j)}-\frac12\right)=0\,;
 \end{align*}
 \item[\textup{(iii)}] when $j<k<i$, we have 
 \begin{align*}
\frac12 (\mu_{ki}^{(j)} + \alpha_{ki}^{(j)})- \mu_{ki}^{(j)} \alpha_{ki}^{(j)}&=\frac14\,, \\
\nu_{ki}^{(j)} \left(\alpha_{ki}^{(j)}-\frac12\right)+ \kappa_k^{(i,j)}\nu_{ij}^{(k)}&=0\,,\\
  \kappa_k^{(i,j)} (\beta_{ij}^{(k)}+\mu_{ij}^{(k)})&=0\,.
 \end{align*}
\end{enumerate}
\end{lem}
\begin{proof}
 In a way similar to Lemma \ref{Lem:strategy-3.2.b}, we get 
 \begin{align*}
\lr{v_{ij},v_{kj},v_{ji}}=&\left(\frac12 \sgn(j-i) (\mu_{ki}^{(j)}+\alpha_{ki}^{(j)}) + \mu_{ki}^{(j)}\alpha_{ki}^{(j)} \right)  e_j \otimes v_{ij} \otimes v_{kj} v_{ji} \\
&+\delta_{(j<k<i)} \kappa_k^{(i,j)} (\mu_{ij}^{(k)}+\beta_{ij}^{(k)}) e_j  \otimes v_{ik} v_{kj} \otimes v_{ki} \\
&+ \left(\nu_{ki}^{(j)}\Big(\alpha_{ki}^{(j)} + \frac12 \sgn(j-i) \Big) + \delta_{(j<k<i)} \kappa_k^{(i,j)} \nu_{ij}^{(k)} \right) e_j \otimes v_{ij} \otimes v_{ki}\,.
 \end{align*}
This coincides with 
$$\lr{v_{ij},v_{kj},v_{ji}}_{\qP}=-\frac14 \, e_j \otimes v_{ij} \otimes v_{kj}v_{ki}\,,$$
if and only if the stated conditions are fulfilled. 
\end{proof}

\begin{lem} \label{Lem:strategy-3.3.b}
In the \emph{\underline{Case 3.3.b}}, the quasi-Poisson property holds for $\lr{v_{ij},v_{ji},v_{kj}}$ 
if and only if the conditions given in one of the following three cases are satisfied: 
\begin{enumerate}
 \item[\textup{(i)}] when $i>j$ and $k\neq i,j$, we have 
 \begin{align*}
  \kappa_a^{(i,j)} (\mu_{ki}^{(j)} + \alpha_{ki}^{(j)})&=0\,, \quad \text{ for all }j<a<i\,, \\
  \frac12 (\mu_{ki}^{(j)} + \alpha_{ki}^{(j)})- \mu_{ki}^{(j)} \alpha_{ki}^{(j)}&=\frac14\,, \\
  \nu_{ki}^{(j)} \left(\mu_{kj}^{(i)}+\frac12\right)&=0\,, \\
  \mu_{ki}^{(j)} + \alpha_{ki}^{(j)} - \nu_{ki}^{(j)}\nu_{kj}^{(i)}&=0\,;
 \end{align*}
 \item[\textup{(ii)}]   when $k<i<j$ or $i<j<k$, we have 
  \begin{align*}
  \kappa_b^{(j,i)} (\mu_{ki}^{(j)} - \mu_{kb}^{(j)})&=0\,, \quad \text{ for all }i<b<j\,, \\
  \kappa_b^{(j,i)} (\alpha_{ki}^{(j)} - \alpha_{kb}^{(j)})&=0\,, \quad \text{ for all }i<b<j\,, \\
  \kappa_b^{(j,i)} \nu_{kb}^{(j)}&=0\,, \quad \text{ for all }i<b<j\,, \\
  \frac12 (\mu_{ki}^{(j)} + \alpha_{ki}^{(j)})+ \mu_{ki}^{(j)} \alpha_{ki}^{(j)}&=-\frac14\,, \\
  \nu_{ki}^{(j)} \left(\mu_{kj}^{(i)}-\frac12\right)&=0\,, \\
  \mu_{ki}^{(j)} + \alpha_{ki}^{(j)} + \nu_{ki}^{(j)}\nu_{kj}^{(i)}&=0\,;
 \end{align*}
 \item[\textup{(iii)}] when $i<k<j$, we have 
   \begin{align*}
  \kappa_b^{(j,i)} (\mu_{ki}^{(j)} - \mu_{kb}^{(j)}+\frac12 \delta_{kb})&=0\,, \quad \text{ for all }i<b<j\,, \\
  \kappa_b^{(j,i)} (\alpha_{ki}^{(j)} - \alpha_{kb}^{(j)}+\frac12 \delta_{kb})&=0\,, \quad \text{ for all }i<b\leq k\,, \\
  \kappa_b^{(j,i)} (\alpha_{ki}^{(j)} - \alpha_{kb}^{(j)}) +\kappa_b^{(j,k)}\kappa_{k}^{(j,i)}&=0\,, \quad \text{ for all }k<b<j\,, \\
  \kappa_b^{(j,i)} \nu_{kb}^{(j)}&=0\,, \quad \text{ for all }i<b<j\,, \\
  \frac12 (\mu_{ki}^{(j)} + \alpha_{ki}^{(j)})+ \mu_{ki}^{(j)} \alpha_{ki}^{(j)}&=-\frac14\,, \\
  \nu_{ki}^{(j)} \left(\mu_{kj}^{(i)}-\frac12\right)&=0\,, \\
  \mu_{ki}^{(j)} + \alpha_{ki}^{(j)} + \nu_{ki}^{(j)}\nu_{kj}^{(i)}+\kappa_k^{(j,i)}&=0\,.
 \end{align*}
\end{enumerate}
\end{lem}
\begin{proof}
  As $i,j,k$ are pairwise distinct, we can compute that 
 \begin{align*}
&\lr{v_{ij},\lr{v_{ji},v_{kj}}}_L=-\frac12 \sgn(i-j) \mu_{ki}^{(j)} v_{kj}v_{ji}v_{ij}\otimes e_i \otimes e_j 
+\mu_{ki}^{(j)} \sum_{i<b<j} \kappa_b^{(j,i)}  v_{kj}v_{jb}v_{bj} \otimes e_i \otimes e_j \\
&\qquad -\mu_{ki}^{(j)}\left(\frac12 \sgn(i-j) - \alpha_{ki}^{(j)}\right) v_{kj}\otimes  v_{ij}v_{ji} \otimes  e_j 
-\mu_{ki}^{(j)} \sum_{i>a>j} \kappa_a^{(i,j)}  v_{kj}\otimes v_{ia}v_{ai} \otimes e_j \\
&\qquad +\nu_{ki}^{(j)}\mu_{kj}^{(i)} v_{ki}v_{ij} \otimes e_i \otimes e_j  
-\left(\sgn(i-j)\mu_{ki}^{(j)} - \nu_{ki}^{(j)}\nu_{kj}^{(i)} \right)  v_{kj} \otimes e_i \otimes e_j \,, 
 \end{align*}
  \begin{align*}
 \tau_{(123)}&\lr{v_{ji},\lr{v_{kj},v_{ij}}}_L\\
 =& 
 \,\frac12\sgn(j-i) \alpha_{ki}^{(j)}   v_{kj} \otimes v_{ij}v_{ji} \otimes e_j  
 -\alpha_{ki}^{(j)} \sum_{i>a>j} \kappa_a^{(i,j)}  v_{kj} \otimes v_{ia}v_{ai} \otimes e_j   \\
 &+\frac12 \sgn(j-i) \alpha_{ki}^{(j)}  v_{kj} \otimes e_i \otimes v_{ji}v_{ij}  
 +\alpha_{ki}^{(j)} \sum_{i<b<j} \kappa_b^{(j,i)}  v_{kj} \otimes e_i \otimes v_{jb}v_{bj} \\  
 &+\sgn(j-i) \alpha_{ki}^{(j)} v_{kj} \otimes e_i \otimes e_j \,,
  \end{align*}
  \begin{align*}
 \tau_{(132)}&\lr{v_{kj},\lr{v_{ij},v_{ji}}}_L\\
 =&
 -\frac12 \sgn(j-i) \mu_{ki}^{(j)}  v_{kj}v_{ji}v_{ij}\otimes e_i \otimes e_j 
 - \sum_{i<b<j} \kappa_b^{(j,i)} \mu_{kb}^{(j)}  v_{kj}v_{jb}v_{bj} \otimes e_i \otimes e_j \\
 &+\frac12 \delta_{(i<k<j)}\kappa_k^{(j,i)} v_{kj}v_{jk}v_{kj} \otimes e_i \otimes e_j
 -\frac12 \sgn(j-i) \alpha_{ki}^{(j)}  v_{kj}\otimes e_i \otimes v_{ji}v_{ij} \\
 &- \sum_{i<b<j} \kappa_b^{(j,i)} \alpha_{kb}^{(j)} v_{kj} \otimes e_i \otimes v_{jb}v_{bj} 
 +\frac12 \delta_{(i<k<j)}\kappa_k^{(j,i)}   v_{kj} \otimes e_i \otimes v_{jk}v_{kj} \\
 &+\delta_{(i<k<j)}\kappa_k^{(j,i)} \sum_{k<c<j} \kappa_c^{(j,k)}  v_{kj} \otimes e_i \otimes v_{jc}v_{cj}
 -\frac12 \sgn(j-i) \nu_{ki}^{(j)}  v_{ki}v_{ij}\otimes e_i \otimes e_j \\
 &-\sum_{i<b<j} \kappa_b^{(j,i)} \nu_{kb}^{(j)}  v_{kb}v_{bj} \otimes e_i \otimes e_j  
 + \delta_{(i<k<j)}\kappa_k^{(j,i)} v_{kj}\otimes e_i \otimes e_j\,. 
  \end{align*}
Summing the expressions when $i>j$, we can get 
\begin{align*}
&\lr{v_{ij},v_{ji},v_{kj}} \qquad \qquad \qquad [\text{ subcase (i) where }i>j]\\
=&
- \sum_{i>a>j}\kappa_a^{(i,j)} (\mu_{ki}^{(j)} + \alpha_{ki}^{(j)}) v_{kj} \otimes v_{ia}v_{ai} \otimes e_j 
   -\left(\frac12 (\mu_{ki}^{(j)} + \alpha_{ki}^{(j)})- \mu_{ki}^{(j)} \alpha_{ki}^{(j)}\right) v_{kj}\otimes  v_{ij}v_{ji} \otimes e_j \\
 &+\nu_{ki}^{(j)} \left(\mu_{kj}^{(i)}+\frac12\right) v_{ki}v_{ij} \otimes e_i \otimes e_j 
-(\mu_{ki}^{(j)} + \alpha_{ki}^{(j)} - \nu_{ki}^{(j)}\nu_{kj}^{(i)})  v_{kj}\otimes e_i \otimes e_j\,.
\end{align*}
Doing the same for $i<j$ when $k\notin \{i+1,\ldots,j-1\}$, 
\begin{align*}
&\lr{v_{ij},v_{ji},v_{kj}} \qquad \qquad \qquad [\text{ subcase (ii) where }k<i<j\text{ or }i<j<k] \\
=& 
 \sum_{i<b<j}\kappa_b^{(j,i)} (\mu_{ki}^{(j)} - \mu_{kb}^{(j)}) v_{kj} v_{jb}v_{bj} \otimes e_i \otimes e_j 
+\sum_{i<b<j}\kappa_b^{(j,i)} (\alpha_{ki}^{(j)} - \alpha_{kb}^{(j)}) v_{kj} \otimes e_i \otimes v_{jb}v_{bj}  \\
&-\sum_{i<b<j}\kappa_b^{(j,i)} \nu_{kb}^{(j)}  v_{kb}v_{bj} \otimes e_{i} \otimes e_j  
+\left(\frac12 (\mu_{ki}^{(j)} + \alpha_{ki}^{(j)})+ \mu_{ki}^{(j)} \alpha_{ki}^{(j)}\right) v_{kj}\otimes  v_{ij}v_{ji} \otimes e_j \\
 &+\nu_{ki}^{(j)} \left(\mu_{kj}^{(i)}-\frac12\right) v_{ki}v_{ij} \otimes e_i \otimes e_j 
+(\mu_{ki}^{(j)} + \alpha_{ki}^{(j)} + \nu_{ki}^{(j)}\nu_{kj}^{(i)})  v_{kj}\otimes e_i \otimes e_j\,.
\end{align*}
In the remaining case, we can write that 
\begin{align*}
&\lr{v_{ij},v_{ji},v_{kj}} \qquad \qquad \qquad [\text{ subcase (iii) where }i<k<j] \\
=& 
 \sum_{i<b<j}\kappa_b^{(j,i)} \left(\mu_{ki}^{(j)} - \mu_{kb}^{(j)}+\frac12 \delta_{kb}\right) v_{kj} v_{jb}v_{bj} \otimes e_i \otimes e_j  \\
&+\sum_{i<b \leq k}\kappa_b^{(j,i)} \left(\alpha_{ki}^{(j)} - \alpha_{kb}^{(j)} +\frac12 \delta_{kb} \right) v_{kj} \otimes e_i \otimes v_{jb}v_{bj}  \\
&+\sum_{k<b<j}\left( \kappa_b^{(j,i)} (\alpha_{ki}^{(j)} - \alpha_{kb}^{(j)}) + \kappa_b^{(j,k)}\kappa_k^{(j,i)} \right) v_{kj} \otimes e_i \otimes v_{jb}v_{bj}  \\
&-\sum_{i<b<j}\kappa_b^{(j,i)} \nu_{kb}^{(j)}  v_{kb}v_{bj} \otimes e_{i} \otimes e_j  
+\left(\frac12 (\mu_{ki}^{(j)} + \alpha_{ki}^{(j)})+ \mu_{ki}^{(j)} \alpha_{ki}^{(j)}\right) v_{kj}\otimes  v_{ij}v_{ji} \otimes e_j \\
 &+\nu_{ki}^{(j)} \left(\mu_{kj}^{(i)}-\frac12\right) v_{ki}v_{ij} \otimes e_i \otimes e_j 
+(\mu_{ki}^{(j)} + \alpha_{ki}^{(j)} + \nu_{ki}^{(j)}\nu_{kj}^{(i)} + \kappa_k^{(j,i)})  v_{kj}\otimes e_i \otimes e_j\,.
\end{align*}
Putting these different expressions to be equal to 
\begin{align*}
 \lr{v_{ij},v_{ji},v_{kj}}_{\qP}=-\frac14 \,v_{kj} \otimes v_{ij}v_{ji}\otimes  e_j\,,
\end{align*}
is equivalent to the claimed conditions.
\end{proof}

\subsubsection{Case 3.4.}
It suffices to consider the following case. 
\begin{lem} \label{Lem:strategy-3.4}
The quasi-Poisson property holds for $\lr{v_{ij},v_{ij},v_{ji}}$  if and only if 
the conditions given in one of the following two cases are satisfied: 
\begin{enumerate}
 \item[\textup{(i)}] when $i<j$, we have for all $i<b<j$ that 
 \begin{align*}
  &\kappa_b^{(j,i)}\Big(\alpha_{ib}^{(j)}-\frac12\Big)=0\,, \quad 
  \kappa_b^{(j,i)}\Big(\mu_{ib}^{(j)}+\frac12\Big)=0\,, \quad 
  \kappa_b^{(j,i)}\nu_{ib}^{(j)}=0\,;
 \end{align*}
 \item[\textup{(ii)}]  when $i>j$, we have for all $i>a>j$ that 
 \begin{align*}
  &\kappa_a^{(i,j)}\Big(\beta_{aj}^{(i)} - \frac12\Big)=0\,, \quad 
  \kappa_a^{(i,j)}\Big(\mu_{aj}^{(i)} + \frac12\Big)=0\,, \quad 
  \kappa_a^{(i,j)}\nu_{aj}^{(i)}=0\,.
 \end{align*}
 \end{enumerate}
\end{lem}
\begin{proof}
 We compute 
\begin{align*}
 &\lr{v_{ij},\lr{v_{ij},v_{ji}}}_L
 =\frac14 v_{ji}v_{ij}\otimes v_{ij} \otimes e_i 
 -\sum_{i<b<j} \kappa_b^{(j,i)} \Big(\alpha_{ib}^{(j)}-\frac12\Big) v_{jb} v_{bj} \otimes v_{ij} \otimes e_i \\
&\quad +\frac14 e_j \otimes v_{ij}v_{ji}v_{ij} \otimes e_i 
 -\sum_{i<b<j} \kappa_b^{(j,i)} \mu_{ib}^{(j)} e_j \otimes v_{ij} v_{jb} v_{bj} \otimes  e_i +\frac12 e_j \otimes v_{ij} \otimes e_i  \\
&\quad +\frac12 \sum_{i>a>j} \kappa_a^{(i,j)} e_j \otimes v_{ia}v_{ai}v_{ij} \otimes e_i 
 - \sum_{i<b<j} \kappa_b^{(j,i)} \nu_{ib}^{(j)} e_j \otimes v_{ib}v_{bj} \otimes e_i \,, \\
 \tau_{(123)}&\lr{v_{ij},\lr{v_{ji},v_{ij}}}_L
 =-\frac14 e_j \otimes v_{ij} \otimes v_{ij}v_{ji} 
 -\sum_{i>a>j} \kappa_a^{(i,j)} \Big(\beta_{ja}^{(i)}+\frac12\Big) e_j \otimes v_{ij} \otimes v_{ia} v_{ai}  \\
&\quad -\frac14 e_j \otimes v_{ij}v_{ji}v_{ij} \otimes e_i 
 +\sum_{i>a>j} \kappa_a^{(i,j)} \mu_{aj}^{(i)} e_j \otimes v_{ia} v_{ai} v_{ij} \otimes  e_i -\frac12 e_j \otimes v_{ij} \otimes e_i  \\
&\quad +\frac12 \sum_{i<b<j} \kappa_b^{(j,i)} e_j \otimes v_{ij}v_{jb}v_{bj} \otimes e_i 
 + \sum_{i>a>j} \kappa_a^{(i,j)} \nu_{aj}^{(i)} e_j \otimes v_{ia}v_{aj} \otimes e_i \,,
\end{align*}
 while $\lr{v_{ji},\lr{v_{ij},v_{ij}}}_L=0$.
We can then write  
\begin{align*}
&\lr{v_{ij},v_{ij},v_{ji}}\stackrel{i<j}{=}  
\frac14 ( v_{ji}v_{ij}\otimes v_{ij} \otimes e_i - e_j \otimes v_{ij} \otimes  v_{ij}v_{ji} ) \\
&\quad -\sum_{i<b<j} \kappa_b^{(j,i)}\left(\Big(\alpha_{ib}^{(j)}-\frac12\Big) v_{jb} v_{bj} \otimes v_{ij} \otimes e_i 
+ \Big(\mu_{ib}^{(j)}+\frac12\Big) e_j \otimes v_{ij}v_{jb}v_{bj} \otimes e_i \right) \\
&\quad - \sum_{i<b<j} \kappa_b^{(j,i)} \nu_{ib}^{(j)} e_j \otimes v_{ib}v_{bj} \otimes e_i \,,\\
&\lr{v_{ij},v_{ij},v_{ji}}\stackrel{i>j}{=}
\frac14 ( v_{ji}v_{ij}\otimes v_{ij} \otimes e_i - e_j \otimes v_{ij} \otimes  v_{ij}v_{ji} ) \\
&\quad +\sum_{i>a>j} \kappa_a^{(i,j)} \left( \Big(\beta_{aj}^{(i)}-\frac12\Big) e_j \otimes v_{ij} \otimes v_{ia} v_{ai} 
 +\Big(\mu_{aj}^{(i)}+\frac12\Big) e_j \otimes v_{ia}  v_{ai} v_{ij}\otimes e_i \right) \\ 
&\quad   + \sum_{i>a>j} \kappa_a^{(i,j)} \nu_{aj}^{(i)} e_j \otimes v_{ia}v_{aj} \otimes e_i \,.
\end{align*}
These expressions are equal to
\begin{align*}
 \lr{v_{ij},v_{ij},v_{ji}}_{\qP}=\frac14 ( v_{ji}v_{ij}\otimes v_{ij} \otimes e_i - e_j \otimes v_{ij} \otimes  v_{ij}v_{ji} )\,,
\end{align*}
if and only if the claimed identities are satisfied. 
\end{proof}

\subsection{Conditions obtained from Case 4} \label{ss:strat4}

The distinct elements $\star,\ast,\bullet$ appearing in the intersection $S$ all appear once in each of the sets $\{i,k,p\}$ and $\{j,l,q\}$. It suffice to consider the following cases 
\begin{enumerate}
 \item[]\underline{Case 4.1.} $\lr{v_{ik},v_{kp},v_{pi}}$ with $i,k,p$ distinct;
 \item[]\underline{Case 4.2.} $\lr{v_{ip},v_{ki},v_{pk}}$ with $i,k,p$ distinct. 
\end{enumerate}

 It is easy to derive the following result. 
\begin{lem} \label{Lem:strategy-4.1}
In the \emph{\underline{Case 4.1}}, the quasi-Poisson property always holds for $\lr{v_{ik},v_{kp},v_{pi}}$. 
\end{lem}

We finally arrive at the last set of conditions in the \underline{Case 4.2}. Note that by cyclicity of the triple bracket, we only need to consider that either $i<k<p$ or $i<p<k$. 
\begin{lem} \label{Lem:strategy-4.2}
In the \emph{\underline{Case 4.2}}, the quasi-Poisson property holds for $\lr{v_{ip},v_{ki},v_{pk}}$
if and only if the conditions given in one of the following two cases are satisfied: 
\begin{enumerate}
 \item[\textup{(i)}] when $i<k<p$, we have:
 \begin{align*}
 &\nu_{ik}^{(p)}+\nu_{kp}^{(i)}-\nu_{pi}^{(k)}=0\,, \qquad 
 \nu_{ik}^{(p)} \left(\mu_{pi}^{(k)}+\frac12 \right)=0\,, \\
 & \nu_{pi}^{(k)} \left(\mu_{kp}^{(i)}-\frac12 \right)=0\,, \qquad \nu_{ik}^{(p)} \left(\mu_{kp}^{(i)}-\frac12 \right)=0\,,   \\  
 &\nu_{pi}^{(k)} \left(\mu_{ik}^{(p)}+\frac12 \right)=0\,, \qquad \nu_{kp}^{(i)} \left(\mu_{ik}^{(p)}+\frac12 \right)=0\,,  \\
 &\nu_{kp}^{(i)} \left(\mu_{pi}^{(k)}-\frac12 \right) + \nu_{pi}^{(k)}\kappa_k^{(p,i)}=0\,, \\
 &\nu_{pi}^{(k)} \kappa_a^{(p,i)}=0\,, \qquad \nu_{ik}^{(p)} \kappa_a^{(k,i)}=0\,, \quad \text{ for all }\ i<a<k\,, \\
 &\nu_{pi}^{(k)} \kappa_c^{(p,i)} - \nu_{kp}^{(i)} \kappa_c^{(p,k)}=0\,, \quad \text{ for all }\ k<c<p\,;
 \end{align*}
 \item[\textup{(ii)}] when $i<p<k$, we have:
  \begin{align*}
 &\nu_{ik}^{(p)}-\nu_{kp}^{(i)}-\nu_{pi}^{(k)}=0\,, \qquad 
 \nu_{pi}^{(k)} \left(\mu_{ik}^{(p)}+\frac12 \right)=0\,, \\
 & \nu_{pi}^{(k)} \left(\mu_{kp}^{(i)}-\frac12 \right)=0\,, \qquad \nu_{ik}^{(p)} \left(\mu_{kp}^{(i)}-\frac12 \right)=0\,,   \\  
 &\nu_{ik}^{(p)} \left(\mu_{pi}^{(k)}+\frac12 \right)=0\,, \qquad \nu_{kp}^{(i)} \left(\mu_{pi}^{(k)}+\frac12 \right)=0\,,  \\
 &\nu_{kp}^{(i)} \left(\mu_{ik}^{(p)}-\frac12 \right) + \nu_{ik}^{(p)}\kappa_p^{(k,i)}=0\,, \\
 &\nu_{pi}^{(k)} \kappa_b^{(p,i)}=0\,, \qquad \nu_{ik}^{(p)} \kappa_b^{(k,i)}=0\,, \quad \text{ for all }\ i<b<p\,, \\
 &\nu_{ik}^{(p)} \kappa_d^{(k,i)} - \nu_{kp}^{(i)} \kappa_d^{(k,p)}=0\,, \quad \text{ for all }\ p<d<k\,.
 \end{align*}
\end{enumerate} 
\end{lem}
\begin{proof}
 Let us first compute the terms appearing in $\lr{v_{ip},v_{ki},v_{pk}}$ for arbitrary $i,k,p$ which are pairwise distinct. We have 
 \begin{align*}
&\lr{v_{ip},\lr{v_{ki},v_{pk}}}_L=
\mu_{pi}^{(k)} \mu_{kp}^{(i)} v_{pk}v_{ki}v_{ip}\otimes e_i \otimes e_k 
- \mu_{pi}^{(k)} \mu_{ik}^{(p)} e_p \otimes v_{ip}v_{pk}v_{ki} \otimes e_k \\ 
&\qquad +\mu_{pi}^{(k)} \nu_{kp}^{(i)} v_{pk}v_{kp}\otimes e_i \otimes e_k 
- \mu_{pi}^{(k)} \nu_{ik}^{(p)} e_p \otimes v_{ik}v_{ki} \otimes e_k \\
&\qquad+\nu_{pi}^{(k)} \sum_{i<b<p} \kappa_b^{(p,i)} v_{pb}v_{bp}\otimes e_i \otimes e_k 
- \nu_{pi}^{(k)} \sum_{i>a>p} \kappa_a^{(i,p)} e_p \otimes v_{ia}v_{ai} \otimes e_k\\
&\qquad-\frac12 \nu_{pi}^{(k)} \sgn(i-p) v_{pi}v_{ip} \otimes e_i \otimes e_k 
-\frac12 \nu_{pi}^{(k)}\sgn(i-p) e_p \otimes v_{ip}v_{pi}\otimes e_k \\
&\qquad-\nu_{pi}^{(k)}\sgn(i-p) e_p \otimes e_i \otimes e_k\,; \\
&\tau_{(123)}\lr{v_{ki},\lr{v_{pk},v_{ip}}}_L=
\mu_{ik}^{(p)} \mu_{pi}^{(k)} e_p \otimes v_{ip}v_{pk}v_{ki} \otimes e_k 
- \mu_{ik}^{(p)} \mu_{kp}^{(i)} e_p \otimes e_i \otimes v_{ki}v_{ip}v_{pk} \\ 
&\qquad +\mu_{ik}^{(p)} \nu_{pi}^{(k)} e_p \otimes v_{ip}v_{pi} \otimes e_k 
- \mu_{ik}^{(p)} \nu_{kp}^{(i)} e_p \otimes e_i \otimes v_{kp}v_{pk}  \\
&\qquad+\nu_{ik}^{(p)} \sum_{k<b'<i} \kappa_{b'}^{(i,k)}  e_p \otimes v_{ib'}v_{b'i} \otimes e_k 
- \nu_{ik}^{(p)} \sum_{k>a'>i} \kappa_{a'}^{(k,i)} e_p \otimes e_i \otimes v_{ka'}v_{a'k} \\
&\qquad-\frac12  \nu_{ik}^{(p)} \sgn(k-i) e_p \otimes v_{ik}v_{ki} \otimes e_k 
-\frac12  \nu_{ik}^{(p)}\sgn(k-i) e_p \otimes  e_i \otimes v_{ki}v_{ik} \\
&\qquad- \nu_{ik}^{(p)}\sgn(k-i) e_p \otimes e_i \otimes e_k\,; \\
&\tau_{(132)}\lr{v_{pk},\lr{v_{ip},v_{ki}}}_L=
\mu_{kp}^{(i)} \mu_{ik}^{(p)} e_p \otimes e_i \otimes  v_{ki}v_{ip}v_{pk}   
- \mu_{kp}^{(i)} \mu_{pi}^{(k)}   v_{pk}v_{ki}v_{ip} \otimes e_i \otimes e_k  \\ 
&\qquad +\mu_{kp}^{(i)} \nu_{ik}^{(p)} e_p \otimes e_i \otimes  v_{ki}v_{ik}   
- \mu_{kp}^{(i)} \nu_{pi}^{(k)}   v_{pi}v_{ip} \otimes e_i \otimes e_k  \\
&\qquad+\nu_{kp}^{(i)} \sum_{p<b''<k} \kappa_{b''}^{(k,p)}  e_p \otimes e_i \otimes v_{kb''}v_{b''k} 
- \nu_{kp}^{(i)}  \sum_{p>a''>k} \kappa_{a''}^{(p,k)}  v_{pa''}v_{a''p} \otimes e_i \otimes e_k \\
&\qquad-\frac12  \nu_{kp}^{(i)}  \sgn(p-k) e_p \otimes e_i \otimes v_{kp}v_{pk} 
-\frac12  \nu_{kp}^{(i)} \sgn(p-k) v_{pk}v_{kp}\otimes  e_i \otimes  e_k  \\
&\qquad- \nu_{kp}^{(i)} \sgn(p-k) e_p \otimes e_i \otimes e_k\,.
 \end{align*}
When $i<k<p$, we get by summing these terms 
\begin{align*}
 &\lr{v_{ip},v_{ki},v_{pk}}=- (\nu_{ik}^{(p)} + \nu_{kp}^{(i)} - \nu_{pi}^{(k)} ) e_p \otimes e_i \otimes e_k \\
 &\quad + \nu_{kp}^{(i)} \left(\mu_{pi}^{(k)}-\frac12 \right) v_{pk}v_{kp} \otimes e_i \otimes e_k 
 - \nu_{pi}^{(k)} \left(\mu_{kp}^{(i)}-\frac12 \right) v_{pi}v_{ip} \otimes e_i \otimes e_k \\ 
 &\quad + \nu_{pi}^{(k)} \sum_{i< b \leq k} \kappa_b^{(p,i)} v_{pb}v_{bp} \otimes e_i \otimes e_k 
 +  \sum_{k< c < p} (\nu_{pi}^{(k)} \kappa_c^{(p,i)} - \nu_{kp}^{(i)} \kappa_c^{(p,k)}  ) v_{pc}v_{cp} \otimes e_i \otimes e_k \\
 &\quad - \nu_{ik}^{(p)} \left(\mu_{pi}^{(k)}+\frac12 \right) e_p \otimes  v_{ik}v_{ki} \otimes e_k 
 + \nu_{pi}^{(k)} \left(\mu_{ik}^{(p)}+\frac12 \right) e_p \otimes v_{ip}v_{pi}  \otimes e_k \\
 &\quad - \nu_{kp}^{(i)} \left(\mu_{ik}^{(p)}+\frac12 \right)e_p \otimes e_i \otimes v_{kp}v_{pk} 
+ \nu_{ik}^{(p)} \left(\mu_{kp}^{(i)}-\frac12 \right)e_p \otimes e_i \otimes v_{ki}v_{ik} \\
 &\quad - \nu_{ik}^{(p)} \sum_{i<b<k}\kappa_b^{(k,i)} e_p \otimes e_i \otimes v_{kb}v_{bk}\,. 
\end{align*}
(Note that in the first sum that appears there is a term with $v_{pk}v_{kp} \otimes e_i \otimes e_k$).

When $i<p<k$, we get by summing these terms 
\begin{align*}
 &\lr{v_{ip},v_{ki},v_{pk}}=- (\nu_{ik}^{(p)} - \nu_{kp}^{(i)} - \nu_{pi}^{(k)} ) e_p \otimes e_i \otimes e_k \\
 &\quad + \nu_{kp}^{(i)} \left(\mu_{pi}^{(k)}+\frac12 \right) v_{pk}v_{kp} \otimes e_i \otimes e_k 
 - \nu_{pi}^{(k)} \left(\mu_{kp}^{(i)}-\frac12 \right) v_{pi}v_{ip} \otimes e_i \otimes e_k \\ 
 &\quad + \nu_{pi}^{(k)} \sum_{i< b <p} \kappa_b^{(p,i)} v_{pb}v_{bp} \otimes e_i \otimes e_k  \\
 &\quad - \nu_{ik}^{(p)} \left(\mu_{pi}^{(k)}+\frac12 \right) e_p \otimes  v_{ik}v_{ki} \otimes e_k 
 + \nu_{pi}^{(k)} \left(\mu_{ik}^{(p)}+\frac12 \right) e_p \otimes v_{ip}v_{pi}  \otimes e_k \\
 &\quad - \nu_{kp}^{(i)} \left(\mu_{ik}^{(p)}-\frac12 \right)e_p \otimes e_i \otimes v_{kp}v_{pk} 
+ \nu_{ik}^{(p)} \left(\mu_{kp}^{(i)}-\frac12 \right)e_p \otimes e_i \otimes v_{ki}v_{ik} \\
 &\quad - \nu_{ik}^{(p)} \sum_{i<b\leq p}\kappa_b^{(k,i)} e_p \otimes e_i \otimes v_{kb}v_{bk}
 + \sum_{p<d<k} (\nu_{kp}^{(i)}\kappa_d^{(k,p)} - \nu_{ik}^{(p)}\kappa_d^{(k,i)}) e_p \otimes e_i \otimes v_{kd}v_{dk}\,. 
\end{align*}
(Note that in the penultimate sum that appears there is a term with $e_p \otimes e_i \otimes v_{kp}v_{pk}$).

Since $\lr{v_{ip},v_{ki},v_{pk}}_{\qP}=0$, the quasi-Poisson property is equivalent to the vanishing of all the coefficients that appear, which is in turn equivalent to the claimed conditions. 
\end{proof}

\subsection{Checking the conditions for the double bracket of Theorem \ref{thm:triangle}}  \label{sec:proof-quasi-Poisson-Particular}

Recall that the (Boalch) algebra $\Bc(\Delta)$ corresponding to the monochromatic triangle is a localization of the path algebra  $\kk\overline{Q}_3=\kk \overline{\Delta}$. Its double bracket given in Theorem \ref{thm:triangle} can be defined on $\kk\overline{Q}_3$ directly, where it takes the form  \eqref{vv0}--\eqref{vvopp} for the coefficients given in Table \ref{Tab1}. 

\begin{table}[ht]
\centering
\begin{tabular}{c|c|c|c|c|c}
 $\alpha_{13}^{(2)}=+\frac12$ &$\alpha_{12}^{(3)}=+\frac12$ &$\alpha_{23}^{(1)}=-\frac12$ &
 $\alpha_{31}^{(2)}=-\frac12$ &$\alpha_{21}^{(3)}=-\frac12$ &$\alpha_{32}^{(1)}=+\frac12$\\ &&&&&\\ 
 $\beta_{13}^{(2)}=-\frac12$ &$\beta_{12}^{(3)}=-\frac12$ &$\beta_{23}^{(1)}=+\frac12$ 
  &$\beta_{31}^{(2)}=+\frac12$ &$\beta_{21}^{(3)}=+\frac12$ &$\beta_{32}^{(1)}=-\frac12$ \\ &&&&&\\
  $\mu_{13}^{(2)}=-\frac12$ &$\mu_{12}^{(3)}=-\frac12$ &$\mu_{23}^{(1)}=+\frac12$
&$\mu_{31}^{(2)}=-\frac12$ &$\mu_{21}^{(3)}=-\frac12$ &$\mu_{32}^{(1)}=+\frac12$ \\ &&&&&\\
   $\nu_{13}^{(2)}=+1$ &$\nu_{12}^{(3)}=0$ &$\nu_{23}^{(1)}=+1$ 
&$\nu_{31}^{(2)}=+1$ &$\nu_{21}^{(3)}=0$ &$\nu_{32}^{(1)}=+1$ \\ &&&&&\\
  &  & &$\kappa^{(3,1)}_{2}=+1$ &  &  
\end{tabular}
 \caption{Non-zero coefficients of the double bracket of the form \eqref{vv0}--\eqref{vvopp} on $\kk\overline{Q}_3$ (with $n=3$) which descends to the double bracket from Theorem \ref{thm:triangle} after localization.}
\label{Tab1}
\end{table}
All the coefficients that do not appear in this table are taken to be zero, in agreement with the properties of the coefficients given in \eqref{vv0}--\eqref{vvopp}. Therefore, if these coefficients satisfy the different conditions derived in the previous subsections, we get a double quasi-Poisson bracket on $\kk\overline{Q}_3$. Since localization preserves the quasi-Poisson property, this will induce that the double bracket given in Theorem \ref{thm:triangle} is quasi-Poisson. 

\medskip

In the remainder of this subsection, we check that the different conditions that have been derived in \ref{ss:strat1}--\ref{ss:strat4} are satisfied when $n=3$ with the coefficients from Table \ref{Tab1}.

\begin{rem}
 An important task consists in classifying all solutions to the conditions obtained in \ref{ss:strat1}--\ref{ss:strat4}. 
Indeed, this would yield potential candidates for the double quasi-Poisson bracket solving Conjecture \ref{Conj:Main} for the monochromatic $n$-partite graph on $n$ vertices. 
A classification for $n=2$ is available in \cite[\S4.2]{F19} (though the conjecture is settled in that case, cf. \ref{sec:the-11-monochromatic-case}). 
The use of a computer algebra software seems relevant; for the interested reader, let us mention Leray's code written in \texttt{Singular} for checking the double Poisson property on a quasi-free DGA \cite[Ann.C]{Ler}. 

Finally, note that there always exist solutions to these conditions. 
Starting with $n$ disjoint copies of the bipartite case and performing fusion as in \cite[\S5.3]{VdB1}, we end up with a double quasi-Poisson bracket with all $\kappa_j^{(i,k)},\nu_j^{(i,k)}=0$ and  
$\alpha_j^{(i,k)},\beta_j^{(i,k)},\mu_j^{(i,k)} \in \{\pm \frac12\}$. 
\end{rem}

\subsubsection{Conditions from Lemma \ref{Lem:strategy-1}}

Let us check that \eqref{Eq:strategy-1i} is always satisfied, and a similar argument for \eqref{Eq:strategy-1ii} holds. 

Assume that $i=k=p$, and that $j,l,q$ are distinct from $i$ and not all the same. 
Since $n=3$, two of the indices $j,l,q$ are the same. 
If $j=l$ is distinct from $q$, \eqref{Eq:strategy-1i} becomes 
$$-(\beta_{jq}^{(i)})^2=\beta_{jj}^{(i)}\beta_{jq}^{(i)}+\beta_{jq}^{(i)}\beta_{qj}^{(i)}+\beta_{qj}^{(i)}\beta_{jj}^{(i)}=-\frac14\,,$$
by skewsymmetry of $\beta^{(i)}$. As  $\beta_{jq}^{(i)}\in \{\pm 1/2\}$ for all distinct indices $i,j,q$, this equality is always satisfied. It is easy to deal with all the other cases in the same way.

\subsubsection{Conditions from Lemma \ref{Lem:strategy-2.1.b}}

Since $i,p,q$ are pairwise distinct and we have at most $n=3$ different indices, \eqref{Eq:strategy-2.1.bi} can only occur with $l=q$ and $i=k$, in which case the condition becomes  
$$\text{either }\ \nu_{iq}^{(p)}=0\,, \quad \text{ or }\  \alpha_{ip}^{(q)}+\beta_{qp}^{(i)}=0\,.$$
It is easy to check that this condition is satisfied for the coefficients given in Table \ref{Tab1} with arbitrary distinct $i,p,q$.

\subsubsection{Conditions from Lemma \ref{Lem:strategy-2.2}}
 
Since we have at most three distinct indices,  and $i\neq j,k,p$ with $j\neq k,p$, we only have to consider the cases where $k=p$.
The condition \eqref{Eq:strategy-2.2ii} becomes 
$$-(\mu_{kj}^{(i)})^{2}=-\frac14\,,$$
which is trivially satisfied since $\mu_{kj}^{(i)}\in \{\pm 1/2\}$ whenever $i,j,k$ are distinct. 
The other condition \eqref{Eq:strategy-2.2i} reads 
$$\nu_{kj}^{(i)}\Big[ \mu_{kj}^{(i)}+ \beta_{ij}^{(k)} \Big]=0\,.$$
Since $\nu_{- -}^{(3)}=0$, we only need to check that $\mu_{kj}^{(i)}=-\beta_{ij}^{(k)}$ whenever $i\neq 3$, and this is easy to see from  Table \ref{Tab1}. 

\subsubsection{Conditions from Lemma \ref{Lem:strategy-2.3}} The discussion is similar to the case of Lemma \ref{Lem:strategy-2.2}.

\subsubsection{Conditions from the Lemmae in \texorpdfstring{\ref{sss:strat3-1}}{A.4.1}} There is nothing to check in all these cases because the conditions rely on the existence of four different indices, which is not possible when $n=3$.

\subsubsection{Conditions from Lemma \ref{Lem:strategy-3.2.a}} We only check the case (iii), and leave the other cases to the reader. For $j<k<i$ to happen with $n=3$, we need $(i,j,k)=(3,1,2)$. The first and second equations can only happen with $a=k=2$, when 
\begin{align*}
 &\kappa_a^{(i,j)} \left(\mu_{jk}^{(i)} + \mu_{ak}^{(i)}+\frac12 \delta_{ak}\right)=1 \left(-\frac12 + 0 + \frac12\right)=0\,, \\
&\kappa_a^{(i,j)} \left(\beta_{jk}^{(i)} + \beta_{ka}^{(i)}+\frac12 \delta_{ak}\right) = 1 \left(-\frac12 + 0 + \frac12\right)=0\,,
\end{align*}
hence they are both satisfied. The third equation can not occur since it requires $a,j,k,i$ to be distinct while we have at most $3$ indices. 
The fourth equation appears with $a=k=2$ only, and it vanishes identically as $\nu_{kk}^{(i)}=0$.  
The remaining three equations are satisfied as 
\begin{align*}
 &\frac12 (\mu_{jk}^{(i)} + \beta_{jk}^{(i)})+ \mu_{jk}^{(i)} \beta_{jk}^{(i)}=\frac12 \left(-\frac12-\frac12\right) + \frac14 =-\frac14\,, \\
 &\nu_{jk}^{(i)} \left(\mu_{ik}^{(j)}-\frac12\right)=0 \left(\frac12 - \frac12 \right)=0\,, \\
 &\mu_{jk}^{(i)} + \beta_{jk}^{(i)} + \nu_{jk}^{(i)}\nu_{ik}^{(j)}+\kappa_k^{(i,j)}=-\frac12 -\frac12 + 0 + 1 =0\,. 
\end{align*}

\subsubsection{Conditions from Lemma \ref{Lem:strategy-3.2.b}}

We only check the case (iii), which occurs for $(i,j,k)=(1,3,2)$. The different identities are 
 \begin{align*}
&\frac12 (\mu_{jk}^{(i)} + \beta_{jk}^{(i)})- \mu_{jk}^{(i)} \beta_{jk}^{(i)}=\frac12\left(\frac12 -\frac12\right)- \frac{-1}{4}=\frac14\,, \\
  &\nu_{jk}^{(i)} \left(\beta_{jk}^{(i)}-\frac12\right)+\kappa_k^{(j,i)}\nu_{ij}^{(k)}=1 \left(-\frac12 - \frac12 \right)+1 =0\,,\\
  &\kappa_k^{(j,i)} (\alpha_{ij}^{(k)}+\mu_{ij}^{(k)})=1 \left(\frac12+\frac{-1}{2} \right)=0\,. 
 \end{align*}

\subsubsection{Conditions from Lemmae \ref{Lem:strategy-3.3.a} and \ref{Lem:strategy-3.3.b}} This is similar to checking the conditions from Lemmae  \ref{Lem:strategy-3.2.b} and \ref{Lem:strategy-3.2.a}, respectively.

\subsubsection{Conditions from Lemma \ref{Lem:strategy-3.4}} 

The only conditions to check are when there exists a nonzero symbol $\kappa_{-}^{(- -)}$. In the case at hand, this only occurs for $\kappa_{2}^{(3,1)}=+1$. Thus, we only need to check the identities in (i) for $(i,j,b)=(1,3,2)$ and in (ii) for $(i,j,b)=(3,1,2)$, which is straightforward.

\subsubsection{Conditions from Lemma \ref{Lem:strategy-4.2}} The case i) only occurs with $(i,k,p)=(1,2,3)$, while the case (ii) only occurs with $(i,k,p)=(1,3,2)$. The different conditions are easily verified for these particular indices (the last two can be omitted each time, because they require to have four distinct indices).



\appendix

\section{Remaining expressions for the double bracket on \texorpdfstring{$\Bc(\Delta)$}{B(Delta)}} 
\label{sec:the-complete-list}
\allowdisplaybreaks

As a consequence of Lemma \ref{lem:localisation-Bgamma} we get that the arrows $\{v_{ij}, v_{ji}\}_{1\leq i<j\leq 3}$ are generators of the Boalch algebra $\Bc(\Delta)$ in the case of the monochromatic triangle. Therefore, it is sufficient to define the double quasi-Poisson bracket $\lr{-,-}$ on such arrows, as done in \eqref{double-bracket-111}.
Nevertheless,  for completeness, we provide the rest of the expressions of the double quasi-Poisson bracket from Theorem \ref{thm:triangle} involving at least one arrow $w_{ij}$. In their determination we used the cyclic antisymmetry \eqref{Eq:cycanti}, the Leibniz rules \eqref{Eq:outder} and \eqref{Eq:inder}, and the identities \eqref{eq:Boalch-eq-111}.

\subsection{The brackets of the form \texorpdfstring{$\lr{w_{ij},v_{kl}}$}{{{w,v}}} } \hfill 
\label{sec:A1}
\allowdisplaybreaks

Using \eqref{Eq:inder} and \eqref{eq:Boalch-eq-111} we can give the complete list:
{\tiny
\begin{equation*}
\begin{split}
\lr{w_{12},v_{12}}&= -  \frac{1}{2}\Big(v_{12}\otimes w_{12}+w_{12}\otimes v_{12}\Big);			  
\\
\lr{w_{12},v_{21}}&=  \;\;\, \frac{1}{2}\Big(v_{21}w_{12}\otimes e_1+e_2\otimes w_{12}v_{21}\Big)-e_2\otimes e_1;     
\\
\lr{w_{12},v_{13}}&= -\frac{1}{2}w_{12}\otimes v_{13};
\\
\lr{w_{12},v_{31}}&=  \;\;\,  \frac{1}{2}v_{31}w_{12}\otimes e_1;
\\
\lr{w_{12},v_{23}}&=  \;\;\,  \frac{1}{2}e_3\otimes w_{12}v_{23};
\\
\lr{w_{12},v_{32}}&=-\frac{1}{2}v_{32}\otimes w_{12};	 
\\
\lr{w_{21},v_{12}}&= \;\;\, e_1\otimes e_2-\frac{1}{2}\Big(v_{12}w_{21}\otimes e_2+e_1\otimes w_{21}v_{12}\Big);
\\
\lr{w_{21},v_{21}}&=  \;\;\, \frac{1}{2}\Big(w_{21}\otimes v_{21}+v_{21}\otimes w_{21}\Big);
\\
\lr{w_{21},v_{13}}&= -\frac{1}{2}e_1\otimes w_{21}v_{13};
\\
\lr{w_{21},v_{31}}&=  \;\;\, \frac{1}{2}v_{31}\otimes w_{21};
\\
\lr{w_{21},v_{23}}&=  \;\;\,  \frac{1}{2}w_{21}\otimes v_{23};
\\
\lr{w_{21},v_{32}}&= 	-\frac{1}{2}v_{32}w_{21}\otimes e_2;	
\\
\lr{w_{13},v_{12}}&=- \frac{1}{2}w_{13}\otimes v_{12};
\\
\lr{w_{13},v_{21}}&= \;\;\, \frac{1}{2}v_{21}w_{13}\otimes e_1-w_{23}\otimes e_1;
\\
\lr{w_{13},v_{13}}&= -\frac{1}{2}\Big(w_{13}\otimes v_{13}+v_{13}\otimes w_{13}\Big);
\\
\lr{w_{13},v_{31}}&= \;\;\, \frac{1}{2}\Big(v_{31}w_{13}\otimes e_1+e_3\otimes w_{13}v_{31} \Big)-e_3\otimes e_1;
\\
\lr{w_{13},v_{23}}&= -\frac{1}{2}v_{23}\otimes w_{13};
\\
\lr{w_{13},v_{32}}&=  \;\;\, \frac{1}{2}e_3\otimes w_{13}v_{32};
\end{split}
\quad
\begin{split}
 \lr{w_{31},v_{12}}&=-\frac{1}{2}e_1\otimes w_{31}v_{12}+e_1\otimes w_{32};
\\
 \lr{w_{31},v_{21}}&=\;\;\, \frac{1}{2}v_{21}\otimes w_{31};
\\
\lr{w_{31},v_{13}}&=\;\;\, e_1\otimes e_3-\frac{1}{2}\Big(v_{13}w_{31}\otimes e_3+e_1\otimes w_{31}v_{13}\Big);
\\
\lr{w_{31},v_{31}}&= \;\;\, \frac{1}{2}\Big(w_{31}\otimes v_{31}+v_{31}\otimes w_{31}\Big);
\\
\lr{w_{31},v_{23}}&=-\frac{1}{2}v_{23}w_{31}\otimes e_3;
\\
\lr{w_{31},v_{32}}&= \;\;\, \frac{1}{2}w_{31}\otimes v_{32};
\\
\lr{w_{23},v_{12}}&=\;\;\, \frac{1}{2}v_{12}w_{23}\otimes e_2;
\\
\lr{w_{23},v_{21}}&=-\frac{1}{2}w_{23}\otimes v_{21};
\\
\lr{w_{23},v_{13}}&=-\frac{1}{2}v_{13}\otimes w_{23};
\\
\lr{w_{23},v_{31}}&=\;\;\, \frac{1}{2}e_3\otimes w_{23}v_{31}-e_3\otimes v_{21};
\\
\lr{w_{23},v_{23}}&=-\frac{1}{2}\Big(v_{23}\otimes w_{23}+w_{23}\otimes v_{23}\Big);
\\
\lr{w_{23},v_{32}}&=  \;\;\, \frac{1}{2}\Big(e_3\otimes w_{23}v_{32}+v_{32}w_{23}\otimes e_2\Big) -e_3\otimes e_2;
\\
\lr{w_{32},v_{12}}&=\;\;\, \frac{1}{2}v_{12}\otimes w_{32};
\\
\lr{w_{32},v_{21}}&=-\frac{1}{2}e_2\otimes w_{32}v_{21};
\\
\lr{w_{32},v_{13}}&=-\frac{1}{2}v_{13}w_{32}\otimes e_3+v_{12}\otimes e_3;
\\
\lr{w_{32},v_{31}}&=\;\;\, \frac{1}{2}w_{32}\otimes v_{31};
\\
\lr{w_{32},v_{23}}&=\;\;\, e_2\otimes e_3-\frac{1}{2}\Big(v_{23}w_{32}\otimes e_3+e_2\otimes w_{32}v_{23}\Big);
\\
\lr{w_{32},v_{32}}&=  \;\;\, \frac{1}{2}\Big(w_{32}\otimes v_{32}+v_{32}\otimes w_{32}\Big).
\end{split}
\end{equation*}
}

\subsection{The brackets of the form \texorpdfstring{$\lr{v_{kl}, w_{ij}}$}{{{v,w}}} } \hfill 
\label{sec:A2}

\allowdisplaybreaks
The brackets $\lr{v_{kl}, w_{ij}}$ are obtained from the brackets $\lr{w_{ij}, v_{kl}}$ appearing in \ref{sec:A2} by applying the antisymmetry \eqref{Eq:cycanti}. 
For instance, 
$$\lr{v_{31},w_{32}}=-\tau_{(12)} \lr{w_{32},v_{31}}=-\frac{1}{2}v_{31}\otimes w_{32}\,.$$

\subsection{The brackets of the form \texorpdfstring{$\lr{w_{ij}, w_{kl}}$}{{{w,w}}} }  \hfill
\label{sec:A3}

\allowdisplaybreaks
For $i,j\in\{1,2,3\}$ and $i\neq j$, we have $\lr{w_{ij},w_{ij}}= 0$. In addition,

{\tiny
\begin{equation*}
\begin{split}
\lr{w_{12},w_{21}}&= \;\;\,  \frac{1}{2}\Big(e_2\otimes w_{12}w_{21}+ w_{21}w_{12}\otimes e_1\Big)
\\
&\qquad \quad  - \gamma^{-1}_2\otimes\gamma_1;
\\
\lr{w_{12},w_{13}}&=  -\frac{1}{2} w_{12}\otimes w_{13};
\\
\lr{w_{12},w_{31}}&= \;\;\,  \frac{1}{2}w_{31}w_{12}\otimes e_1;
\\
\lr{w_{12},w_{23}}&= \;\;\,  \frac{1}{2}    e_2\otimes w_{12}w_{23}-e_2\otimes w_{13};
\\
\lr{w_{12},w_{32}}&=	 -\frac{1}{2}w_{32}\otimes w_{12};
\\
\lr{w_{21},w_{13}}&= -\frac{1}{2}e_1\otimes w_{21}w_{13};
\\
\lr{w_{21},w_{31}}&= \;\;\, \frac{1}{2}w_{31}\otimes w_{21};
\end{split}
\quad
\begin{split}
\lr{w_{21},w_{23}}&= \;\;\, \frac{1}{2}    w_{21}\otimes w_{23};
\\
\lr{w_{21},w_{32}}&= -\frac{1}{2}    w_{32}w_{21}\otimes e_2+w_{31}\otimes e_2;		
\\
\lr{w_{13},w_{31}}&=\;\;\, \frac{1}{2}\Big(    e_3\otimes w_{13}w_{31}+w_{31}w_{13}\otimes e_1\Big) +\gamma^{-1}_3\otimes w_{13}\gamma_3w_{31}
\\
&\qquad  \quad -\gamma^{-1}_3\otimes e_1;
\\
\lr{w_{13},w_{23}}&= -\frac{1}{2}    w_{23}\otimes w_{13};
\\
\lr{w_{13},w_{32}}&=\;\;\, \frac{1}{2}e_3\otimes w_{13}w_{32}-\gamma^{-1}_3\otimes w_{12}\gamma_2;
\\
\lr{w_{31},w_{23}}&=-\frac{1}{2}w_{23}w_{31}\otimes e_3+\gamma_2w_{21}\otimes \gamma^{-1}_3;
\\
\lr{w_{31},w_{32}}&=\;\;\, \frac{1}{2}w_{31}\otimes w_{32};
\\
\lr{w_{23},w_{32}}&=\;\;\, \frac{1}{2}\Big(e_3\otimes w_{23}w_{32}+w_{32}w_{23}\otimes e_2\Big) -\gamma^{-1}_3\otimes \gamma_2.
 \end{split}
 \end{equation*}
 }%
 The reader can obtain the 15 missing expressions by applying the antisymmetry \eqref{Eq:cycanti} to these identities. For example, $\lr{w_{23},w_{21}}=-\tau_{(12)}\lr{w_{21},w_{23}}=-\frac{1}{2}w_{23}\otimes w_{21}$.



\end{document}